\documentclass[11pt,a4paper]{amsart}
\usepackage[english]{babel}
\usepackage[T1]{fontenc}
\usepackage{mathpazo,amssymb,cancel}
\usepackage{upref}
\usepackage{enumerate}
\usepackage{graphicx,xcolor}
\usepackage{cancel}
\usepackage{dsfont}       
\usepackage[colorlinks=true, pdftitle={The Dirichlet-to-Neumann
  operator for $\Delta_{1}$}, pdfauthor={Daniel Hauer
and Jos\'e M. Maz\'on} ]{hyperref}

\def\z{{\bf z}}

\title[The DtN-operator for the $1$-Laplacian]{
  The Dirichlet-to-Neumann operator\\ associated with the $1$-Laplacian\\ and evolution problems}

\author{Daniel Hauer} \address[Daniel Hauer]{School of Mathematics and
  Statistics, The University of Sydney, NSW 2006, Australia}
\email{\href{mailto:daniel.hauer@sydney.edu.au}{\nolinkurl{daniel.hauer@sydney.edu.au}}}

\author{Jos\'e M. Maz\'on} \address[Jos\'e M. Maz\'on]{Departamento
  de An\'{a}lisis Matem\'{a}tico, Universitat de Val\`encia, Valencia, Spain}
\email{\href{mailto:mazon@uv.es}{\nolinkurl{mazon@uv.es}}}
\thanks{The first author is very grateful for the kind invitation to
  the Universitat de Val\`encia and their hospitality. In the second
  half of the research project, he was supported by the Australian Research Council grant DP200101065.
The second author was partially supported  by the Spanish MCIU and FEDER, project PGC2018-094775-B-100.}



\numberwithin{equation}{section}

\newtheorem{theorem}{Theorem}[section]
\newtheorem{proposition}[theorem]{Proposition}
\newtheorem{lemma}[theorem]{Lemma}
\newtheorem{corollary}[theorem]{Corollary}

\theoremstyle{definition}
\newtheorem{definition}[theorem]{Definition}

\newtheorem{notation}[theorem]{Notation}
\newtheorem{remark}[theorem]{Remark}
\newtheorem{example}[theorem]{Example}

\newcommand\R{{\mathbb{R}}}
\newcommand\N{\mathbb{N}}

\newcommand\E{\mathcal{E}}

\newcommand\dx{\mathrm{d}x }

\newcommand\dr{\mathrm{d}r }
\newcommand\ds{\mathrm{d}s }

\newcommand\dmu{\mathrm{d}\mu}
\newcommand\dt{\mathrm{d}t }
\newcommand\td{\mathrm{d} }

\newcommand\dH{\mathrm{d}\mathcal{H}}

\DeclareMathOperator*{\divi}{div}

\DeclareMathOperator{\sign}{sign}

\DeclareMathOperator{\Rg}{Rg}
\DeclareMathOperator{\T}{\mathit{Tr}}

\newcommand\abs[1]{\lvert#1\rvert}
\newcommand\labs[1]{\left\lvert#1\right\rvert}
\newcommand\norm[1]{\lVert#1\rVert}
\newcommand\lnorm[1]{\left\lVert#1\right\rVert}

\definecolor{darkred}{rgb}{0.7,0.1,0.1}

\def\1{\raisebox{2pt}{\rm{$\chi$}}}

\begin{document}
\date{\today}
\begin{abstract}
  We present first results on the Dirichlet-to-Neumann operator
  associated with the $1$-Laplace operator in $L^1$. In particular,
  we show that this operator can be realized as a sub-differential
  operator in $L^1\times L^{\infty}$ of a homogeneous convex,
  continuous functional with effective domain $L^1$. Even though the
  Dirichlet problem associated with the $1$-Laplace operator loses the
  property that weak solutions for boundary data in $L^1$ are unique,
  we prove a type of stability/compactness result with respect to the
  boundary data in $L^1$ of this problem. We apply our results for the
  stationary Dirichlet problem to evolution problems governed by the
  Dirichlet-to-Neumann operator, which can equivalently be formulated
  as singular coupled elliptic-parabolic initial boundary-value
  problems. For initial data in $L^q$, $1\le q\le \infty$, we obtain
  well-posedness, that every mild solution is, indeed, a strong
  solution, and establish long-time stability of the semigroup
  generated by the negative Dirichlet-to-Neumann operator associated
  with the $1$-Laplace operator.
\end{abstract}
\maketitle

\tableofcontents

%
%

\section{Introduction and Main Results}

The goal of this paper is to provide a first insight on the
\emph{Dirichlet-to-Neumann} operator
\begin{displaymath}
  \displaystyle\Lambda : h_{\vert\partial\Omega}\mapsto
  \left[\dfrac{Du}{\abs{Du}}\cdot\nu\right]_{\vert\partial\Omega}\text{
  (by a \emph{weak solution} $u$ of~\eqref{eq:21} below)}
\end{displaymath}
associated with the \emph{$1$-Laplace operator}
\begin{displaymath}
  \Delta_{1}u:=\divi\left(\frac{Du}{\abs{Du}}\right)
\end{displaymath}
on bounded domains $\Omega\subseteq \R^{d}$ with a $C^1$-boundary
$\partial\Omega$, $d\ge 2$. Here, $\nu$ denotes the outward pointing unit
normal vector on $\partial\Omega$. In other words, $\Lambda$ assigns
Dirichlet data $h$ to the co-normal derivative $Du\cdot\nu/\abs{Du}$
on $\partial\Omega$ of an extension $u$ on $\Omega$ of $h$, which is a
\emph{weak solution} of the
singular \emph{Dirichlet problem} for the $1$-Laplace operator
\begin{equation}
  \label{eq:21}
  \begin{cases}
    \hspace{1.18cm}\displaystyle-\divi\left(\frac{Du}{\abs{Du}}\right)=0&
    \quad\text{in $\Omega$,}\\
     \hspace{3.45cm} u=h& \quad\text{on $\partial\Omega$,}
  \end{cases}
\end{equation}
in the following sense.

\begin{definition}
  \label{def:sols-DP-Intro}
  For given $h\in L^{1}(\partial\Omega)$, we call a function
  $u\in BV(\Omega)$ a \emph{weak solution} of Dirichlet
  problem~\eqref{eq:21} if there is a vector field
  $\z_{h}\in L^{\infty}(\Omega;\R^{d})$ generalizing $Du/\abs{Du}$
  through the three conditions
\begin{align}
  \label{z21intro} \norm{\z_h}_\infty&\le 1,\\
  \label{z22intro} -\divi(\z_h)&=0\qquad \text{in
                            $\mathcal{D}^\prime(\Omega)$, and}\\
    \label{z32intro} (\z_h,Du)&=|Du|\qquad \text{as Radon measures}
\end{align}
and the \emph{weak trace} $[\z_{h},\nu]$ on $\partial\Omega$  (see
Definition~\ref{def:weak-trace-of-z}) of the generalized co-normal
derivative $\z_{h}\cdot\nu$ satisfies
\begin{equation}
  \label{eq:98}
    [\z_{h},\nu]\in \sign(h-\T(u)) \qquad\text{$\mathcal{H}^{d-1}$-a.e. on
        $\partial\Omega$.}
\end{equation}
\end{definition}

It is well-known that for every $h\in L^{1}(\partial\Omega)$, there
exist a weak solution $u$ of Dirichlet problem~\eqref{eq:21}. But
difficulties in deriving properties of the Dirichlet-to-Neumann
operator $\Lambda$ arise, for instance, from the fact that the notion
of weak solutions $u$ of~\eqref{eq:21} merely requires that the Dirichlet boundary data
$u=h$ on $\partial\Omega$ is satisfied in the \emph{very weak} sense~\eqref{eq:98}.
Because of this, the Dirichlet problem~\eqref{eq:21} might have
infinitely many weak solutions $u$ (see
Remark~\ref{rem:uniqueness-of-DP} for more details). But, in addition, for
each weak solutions $u$ of~\eqref{eq:21}, there might be infinitely
many vector fields $\z_{h}\in L^{\infty}(\Omega;\R^{d})$
satisfying~\eqref{z21intro}-\eqref{eq:98}. In Section~\ref{sec:DP}, we
provide a brief review of the literature about the current state of
existence and (non)-uniqueness of weak solutions to Dirichlet
problem~\eqref{eq:21}. Thus the following realization of the
Dirichlet-to-Neumann operator $\Lambda$ in $L^{1}(\partial\Omega)$
defines a possibly multi-valued operator.

In the following, let $\overline{B}_{L^{\infty}(\partial\Omega)}$ denote the closed
unit ball of $L^{\infty}(\partial\Omega)$ centered at $h=0$.


\begin{definition}
  \label{def:LambdaL1-informal}
  Let $\Lambda $ be the set of all pairs
  $(h,g)\in \subseteq L^{1}(\partial\Omega)\times
  \overline{B}_{L^{\infty}(\partial\Omega)}$ with the property that
  there is a weak solution $u_{h}\in BV(\Omega)$ of Dirichlet
  problem~\eqref{eq:21} for Dirichlet data $h$ and there is a vector
  field $\z_{h}\in L^{\infty}(\Omega;\R^{d})$ satisfying
  \eqref{z21intro}-\eqref{eq:98} with $u_{h}$ and
    \begin{equation}
      \label{eq:92}
      [\z_{h},\nu]=g\qquad\text{$\mathcal{H}^{d-1}$-a.e. on $\partial\Omega$.}
    \end{equation}
   Then, we call $\Lambda$ the \emph{Dirichlet-to-Neumann operator in
      $L^{1}(\partial\Omega)$ associated with the
    $1$-Laplace operator} $\Delta_{1}$.
\end{definition}

Now, our first main result reads as follows.  Here, we write
$L^{\infty}_{\sigma}(\partial\Omega)$ for the space
$L^{\infty}(\partial\Omega)$ equipped with the
weak$\mbox{}^{\ast}$-topology
$\sigma(L^{\infty}(\partial\Omega),L^{1}(\partial\Omega))$.

\begin{theorem}\label{thm:main1}
  The Dirichlet-to-Neumann operator $\Lambda$ in
      $L^{1}(\partial\Omega)$ associated with the
  $1$-Lapla\-ce operator $\Delta_{1}$ admits the following properties.
  \begin{enumerate}[(1.)]
    \item \label{thm:main1-claim1} $\Lambda $ is $m$-completely accretive in
     $L^{1}(\partial\Omega)$ and has effective domain
    $D(\Lambda)=L^{1}(\partial\Omega)$;
   \item \label{thm:main1-claim2} $\Lambda$ is homogeneous of order
     zero;
   \item \label{thm:main1-claim3} $\Lambda$ is closed in
    $L^{1}(\partial\Omega)\times L^{\infty}_{\sigma}(\partial\Omega)$;
   \item \label{thm:main1-claim4} $\Lambda$ can be characterized by
     \begin{equation}
       \label{eq:104}
      \Lambda=\partial_{L^{1}\times L^{\infty}(\partial\Omega)}\varphi
    \end{equation}
   for the sub-differential operator $\partial_{L^{1}\times L^{\infty}(\partial\Omega)}\varphi$ in
    $L^{1}\times L^{\infty}(\partial\Omega)$ of the convex, even,
    homogeneous of order one, and continuous
    functional $\varphi : L^{1}(\partial\Omega)\to [0,\infty)$ defined
    by
    \begin{equation}
      \label{eq:4}
      \varphi(h)=\int_{\partial\Omega}[\tilde{\z}_{h},\nu]\,h\,\dH^{d-1}
    \end{equation}
    for every $h\in L^{1}(\partial\Omega)$, where
    $\tilde{\z}_{h}\in L^{\infty}(\Omega;\R^{d})$ is a vector field
    satisfying \eqref{z21intro}-\eqref{z32intro} for some weak
    solution $u_{h}\in BV(\Omega)$ of Dirichlet problem~\eqref{eq:21}
    with boundary data $u_{h}=h$.
  \end{enumerate}
\end{theorem}

In Section~\ref{sec:proofs}, we give the details of the proof of this theorem.
Further, we refer to Definition~\ref{def:completely-accretive-operators} and
Definition~\ref{def:homogeneous-operators} for the two notions of
\emph{$m$-completely accretive} operators and \emph{homogeneous
  operators of order $\alpha\in \R$}. Both
statements~\eqref{thm:main1-claim1} and \eqref{thm:main1-claim2} in
Theorem~\ref{thm:main1} follow from a careful study of the Dirichlet
problem~\eqref{eq:21} (see
Proposition~\ref{prop:Lambda-completely-accretive} and
Proposition~\ref{prop:Lambda-homogeneity}). We recall the notion of an
\emph{even} functional in Definition~\ref{def:even} and prove this
property of $\varphi$ together with the continuity, homogeneity and
convexity properties in Proposition~\ref{prop:DtN-functional}. To
prove the characterization~\eqref{eq:104}, we first show in
Proposition~\ref{prop:7} that the closure
$\overline{\Lambda}^{\mbox{}_{L^{1}\times L^{\infty}_{\sigma}}}$ of
$\Lambda$ in $L^{1}(\partial\Omega)\times
L^{\infty}_{\sigma}(\partial\Omega)$ is contained in the
sub-differential operator $\partial_{L^{1}\times
  L^{\infty}(\partial\Omega)}\varphi$. Since $\partial_{L^{1}\times
  L^{\infty}(\partial\Omega)}\varphi$ is completely accretive in $L^{1}(\Omega)$, once we have shown that
$\Lambda$ is $m$-completely accretive in $L^{1}(\Omega)$, the
characterization~\eqref{eq:104} follows from a classical result by
B\'enilan and Crandall (see
Proposition~\ref{prop:characterization-of-A-L0}). The
property that the Dirichlet-to-Neumann operator $\Lambda$ is closed in
$L^{1}(\partial\Omega)\times L^{\infty}_{\sigma}(\partial\Omega)$
(statement~\eqref{thm:main1-claim3}) is proved in
Proposition~\ref{prop:charact-closure-Lambda}, and provides the
following surprising stability/compactness result related to the
Dirichlet problem~\eqref{eq:21}.

\begin{corollary}[{stability/compactness}]\label{cor:stability-of-DP}
  For every sequence $(h_{n})_{n\ge 1}$ in
  $L^{1}(\partial\Omega)$ converging to some $h$ in
  $L^{1}(\partial\Omega)$, there is a weak solution $u_{h}$ of
  Dirichlet problem~\eqref{eq:21} satisfying boundary data $u_{h}=h$
  and a sub-sequence $(h_{k_{n}})_{n\ge 1}$ of $(h_{n})_{n\ge 1}$ such
  that the generalized co-normal derivative $[\z_{h_{k_{n}}},\nu]$
  corresponding to $h_{k_{n}}$ converges weakly$\mbox{}^{\ast}$ to
  $[\z_{h},\nu]$ in $L^{\infty}(\partial\Omega)$ and
  \begin{equation}
    \label{eq:88}
    \lim_{n\to\infty}\varphi(h_{n})=\varphi(h).
  \end{equation}
\end{corollary}

Of course, the limit~\eqref{eq:88} follows from the continuity
property of $\varphi$, but the surprising fact here is that the two
divergence free vector field $\tilde{\z}_{h}$ in $\varphi(h)$ and
$\z_{h}$ in the limit $[\z_{h},\nu]$ don't have to be the same. In
fact, we show in Theorem~\ref{thm:9} that any two divergence free
vector fields $\z_{h}$ and $\hat{\z}_{h}$ are interchangeable among
the set of weak solutions $u_{h}$ and $\hat{u}_{h}$ of Dirichlet
problem~\eqref{eq:21}. But according to Theorem~\ref{thm:8}, the value
of $\varphi(h)$ is invariant among all divergence free vector fields
$\z_{h}$, for which there is a weak solution $u_{h}$ of Dirichlet
problem~\eqref{eq:21}. Hence, $\varphi$ given by~\eqref{eq:4} is a
well-defined mapping on $L^{1}(\partial\Omega)$. The fact that
$\varphi$ is continuous on $L^{1}(\partial\Omega)$ can easily be
deduces from its convexity property and that $\varphi$ is upper bounded
on an any bounded subset of its effective
domain $D(\varphi)=L^{1}(\partial\Omega)$ (see
Proposition~\ref{prop:DtN-functional} for more details).\medskip

Further, we establish well-posedness and comparison principles in the sense of \emph{mild}
solutions (see Definition~\ref{def:mild}), sufficient conditions
implying improved regularity properties of mild solutions, and the long-time stability (in
the case $F\equiv 0$ and $g\equiv 0$) of the inhomogeneous Cauchy
problem (in $L^{1}(\partial\Omega)$)
\begin{equation}
  \label{eq:24}
  \begin{cases}
    \dfrac{\td h}{\dt}(t)+\Lambda h(t)+F(h(t))\ni g(t)&
    \quad\text{for $t\in(0,T)$,}\\
    \hspace{2.95cm} h(0)=h_{0}.& \quad\text{on $\partial\Omega$,}
  \end{cases}
\end{equation}
for every given $g\in L^{1}(0,T;L^{1}(\partial\Omega))$ and
$h_{0}\in L^{1}(\partial\Omega)$. In the differential
inclusion~\eqref{eq:24}, the lower order term
$F : L^{1}(\partial\Omega)\to L^{1}(\partial\Omega)$ denotes the
Nemytskii operator
\begin{equation}
  \label{eq:28}
  F(h)(x):=f(x,h(x))\quad\text{for a.e. $x\in \partial\Omega$, $h\in L^{1}(\partial\Omega)$,}
\end{equation}
of a \emph{Lipschitz-Carath\'eodory function} $f :
\partial\Omega\times \R\to \R$ satisfying $f(x,0)=0$, ($x\in \partial\Omega$); that is,
\begin{equation}\label{eq:27}
  \begin{cases}
    &\text{there is an $\omega>0$ such that
      $\abs{f(x,h)-f(x,\hat{h})}\le \omega\,\abs{h-\hat{h}}$}\\
    &\text{for all $h$, $\hat{h}\in \R$, uniformly for a.e. $x\in
      \partial\Omega$, and}\\
    &\text{for all $h\in \R$, $x\mapsto f(x,h)$ is measurable on
      $\partial\Omega$.}
  \end{cases}
\end{equation}

It is worth noting that well-posedness, regularity and stability analysis of Cauchy problem~\eqref{eq:24}
are equivalent topics of the following singular \emph{elliptic-parabolic boundary value problem}
\begin{equation}
  \label{eq:119}
  \begin{cases}
    \displaystyle\hspace{1.43cm}-\divi\left(\frac{Du_{h}}{\abs{Du_{h}}}\right)=0& \quad\text{in
      $\Omega\times (0,T)$,}\\
     \hspace{3,7cm} u_{h}=h& \quad\text{on $\partial\Omega\times (0,T)$,}\\
    \displaystyle\,\partial_{t}h+\dfrac{Du_{h}}{\abs{Du_{h}}}\cdot\nu+f(\cdot,h)\ni g&
    \quad\text{on $\partial\Omega\times (0,T)$,}\\
    \hspace{3.9cm} h=h_{0}& \quad\text{on $\partial\Omega\times \{t=0\}$.}
  \end{cases}
\end{equation}
Recently, existence and uniqueness to the elliptic-parabolic
  boundary value problem
  \begin{displaymath}
  \begin{cases}
    \displaystyle\hspace{0.65cm}\lambda h-\divi\left(\frac{Du_{h}}{\abs{Du_{h}}}\right)=0& \quad\text{in
      $\Omega\times (0,T)$,}\\[7pt]
     \hspace{3,4cm} u_{h}=h& \quad\text{on $\partial\Omega\times (0,T)$,}\\
    \hspace{1.3cm}\displaystyle\,\partial_{t}h+\dfrac{Du_{h}}{\abs{Du_{h}}}\cdot\nu\ni g&
    \quad\text{on $\partial\Omega\times (0,T)$,}\\
    \hspace{3.63cm} h=h_{0}& \quad\text{on $\partial\Omega\times \{t=0\}$}
  \end{cases}
\end{displaymath}
for parameter $\lambda>0$ and with initial data $h_{0}\in L^{2}(\partial\Omega)$ was obtained in
\cite{MR3798643}. We emphasize that for $\lambda>0$, the associated
Dirichlet problem
\begin{displaymath}
  \begin{cases}
    \displaystyle\hspace{0.78cm}\lambda u_{h}-\divi\left(\frac{Du_{h}}{\abs{Du_{h}}}\right)=0&
    \quad\text{in $\Omega$,}\\
     \hspace{3.75cm} u_{h}=h& \quad\text{on $\partial\Omega$.}
  \end{cases}
\end{displaymath}
is uniquely solvable, but which is not true for the case $\lambda=0$ (that is,
Dirichlet problem~\eqref{eq:21}). This makes this singular elliptic-parabolic boundary value problem~\eqref{eq:119}
more appealing, but also complements the
research in~\cite{MR3798643}.\medskip

To study stronger regularity properties of mild solutions to
Cauchy problem~\eqref{eq:24}, we introduce the following operators.

\begin{notation}\label{not:Lambda-Lq}
  For every $1\le q\le \infty$, we write $\Lambda_{\vert L^{q}}$ for
  the restriction of $\Lambda$ on
  $L^{q}(\partial\Omega)\times
  \overline{B}_{L^{\infty}(\partial\Omega)}$. In other words,
  \begin{displaymath}
    \Lambda_{\vert L^{q}}=\Lambda\cap (L^{q}(\partial\Omega)\times
  \overline{B}_{L^{\infty}(\partial\Omega)}),
 \end{displaymath}
 and call the operator $\Lambda_{\vert L^{q}}$ the
 \emph{Dirichlet-to-Neumann operator on $L^{q}(\partial\Omega)$.}
\end{notation}

Thanks to the continuous embedding from $L^{q}(\partial\Omega)$ into
$L^{1}(\partial\Omega)$, the first three statements in the next
corollary follow immediately from \eqref{thm:main1-claim1} of Theorem~\ref{thm:main1} and
Corollary~\ref{cor:stability-of-DP}, and statement~\eqref{cor:Accretivity-in-Lq-claim4}
with the characterization~\eqref{eq:118} from Proposition~\ref{prop:sub-diff-L2}.

\begin{corollary}
  \label{cor:Accretivity-in-Lq}
  Let $1\le q\le \infty$. Then the following statements on the
  Dirichlet-to-Neumann operator $\Lambda_{\vert L^{q}}$ in $L^{q}(\partial\Omega)$ hold.
  \begin{enumerate}
  \item \label{cor:Accretivity-in-Lq-claim1} $\Lambda_{\vert L^{q}}$ is $m$-completely accretive in
    $L^{q}(\partial\Omega)$ with effective domain
    $D(\Lambda_{\vert L^{q}}))=L^{q}(\partial\Omega)$;
   \item $\Lambda_{\vert L^{q}}$ is homogeneous of
     order zero;
   \item $\Lambda_{\vert L^{q}}$ is closed in
     $L^{q}(\partial\Omega)\times
     L^{\infty}_{\sigma}(\partial\Omega)$;
   \item \label{cor:Accretivity-in-Lq-claim4} $\Lambda_{\vert L^{2}}$ can be characterized by
     \begin{equation}
       \label{eq:118}
       \Lambda_{\vert L^{2}}=\partial_{L^{2}}\varphi_{\vert L^{2}}
     \end{equation}
     where $\partial_{L^{2}}\varphi_{\vert L^{2}}$ denotes the
     sub-differential operator on $L^{2}(\partial\Omega)$ of the
     restriction $\varphi_{\vert L^{2}}$ on $L^{2}(\partial\Omega)$ of the functional
     $\varphi$ given by~\eqref{eq:4}.
  \end{enumerate}
\end{corollary}

The property that the operator $\Lambda_{\vert L^{q}}$ is
$m$-accretive in $L^{q}(\partial\Omega)$ yields the well-posedness of
Cauchy problem~\eqref{eq:24} for initial values $u_{0}$ in $L^{q}(\partial\Omega)$
and forcing term $g\in L^{1}(0,T; L^{q}(\partial\Omega))$ in the sense
of \emph{mild solutions} in $L^{q}(\partial\Omega)$.

\begin{corollary}[{Existence \& Uniqueness in $L^{q}(\partial\Omega)$}]
  \label{coro:well-posed}
  Let $1\le q\le \infty$ and suppose $F$ is given by~\eqref{eq:28} with $f$
  satisfying~\eqref{eq:27}. Then, for every
  $u_{0}\in L^{q}(\partial\Omega)$ and
  $g\in L^{1}(0,T;L^{q}(\partial\Omega))$, there is
  a unique mild solution of Cauchy problem~\eqref{eq:24} in
  $L^{q}(\partial\Omega)$.
\end{corollary}

Due to the fact that $\Lambda_{\vert L^{q}}$ is
completely accretive and by~\cite{Benilanbook} (see
also~\cite{CoulHau2017}), the following
comparison principle is available. Here, we write $[u]^{\nu}$ with
$\nu\in \{+,1\}$ for either denoting the \emph{positive part}
$[u]^{+}=\max\{0,u\}$ of $u$ or $u:=[u]^{1}$ itself.

\begin{corollary}[{Comparison principle \& Well-posedness}]
  \label{coro:comparision}
  Let $1\le q\le \infty$ and suppose $F$ is given by~\eqref{eq:28} with $f$
  satisfying~\eqref{eq:27}. Then, for every $h_{0}$ and $\hat{h}_{0}\in L^{q}(\partial\Omega)$,
  $g$, $\hat{g}\in L^{1}(0,T;L^{q}(\partial\Omega))$, and
  corresponding two mild solutions $h$ and $v$
  of Cauchy problem~\eqref{eq:24}, one has that
  \begin{displaymath}
       \norm{[h(t)-\hat{h}(t)]^{\nu}}_{q}\le e^{\omega
      t}\norm{[h(s)-\hat{h}(s)]^{\nu}}_{q} +\int_{s}^{t}e^{\omega (t-s)}\norm{[g(r)-\hat{g}(r)]^{\nu}}_{q}\,\dr
  \end{displaymath}
 for every $0\le s<t\le T$, and $\nu\in \{+,1\}$.
\end{corollary}


Our next theorem is concerned with the \emph{regularizing effect} that
a mild solution of Cauchy problem~\eqref{eq:24} is, indeed,  a \emph{strong solution}
of~\eqref{eq:24} (see Definition~\ref{def:strong}). The regularizing
effect described in the first two statements is due to the fact that the Dirichlet-to-Neumann operator $\Lambda_{\vert
  L^{2}}$ in $L^{2}(\partial\Omega)$
can be realized as a sub-differential operator (see~\eqref{eq:104}
or~\eqref{eq:118}) and hence, follows from a classic result due to
Brezis~\cite{MR0348562} (see, also~\cite{MR4041276}). The regularizing
effect stated in~\eqref{thm:main2-claim3} results from the property
that the Dirichlet-to-Neumann operator $\Lambda_{\vert L^{q}}$ is
homogeneous of order zero and so, follows from an application
of~\cite{MR4200826} (see also~\cite{MR648452}
and~\cite{MR4031770}). We give the details of the proof of this
theorem in Section~\ref{subsec:DtN-in-L1-continued}.


\begin{theorem}[{Regularizing effect}]\label{thm:main2}
  Let $F$ be given by~\eqref{eq:28} with $f$
  satisfying~\eqref{eq:27}, and $\E :
  L^{1}(\partial\Omega)\to \R$ denote the functional given by
  \begin{displaymath}
    \E(h):=\varphi(h)+\int_{0}^{h}
    f(\cdot,r)\,\dr\qquad
    \text{for every $h\in L^{1}(\partial\Omega)$,}
  \end{displaymath}
  where $\varphi$ is the functional defined by~\eqref{eq:4}.
  Then the following statements hold.
  \begin{enumerate}
  \item \label{thm:main2-claim1} (\textrm{Max. $L^2$-regularity}) For every
    $h_{0}\in L^{1}(\partial\Omega)$ and $g\in L^{2}(0,T;L^{2}(\partial\Omega))$, the mild solution $h$ of Cauchy
    problem~\eqref{eq:24} in $L^{1}(\partial\Omega)$ is a strong
    solution in $L^{1}(\partial\Omega)$ with time-derivative
      \begin{displaymath}
        \dfrac{\td h}{\dt}\in L^{2}(0,T;L^{2}(\partial\Omega))
      \end{displaymath}
      and global estimate
      \begin{equation}
        \label{eq:1}
        \tfrac{1}{2}\int_{0}^{t}\lnorm{\dfrac{\td
            h}{\ds}(s)}_{2}^{2}\,\ds+\E(h(t))
        \le\E(h_{0})+\tfrac{1}{2}\int_{0}^{t}\norm{g(s)}_{2}^{2}\ds
      \end{equation}
      for every $0\le t\le T$.
    \item \label{thm:main2-claim2} For every $h_{0}\in L^{2}(\partial\Omega)$ and
      $g\in L^{2}(0,T;L^{2}(\partial\Omega))$, the mild solution $u$
      of Cauchy problem~\eqref{eq:24} in $L^{2}(\partial\Omega)$ is a
      strong solution in $L^{2}(\partial\Omega)$ with
      \begin{displaymath}
        h\text{ and }\dfrac{\td h}{\dt}\in L^{2}(0,T;L^{2}(\partial\Omega)).
      \end{displaymath}
    \item \label{thm:main2-claim3} (\textrm{$L^1$ Aronson-B\'enilan type estimates}) Let
      $1\le q\le \infty$. Then for every
      $h_{0}\in L^q(\partial\Omega)$ and
      $g\in W^{1,1}(0,T;L^{q}(\partial\Omega))$, the mild solution $h$
      of Cauchy problem~\eqref{eq:24} in $L^{q}(\partial\Omega)$ is a
      strong solution in $L^{q}(\partial\Omega)$ satisfying
    \begin{equation}
      \label{eq:123}
      \lnorm{\frac{\td h}{\dt}_{\!\! +}\!\!(t)}_{q}
      \le\frac{1}{t}\!\! \left[a_{\omega}(t)+ \omega\!\!\int_{0}^{t}a_{\omega}(s)
        e^{\omega (t-s)}\ds\right]\quad\text{for a.e. $t\in (0,T)$,}
    \end{equation}
    where
    \begin{align*}
      a_{\omega}(t)&:= \int_{0}^{t}\norm{g'(s)}_{q}\,s\,\ds+ \bigg[\left(1+e^{\omega
                     t} \right)\,\norm{h_{0}}_{q}\bigg.\\
                   &\hspace{1cm}\left. +\int_{0}^{t}\norm{g(s)}_{q}\,\ds
                     +\omega\,\int_{0}^{t} \int_{0}^{s}e^{-\omega r} \norm{g(r)}_{q}\dr\,\ds\right].
    \end{align*}
  \end{enumerate}
\end{theorem}

According to Corollary~\ref{cor:Accretivity-in-Lq}, for every
$1\le q\le \infty$, the operator $-(\Lambda_{\vert L^{q}}+F)$
generates a strongly continuous semigroup
$\{e^{-t (\Lambda_{\vert L^{q}}+F)}\}_{t\ge 0}$ of quasi-contractions
on $L^{q}(\partial\Omega)$ (see Section~\ref{sec:semigroups} for a
concise review of nonlinear semigroup theory). But since
$\partial\Omega$ is assumed to be compact, the semigroup
$\{e^{-t (\Lambda_{\vert L^{q}}+F)}\}_{t\ge 0}$ generated by
$-(\Lambda_{\vert L^{q}}+F)$ on $L^{q}(\partial\Omega)$ coincides with
the semigroup $\{e^{-t (\Lambda+F)}\}_{t\ge 0}$ generated by
$-(\Lambda+F)$ on $L^{1}(\partial\Omega)$. For this reason, it is
sufficient to consider only the semigroup
$\{e^{-t (\Lambda+F)}\}_{t\ge 0}$ on $L^{q}(\partial\Omega)$, which is
quasi-contractive on $L^{q}(\partial\Omega)$ for all
$1\le q\le \infty$. The next corollary summarizes the regularity
properties of the semigroup $\{e^{-t (\Lambda+F)}\}_{t\ge 0}$. Here,
$\Lambda^{\circ}$ denotes the \emph{minimal selection} of $\Lambda$
defined by~\eqref{eq:65} in Section~\ref{sec:semigroups}.


\begin{corollary}\label{cor:regularity-semigroup}
  Let $F$ be given by~\eqref{eq:28} with $f$ satisfying~\eqref{eq:27}
  and $1\le q\le \infty$. Then the operator $-(\Lambda+F)$ generates a
  stron\-gly continuous semigroup $\{e^{-t (\Lambda+F)}\}_{t\ge 0}$ on
  $L^{1}(\partial\Omega)$, which is $\omega$-quasi complete
  contractive on $L^{q}(\partial\Omega)$ for every $q$. Moreover,
  $\{e^{-t (\Lambda+F)}\}_{t\ge 0}$ has the following regularity
  properties:
  \begin{enumerate}
  \item \label{cor:regularity-semigroup-claim1} (\textrm{$L^1$ Aronson-B\'enilan type estimates}) For every
    $h_{0}\in L^{q}(\partial\Omega)$, the mapping
    $t\mapsto e^{-t (\Lambda+F) }h_{0}$ is differentiable in
    $L^{q}(\partial\Omega)$ at a.e. $t\in (0,\infty)$ and
    \begin{displaymath}
      \lnorm{\frac{\td }{\dt}_{\!\! +}\!\! e^{-t (\Lambda+F)}h_{0}}_{q}
    \le\frac{2+\omega\,t}{t}e^{\omega
      t}\,\norm{h_{0}}_{q}\qquad\text{for every $t>0$;}
   \end{displaymath}
 \item \label{cor:regularity-semigroup-claim2}  If $F\equiv 0$, then for every $h_{0}\in
   L^{1}(\partial\Omega)$ and $t>0$, $\frac{\td }{\dt}_{\!\! +}\!\! e^{-t
     \Lambda}h_{0}$ exists in $L^{1}(\partial\Omega)$
   and
   \begin{equation}
     \label{eq:45}
     \abs{\Lambda^{\circ}e^{-t \Lambda}h_{0}}\le
     2\dfrac{\abs{h_{0}}}{t}\qquad\text{$\mathcal{H}^{d-1}$-a.e. on $\partial\Omega$;}
   \end{equation}
    \item \label{cor:regularity-semigroup-claim3} (\textrm{Order
        preservation of the semigroup})
      For every $h_{0}$, $\hat{h}_{0}\in L^{q}(\partial\Omega)$,
      one has that $h_{0}\le \hat{h}_{0}$ yields that
      \begin{displaymath}
        e^{-t (\Lambda+F)}h_{0}\le e^{-t (\Lambda+F)}\hat{h}_{0}\qquad\text{ for all
      $t\ge 0$.}
    \end{displaymath}

    \item \label{cor:regularity-semigroup-claim4} (\textrm{Point-wise Aronson-B\'enilan type estimates}) For
      every $h_{0}\in L^{q}(\partial\Omega)$ positive, one has that
      \begin{displaymath}
        \frac{\td }{\dt}_{\!\! +}\!\! e^{-t
      (\Lambda+F)}h_{0}\le \dfrac{1}{t}e^{-t
      (\Lambda+F)}h_{0}+g_{0}(t)\quad\text{for a.e. $t>0$,}
      \end{displaymath}
      where $g_{0} : (0,\infty)\to L^{q}(\partial\Omega)$ is
      a measurable function satisfying
      \begin{displaymath}
        \norm{g_{0}(t)}_{q}\le \frac{\omega}{t}\int_{0}^{t}e^{\omega
          (t-r)}\,
        \lnorm{\frac{\td }{\dt}_{\!\! +}\!\! e^{-r
      (\Lambda+F)}h_{0}}_{q}\,\dr \quad\text{for a.e. $t>0$.}
       \end{displaymath}
\end{enumerate}
\end{corollary}

The statements~\eqref{cor:regularity-semigroup-claim1} and \eqref{cor:regularity-semigroup-claim3} follow from
Theorem~\ref{thm:main2} and
Corollary~\ref{coro:comparision}. Statements~\eqref{cor:regularity-semigroup-claim2}
and \eqref{cor:regularity-semigroup-claim4} are direct applications
of~\cite[Theorem~2.9 and Theorem~4.14]{MR4200826} (see
also~\cite{hauer2020perturbation} including a correction, and \cite{MR4031770} for
inequality~\eqref{eq:45}).

\begin{remark}
  \begin{enumerate}[(a)]
    \item Even though, the semigroup
      $\{e^{-(\Lambda_{p}+F)}\}_{t\ge 0}$ generated by the negative
      Dirichlet-to-Neumann operator $-\Lambda_{p}$ associated with the
      $p$-Laplace operator $\Delta_{p}$, $1<p<\infty$, admits an
      \emph{$L^{q}$-$L^{r}$ regularization effect} for $1\le q<r\le \infty$ of the form
      \begin{displaymath}
        \norm{e^{-(\Lambda+F)}h_{0}}_{r}\lesssim
        \dfrac{\norm{h_{0}}_{q}^{\gamma}}{t^{\delta}},\qquad t>0,
      \end{displaymath}
      with exponents $\gamma:=\gamma(p,d,r,q)$,
      $\delta:=\delta(p,d,r,q)>0$, we stress that
      one can not expect a similar regularization effect for the semigroup
      $\{e^{-(\Lambda+F)}\}_{t\ge 0}$ generated by the negative
      Dirichlet-to-Neumann operator $-\Lambda$ associated with the
      $p$-Laplace operator $\Delta_{1}$ since the
      \emph{trace-Sobolev} inequality on $BV(\Omega)$, merely maps
      into $L^{1}(\partial\Omega)$. We refer the interested reader to the
      monograph~\cite{CoulHau2017} for further discussion on this
      topic.\medskip
      
  \item We recall from the \emph{linear semigroup theory} the
    following striking theorem by Lotz~\cite{MR797538} (see also
    ~\cite[Corollary~4.3.19]{MR2798103}):\medskip

    {\bfseries Theorem.}\;\emph{If $\{e^{-tA}\}_{t\ge 0}$ is a
      strongly continuous linear semigroup on the Banach space
      $L^{\infty}(\Sigma,\mu)$, where $(\Sigma,\mu)$ is a measure
      space, then the infinitesimal generator $-A$ has to be a bounded
      linear operator on $L^{\infty}(\Sigma,\mu)$.}\medskip

    Despite the linearity, the Dirichlet-to-Neumann operator
    $\Lambda_{\vert L^{\infty}}$ maps boun\-ded sets of
    $L^{\infty}(\partial\Omega)$ into (possibly several) subsets of
    the closed unit ball $\overline{B}_{L^{\infty}(\partial\Omega)}$
    in $L^{\infty}(\partial\Omega)$.  Thus, this operator provides a
    first example that Lotz's theorem might have a valid analogue in
    the nonlinear semigroup theory.
  \end{enumerate}
\end{remark}

Our last theorem is dedicated to the long-time stability of the semigroup $\{e^{-t
  (\Lambda_{\vert L^{q}}+F)}\}_{t\ge 0}$ generated by
$-(\Lambda_{\vert L^{q}(\partial\Omega)}+F)$ on $L^{q}(\partial\Omega)$.

\begin{theorem}\label{thm:stability}
 Let $1\le q\le \infty$ and $F$ be given by~\eqref{eq:28} with $f$
  satisfying~\eqref{eq:27}. Then the following statements hold.
  \begin{enumerate}
     \item \label{thm:stability-claim1} (\textrm{Energy decreasing}) For every $h_{0}\in
         L^{1}(\partial\Omega)$, the energy functional $\varphi$ given
         by~\eqref{eq:4} is monotonically decreasing along the
         trajectory
         \begin{displaymath}
           \{e^{-t (\Lambda+F)}h_{0}\,\vert\,t\ge 0\}.
         \end{displaymath}
         In particular, one has that
         \begin{displaymath}
           \varphi_{\infty}:=\lim_{n\to \infty} \varphi(e^{-t
             (\Lambda+F)}h_{0})\qquad\text{exists.}
         \end{displaymath}

        \item \label{thm:stability-claim2} (\textrm{Conservation of mass}) If $F\equiv 0$, then
          one has that
          \begin{displaymath}
            \int_{\partial\Omega}e^{-t
              \Lambda}h_{0}\,\dH^{d-1}=\overline{h}_{0}
            :=\tfrac{1}{\mathcal{H}^{d-1}(\partial\Omega)}\int_{\partial\Omega}h_{0}\,\dH^{d-1}
            \qquad\text{ for all $t\ge 0$}
          \end{displaymath}
          and all $h_{0}\in L^{1}(\partial\Omega)$.

          \item \label{thm:stability-claim3} (\textrm{Long-time stability in $L^{q}(\partial\Omega)$}) If $F\equiv 0$, then
          for every $h_{0}\in L^{q}(\partial\Omega)$ and $q<\infty$,
          then one has that
          \begin{displaymath}
            \lim_{t\to \infty}e^{-t \Lambda}h_{0}=\overline{h}_{0}
            \qquad\text{ in $L^{q}(\partial\Omega)$}
          \end{displaymath}
          and $\varphi_{\infty}=\varphi(\overline{h}_{0})=0$.

          \item \label{thm:stability-claim4} (Entropy-Transport inequality)
            If $F\equiv 0$, then there is a $C>0$ such that
    \begin{displaymath}
      \norm{e^{-t  \Lambda}h_{0}-\overline{h_{0}}}_{1}\le C\,
      \varphi(e^{-t  \Lambda}h_{0})\qquad\text{for all $t>0$;}
    \end{displaymath}

    \item \label{thm:stability-claim5} For every $h_{0}\in
      L^{2}(\partial\Omega)$, one has that
      \begin{displaymath}
      \varphi(e^{-t  \Lambda}h_{0})\le
      2\dfrac{\norm{h_{0}}_{2}^{2}}{t}
      \qquad\text{for all $t>0$.}
    \end{displaymath}
  \end{enumerate}
\end{theorem}

The statements of Theorem~\ref{thm:stability} are established
in the
Propositions~\ref{prop:energy-decreasing}-\ref{propo:convergence} in
Section~\ref{subsec:long-time-stability}.

\begin{remark}{(Conjecture)}
  We conjecture that for every $q>>1$ large enough and $h_{0}\in
  L^{q}(\partial\Omega)$, the trajectory $t\mapsto
  e^{-t\Lambda}h_{0}-\overline{h_{0}}$ extincts in finite time.
\end{remark}

We conclude this first section with some interesting remarks and
historical development on the $1$-Laplace operator and the
Dirichlet-to-Neumann operator.

\begin{remark}\label{rem:thm-main2}
  \begin{enumerate}[(a)]
      \item \label{rem:thm-main2-b} It is well-known that for given $h\in
      L^{2}(\partial\Omega)$, there is a vector field
    $\z_{h}\in L^{\infty}(\Omega,\R^{d})$ satisfying \eqref{z21intro}-\eqref{z32intro} for some weak
    solution $u_{h}\in BV(\Omega)$ of Dirichlet problem~\eqref{eq:21}.
   By definition of $\Lambda_{\vert
      L^2}$, it is clear that $(h,[\z_{h},\nu])\in
    \Lambda_{\vert L^2}$. Now, on the one hand, an integrating by parts
    (Proposition~\ref{prop:ibp}) gives that
    \begin{displaymath}
      \int_{\partial\Omega}[\z_{h},\nu]\,u_{h}\,\dH^{d-1}=\int_{\Omega}\abs{Du_{h}}.
    \end{displaymath}
    But on the other hand, the integral equality
    \begin{equation}
      \label{eq:18}
      \int_{\partial\Omega} [\z,\nu]\,h\,\dH^{d-1}=\int_{\Omega}\abs{Du}
    \end{equation}
    is, in general, not true since for the notion of \emph{weak
    solutions} $u_{h}$ of Dirichlet problem~\eqref{eq:21N}, it is not
  required that the Dirichlet
  boundary condition $u_{h}=h$ on $\partial\Omega$ is satisfied in the \emph{trace sense}:
  that is, there is a $H\in BV(\Omega)$ such that $\T(H)=h$ and
  $H-u_{h}\in BV_{0}(\Omega)$. Here, we denote by $BV_{0}(\Omega)$ the closure
  $\overline{C^{\infty}_{c}(\Omega)}^{\mbox{}_{BV(\Omega)}}$ of the
  set of test functions $C^{\infty}_{c}(\Omega)$ in $BV(\Omega)$.\medskip

\item Our comment in \eqref{rem:thm-main2-b} of this remark provides a
  strong reasoning, but not an explicit proof, for why the recently
  developed theory~\cite{MR3465809} of \emph{$j$-elliptic functionals}
  can not be applied to the functional
        \begin{equation}
          \label{eq:16}
          \hat{\varphi}(u)=\int_{\Omega} \abs{Du},\qquad(u\in V_2(\Omega)),
        \end{equation}
        where
        $V_{2}(\Omega):= \big\{ u\in BV(\Omega) \;\Big\vert\; \T(u)\in
        L^{2}(\partial\Omega) \big\}$, in order to obtain
        well-posedness of the Cauchy problem~\eqref{eq:24} in
        $L^{2}(\partial\Omega)$. To be more precise, we briefly recall
        from~\cite{MR3465809} that the $j$-subdiffer\-ential operator
        $\partial_{j}\hat{\varphi}$ in $L^{2}(\partial\Omega)$ for
        $\hat{\varphi}$ given by~\eqref{eq:16} and
        $j=\T_{\vert V_{2}}$ the standard trace operator
        $\T : BV(\Omega)\to L^{1}(\partial\Omega)$ restricted on
        $V_{2}$ is  defined by the
        set of all pairs
        $(h,g)\in L^{2}(\partial\Omega)\times
        \overline{B}_{L^{\infty}(\partial\Omega)}$, for which there are
        a weak solution $u_{h}\in V_{2}$ of Dirichlet
        problem~\eqref{eq:21} satisfying $u_{h}=h$ in the traces sense, and a vector field
        $\z_{h}\in L^{\infty}(\Omega;\R^{d})$ satisfying
        \eqref{z21intro}-\eqref{eq:98} with $u_{h}$
        and~\eqref{eq:92}. Clearly, one has that $\partial_{\T_{\vert
            V_{2}}}\hat{\varphi}\subseteq \Lambda_{\vert
          L^{2}}$. But we claim that the equation
        \begin{equation}
          \label{eq:17}
          \Lambda_{\vert L^{2}}=\partial_{\T_{\vert V_{2}}}\hat{\varphi}
        \end{equation}
        can not be true in general. To see this, we recall that Sprandlin
        and Tamasan~\cite{MR3298723} (cf. \cite{MR3713816})
        constructed a boundary function
        $h_{0}\in L^{\infty}(\mathcal{S}_{1})$ on the the unit circle
        $\mathcal{S}_{1}$ in the plane $\R^{2}$, for there is no
        solution $u_{h_{0}}\in BV(D_{1})$ of the
        minimization problem
        \begin{equation}
          \label{eq:25}
          \inf\Big\{ \int_{D_{1}} \abs{Dv}\, \Big\vert\, v \in
            BV(D_{1}),\;  v= h_{0}\text{ in the weak \emph{sense of traces}} \Big\},
        \end{equation}
        where we write $D_{1}$ to denote the open unit disc in
        $\R^{2}$. Hence, one has that
        $h_{0}\notin D(\partial_{\T_{\vert V_{2}}}\hat{\varphi})$. But
        on the other hand, since the effective domain
        $D(\Lambda_{\vert L^{2}})$ of $\Lambda_{\vert L^{2}}$ is
        $L^{2}(\partial\Omega)$, and since $\mathcal{S}_{1}$ is
        compact, we have that $h_{0}\in D(\Lambda_{\vert L^{2}})$,
        showing~\eqref{eq:17} can't be true.\medskip

    \item Suppose $\Omega$ is a bounded Lipschitz domain in $\R^{d}$ whose boundary
      $\partial\Omega$ satisfies the following two conditions:
      \begin{enumerate}[(i)]
      \item For every $x\in \partial\Omega$ there exists a $r_{0}>0$
        such that for every set $A\subset\subset B(x,r_{0})$ of finite
        perimeter (that is, $P(A,\Omega):=\abs{D\mathds{1}_{A}}(\Omega)$ is finite), one has that
        \begin{displaymath}
          P(\Omega,\R^{d})\le P(\Omega\cup A,\R^{d});
        \end{displaymath}
        \item For every $x\in \partial\Omega$ and every $r>0$
        there is a set $A\subset\subset B(x,r)$ of finite
        perimeter such that
        \begin{displaymath}
          P(\Omega, B(x,r))> P(\Omega\setminus A, B(x,r)).
        \end{displaymath}
      \end{enumerate}
      Then by~\cite[Theorem~3.7 \& Corollary~4.2]{MR1172906} and by the
        characterization~\cite[Theorem~1.1]{MR3263922} of functions of
        least gradients and weak solutions to Dirichlet
        problem~\eqref{eq:21}, for every boundary data
        $h\in C(\partial\Omega)$ there is a unique weak solution
        $u\in BV(\Omega)$ of~\eqref{eq:21} satisfying $u=h$ a.e. on
        $\partial\Omega$. Due to this existence and uniqueness result,
        we know that at least in this situation, the integral
        equation~\eqref{eq:18} holds for every boundary data
        $h\in C(\partial\Omega)$.
  \end{enumerate}
\end{remark}

The $1$-Laplace operator $\Delta_{1}$ is not only interesting from his
geometric perspectives and its applications to engineering sciences,
but also by his mathematical challenges. For a given
$u\in BV(\Omega)$, $\Delta_{1}u$ is the \emph{scalar mean
  curvature} of the level sets of $u$. Thus, every level surface
$\{u= t\}$ of a function $u$ of least gradient has mean
curvature zero; a necessary condition for functions $u$ whose
super-level sets $\{u\ge t\}$ are area-minimizing. Functions of
least gradient do not have too much regularity, in the sense, that
even though $u$ might be essentially bounded, necessarily,
$u$ need not admit a continuous representative on
$\overline{\Omega}$. In fact, in some applications, this property of
functions of least gradient is strongly desired, for example, in image
processing (see~\cite{MR2033382} and the references therein); if the
nonlinear diffusion process associated with $\Delta_{1}$ is used to
recover a blurred picture $u_{0} : \Omega\to [0,1]$,
($\Omega\subseteq \R^{2}$), then the contours in $u_{0}$ are
maintained and not smoothened as compared to diffusion processes
involving linear or degenerate differential operators. But the
operator $\Delta_{1}$ also appears in other engineering fields. For
example in free material design (see~\cite{MR3596671}), or
conductivity imaging (see~\cite{MR3739314}).\medskip

If $\Omega$ represents, for example, an electricity conducting medium,
then the operator $\Lambda$ associated with the classical Laplace
operator $\Delta u:=\sum_{i=1}^{d}D_{ii}u$ appears in a
natural way in measuring the current through the boundary for given
voltages on the boundary. Thus the operator $\Lambda$ is the main object in
Calder\'on's inverse problem \cite{MR590275}. The Dirichlet-to-Neumann operator
$\Lambda$ can be constructed with various kind of differential
operators (linear, nonlinear, singular, or degenerate) provided the
corresponding Dirichlet problem admits a solution; for $1<p<\infty$,
the Dirichlet-to-Neumann operator $\Lambda$
associated with the $p$-Laplace operator
$\Delta_{p}u:=\divi\left(\abs{Du}^{p-2}Du\right)$ is
also referred to as the interior capacity operator
(cf~\cite{MR1036731}) and was studied intensively by many authors
including by D{\'{\i}}az and Jim\'{e}nez~\cite{MR938995},
Ammar, Andreu and Toledo~\cite{MR2294196}, Arendt and Ter
Elst~\cite{MR2823661}, Salo and Zhong~\cite{MR3023384},
Brander~\cite{MR3415587}, the first author~\cite{MR3369257}, and with
co-authors~\cite{MR3465809,CoulHau2017,MR4041276}.

\section{Preliminaries.}
\label{sec:prelim}

We begin by summarizing some fundamental notions,
definitions, and results which we will apply later in this paper.

%
%

\subsection{Functions of bounded variation}\label{subsec:bv}
We begin by recalling some fundamental facts about functions of
bounded variation. For more details on this topic, we refer the
interested reader to~\cite{MR1857292}, or~\cite{MR1014685}.\medskip

Let $\Omega$ an open subset of $\R^d$, $d\ge 1$. Then, a function
$u \in L^1(\Omega)$ is said to be a \emph{function of bounded
  variation in $\Omega$}, if the distributional partial derivatives
$D_{1}u:=\tfrac{\partial u}{\partial x_{1}}$, $\dots,$
$D_{d}u:=\tfrac{\partial u}{\partial x_{d}}$ are finite Radon measures in $\Omega$, that is, if
\begin{displaymath}
  \int_{\Omega} u\,D_{i}\varphi\,\dx= -
  \int_{\Omega}\varphi\,\textrm{d} D_{i}u
\end{displaymath}
for all $\varphi\in C^{\infty}_{c}(\Omega)$, $i=1, \dots, d$. The
linear vector space of functions $u \in L^1(\Omega)$ of bounded
variation in $\Omega$ is denoted by $BV(\Omega)$. Further, we set
$Du=(D_{1}u,\dots,D_{d}u)$ for the \emph{distributional gradient} of
$u$. Then, $Du$ belongs to the class $M^{b}(\Omega,\R^{d})$ of $\R^{d}$-valued bounded
Radon measure on $\Omega$, and 
throughout this paper, we either write $\abs{Du}(\Omega)$ or
$\int_\Omega \vert Du \vert$ to denote the \emph{total variation
  measure} of $Du$. The space $BV(\Omega)$ equipped with the norm
\begin{displaymath}
  \norm{u}_{BV(\Omega)}:= \norm{u}_{L^1(\Omega)} +
  \abs{Du}(\Omega),
\end{displaymath}
forms a Banach space. For $u \in L^{1}_{loc}(\Omega)$, the
\emph{variation of $u$ in $\Omega$} is defined by
\begin{displaymath}
  V(u,\Omega):=\sup \left\{ \int_{\Omega} \, u \,
    {\rm div} \, \z \, dx\;\Big\vert\; \z \in C^{\infty}_{0}(\Omega,
    \R^d), |\z(x)| \leq 1 \, \, \hbox{for $x \in \Omega$}
  \right\}
\end{displaymath}
and if $u$ is continuously differentiable, then an integration by
parts shows that $V(u,\Omega)=\int_{\Omega}\abs{\nabla u}\,\dx$. The
variation $V(\cdot,\Omega)$ is directly related to $BV(\Omega)$ via
the property (cf~\cite[Proposition~3.6]{MR1857292}), that a function
$u\in L^{1}(\Omega)$ belongs to $BV(\Omega)$ if and only if
$V(u,\Omega)$ is finite. In addition, it is worth noting that
$V(u,\Omega)=\abs{Du}(\Omega)$ for $u\in BV(\Omega)$ and
$u\mapsto V(u,\Omega)$ is lower semicontinuous with respect to the
$L^{1}_{loc}$-topology.

By Riesz's theorem (cf~\cite[Theorem~6.19]{MR924157}), the
dual space $(C_{0}(\Omega))^{\ast}$ is isometrically isomorphic with
the space $M^{b}(\Omega)$ of bounded Radon-measures. Thus, for a sequence
$(\mu_n)_{n\ge 1}$ and $\mu$ in $M^{b}(\Omega)$, $(\mu_n)_{n\ge 1}$ is said to be
\emph{weakly$\mbox{}^{\ast}$-convergent to $\mu$ in $M^{b}(\Omega)$} if
\begin{displaymath}
  \int_{\Omega}\xi \textrm{d}\mu_{n}\to \int_{\Omega}\xi
  \textrm{d}\mu\qquad\text{for every $\xi\in C_{0}(\Omega)$.}
\end{displaymath}
Following this definition, one calls a sequence $(u_n)_{n\ge 1}$ in $BV(\Omega)$ to be
\emph{weakly$\mbox{}^{\ast}$-convergent to $u$ in $BV(\Omega)$} if
$u_n\to u$ in $L^{1}(\Omega)$ as $n\to +\infty$ and $Du_{n}$
weakly$\mbox{}^{\ast}$-converges to $Du$ in
$M^{b}(\Omega;\R^{d})$ as $n\to +\infty$. For this type of convergence
the following compactness result holds.

\begin{theorem}[{\cite[Theorem~3.23]{MR1857292}}]\label{thm:BVcompactness}
  Let $\Omega$ be a bounded domain with a Lipschitz continuous
  boundary. Then, every bounded sequence $(u_{n})_{n\ge 1}$ in
  $BV(\Omega)$ admits a weakly$\mbox{}^{\ast}$-convergent subsequence
  in $BV(\Omega)$.
\end{theorem}

The following result due to Modica~\cite{MR921549} is crucial for the
minimization problem related to the Dirichlet problem for the
$1$-Laplace operator.

\begin{proposition}[{\cite[Proposition~1.2]{MR921549}}]\label{prop:4}
  Let $\Omega$ be a bounded domain with a boundary
  $\partial\Omega$ of class $C^1$, and $\tau : \partial\Omega\times \R\to \R$ be a
  contraction in the second variable, uniformly with
  respect to the first one. Then, the functional $F : BV(\Omega)\to
  \R$ given by
  \begin{displaymath}
    F(u)=\int_{\Omega}\abs{Du}+\int_{\partial\Omega}\tau(x,\T(u))\,\dH^{d-1}
  \end{displaymath}
  is lower semicontinuous on $BV(\Omega)$ with respect
  to the topology of $L^1(\Omega)$.
\end{proposition}

According to \cite[Theorem~5.3.1]{MR1158660} and
\cite[Theorem~3.87]{MR1857292}, if $\Omega$ is an open and bounded
subset of $\R^{d}$ with a Lipschitz continuous boundary
$\partial\Omega$, then there is a bounded linear mapping
$\T : BV(\Omega)\to L^{1}(\partial\Omega)$ assigning to each
$u\in BV(\Omega)$ an element $\T(u)\in L^{1}(\partial\Omega)$ such
that for $\mathcal{H}^{d-1}$-almost every $x \in \partial \Omega$, one
has that $\T(u)(x)\in \R$ and
\begin{displaymath}
  \lim_{\rho \downarrow 0} \rho^{-d}\int_{\Omega \cap B_\rho(x)}
  \vert u(y) - \T(u)(x) \vert dy = 0.
\end{displaymath}
Moreover, $\T$ is surjective, and for every $u\in BV(\Omega)$,
\begin{equation}
  \label{eq:2}
  \int_{\Omega}u\,\textrm{div} \xi\,\dx=-\int_{\Omega}\xi\cdot
  \textrm{d}Du
  +\int_{\partial\Omega} (\xi\cdot\nu)\,\T(u)\,\td\mathcal{H}^{d-1}
\end{equation}
for all $\xi\in C^1(\R^{d},\R^{d})$, where $\nu$ denotes the outer
unit normal vector on $\partial\Omega$. We call $\T(u)$ the (weak)
\emph{trace} of $u$ and $\T$ the \emph{trace operator} on
$BV(\Omega)$. Note, if there is no danger of confusion, we sometimes
also write simply $u$.

An important notion of convergence of measures in $M^{b}(\Omega)$
is the \emph{strict convergence}; we say that  a sequence $(u_n)_{n\ge
  1}$ in $BV(\Omega)$ \emph{converges strictly} to some $u\in
BV(\Omega)$ if $\int_\Omega \vert Du_{n} \vert$ converges to $\int_\Omega
\vert Du \vert$ and $u_{n}$ converges to $u$ in
$L^1(\Omega)$. We have the following useful result.

\begin{proposition}[{\cite[Theorem~3.88]{MR1857292}}]
  \label{propo:continuity-of-trace}
  Let $\Omega$ be an open bounded subset of $\R^d$ with a Lipschitz
  continuous boundary $\partial\Omega$. Then, the trace operator
  $\T : BV(\Omega)\to L^{1}(\partial\Omega)$ is continuous from
  $BV(\Omega)$ equipped with the strict topology to
  $L^1(\partial\Omega)$, and surjective. Moreover, there exists a
  constant $C >0$ such that
  \begin{equation}\label{BounTrac}
  \norm{\T(u)}_1 \le \norm{u}_{BV(\Omega)} \quad \text{for all $u \in BV(\Omega)$.}
  \end{equation}
\end{proposition}

The next proposition on Poincar\'e's inequality for $BV$-functions
can be deduced from \cite[Lemma~4.1.3]{MR1014685}. Here, we use the
notation $\overline{h}$ to denote the \emph{mean value} of a function $h\in
L^{1}(\partial\Omega)$, defined by
\begin{displaymath}
  \overline{h}=\tfrac{1}{\mathcal{H}^{d-1}(\partial\Omega)}\int_{\partial\Omega}h\,\dH^{d-1}.
\end{displaymath}

\begin{proposition}
  \label{prop:poincare}
   Let $\Omega$ be an open bounded subset of $\R^d$ with a Lipschitz
  continuous boundary $\partial\Omega$. Then, there is a constant $C>0$
  such that
  \begin{equation}
    \label{eq:61}
    \norm{\T(u)-\overline{\T(u)}}_{1}\le C\,\int_{\Omega}\abs{Du}\qquad\text{for
      all $u\in BV(\Omega)$.}
  \end{equation}
\end{proposition}

Next, we recall the following embedding theorems as stated in \cite[Theorem
6.5.7/1, Theorem 9.5.7]{MR2777530} and \cite{MR3683462}.

\begin{theorem}\label{mazinq}
  Suppose that $\Omega \subset \R^d$ is an open bounded set with
  Lipchitz boundary.  Then for every
  $1 \le p <\infty$, there is a constant $C_{p,d} > 0$ such that
  \begin{displaymath}\label{inmaz0}
    \norm{u}_{L^{\frac{pd}{d-p}}(\Omega)} \le C_{p,d}
    \left[\norm{\nabla u}_{L^p(\Omega)} + \norm{\T(u)}_{L^p(\partial \Omega)}\right]
  \end{displaymath}
  function $u\in W^{1,p}(\Omega)$. Moreover,
  \begin{equation}
    \label{eq:7}
    \norm{u}_{L^{\frac{d}{d-1}}(\Omega)} \le C_d
    \left[\abs{Du}(\Omega)
      + \norm{\T(u)}_{L^1(\partial \Omega)}\right]
  \end{equation}
  for every $u \in BV (\Omega)$.
\end{theorem}

Further, we later need the following \emph{chain rule} for $BV$-functions.

\begin{theorem}[{\cite[Theorem~3.96]{MR1857292}, Chain rule in $BV$}]
  \label{thm:BV-chain-rule}
  Let $u\in BV(\Omega)$ and $T : \R\to \R$ a Lipschitz continuous
  function with Lipschitz constant $M$ and suppose $T(0)=0$ if the
  Lebesgue measure $\abs{\Omega}$ of $\Omega$ is infinite. Then
  $v=T\circ u\in BV(\Omega)$ satisfying
  \begin{displaymath}
    \abs{Dv}(A)\le M\,\abs{Du}(A)\qquad\text{for every open subset
      $A\subseteq \Omega$}.
  \end{displaymath}
\end{theorem}

%
%


For the rest of this subsection, we recall several results from \cite{MR750538} (see also
cf~\cite{MR2033382}). Let $\Omega \subset \R^d$ be a bounded domain with a
Lipschitz continuous boundary $\partial\Omega$.\medskip

For $1\le p\le d$ and $p^{\mbox{}_{\prime}}$ given by
  $1=\tfrac{1}{p}+\tfrac{1}{p^{\prime}}$, we introduce the following spaces
\begin{align*}
X_p(\Omega) &:= \Big\{ \z \in L^{\infty}(\Omega, \R^d) \Big\vert \; \divi(\z)\in L^p(\Omega) \Big\}, \text{ and}\\
BV(\Omega)_{p^\prime} & :=BV(\Omega)\cap L^{p^{\mbox{}_{\prime}}}(\Omega).
\end{align*}
Then, by the Maz'ya-Sobolev embedding~\eqref{eq:7}, one has that
\begin{displaymath}
  BV(\Omega) =BV(\Omega)_{d/(d-1)}.
\end{displaymath}

Now, for given $w\in C^{1}(\Omega)$,
$\z\in L^{\infty}(\Omega;\R^{d})$, and  open subset $A$ of $\Omega$, the integral
\begin{equation}
  \label{eq:11}
\mu(A):=\int_{A} \z\cdot\nabla w\,\dx
\end{equation}
defines a signed Radon measure on $\Omega$. Inspired by~\eqref{eq:11}, one can define a
bilinear mapping
$(\cdot,D\cdot) : X_p(\Omega)\times BV(\Omega)_{p^{\mbox{}_{\prime}}}\to
M^{b}(\Omega)$ by
\begin{equation}
    \label{eq:22}
\langle (\z,Dw),\varphi \rangle = - \int_{\Omega} w \, \varphi \,
{\rm div}(\z) \, \dx - \int_{\Omega} w \, \z \cdot \nabla \varphi
\, \dx
\end{equation}
for all $\varphi\in C^{\infty}_{0}(\Omega)$, $\z \in X_p(\Omega)$ and
$w \in BV(\Omega)_{p^{\mbox{}_{\prime}}}$. From~\eqref{eq:22}, one obtains the following.

\begin{proposition}[{\cite[Theorem~1.5]{MR750538}}],
 For every open set $A\subseteq \Omega$ and for all $\varphi\in
 C^{\infty}_{0}(A)$, one has that
 \begin{equation}
   \label{eq:124}
    \abs{\langle (\z,Dw),\varphi \rangle}\le
    \norm{\varphi}_{L^{\infty}(A)}\,\norm{\z}_{L^{\infty}(A)}\int_{A}\abs{D
      w}.
  \end{equation}
  In particular, for given
  $\z \in X_p(\Omega)$ and $w \in BV(\Omega)_{p^{\mbox{}_{\prime}}}$,
  the linear functional
  $(\z,Dw): C^{\infty}_{0}(\Omega) \rightarrow \R$ is a signed Radon
  measure in $\Omega$ with total variation measure $\abs{(\z,Dw)}$.
\end{proposition}

We shall denote by
\begin{displaymath}
 \int_{A}(\z,Dw) \quad\text{and}\quad
\int_{A}\abs{(\z,Dw)}
\end{displaymath}
the value of the measures $(\z,Dw)$ and $\abs{(\z,Dw)}$ on Borel subsets $A$
of $\Omega$. In fact,
the measure $(\z,Dw)$ represents an extension
of~\eqref{eq:11}; namely, one has that
\begin{equation}
  \label{eq:41}
  \int_{\Omega} (\z,Dw) = \int_{\Omega} \z \cdot \nabla w \, \dx
\end{equation}
for every $w \in W^{1,1}(\Omega) \cap L^{\infty}(\Omega)$ and $\z \in
X_p(\Omega)$.

\begin{proposition}[{\cite[Corollary~1.6]{MR750538}}]\label{prop:absolutely-cont-measures}
  Let $\Omega$ be a bounded domain with a Lipschitz-continuous boundary
 $\partial\Omega$ and for $1\le p\le d$ and $p^{\mbox{}_{\prime}}$ given by
  $1=\tfrac{1}{p}+\tfrac{1}{p^{\prime}}$, let $u\in
  BV(\Omega)_{p^{\mbox{}_{\prime}}}$ and $\z\in X_{p}(\Omega)$. Then
  the measures $(\z,Du)$ and $\abs{(\z,Du)}$ are absolutely continuous
  with respect to the measure $\abs{Du}$ in $\Omega$ and
  \begin{equation}
    \label{eq:32}
    \labs{\int_{B} (\z,Du)} \le \int_{B}
    \abs{(\z,Du)} \le \norm{\z}_{L^{\infty}(A;\R^{d})} \int_{B} \vert Du \vert
  \end{equation}
  for every Borel set $B$ and all open sets $A$ such that $B\subseteq A \subseteq \Omega$.
\end{proposition}

Thus, there is a density function $\theta(\z, Dw, \cdot) \in L^{1}(\Omega, |Dw|)$ satisfying
\begin{equation}
  \label{value-DN-derivative}
  \theta(\z, Dw, \cdot)=\frac{\td (\z,Dw)}{\td |Dw|} \;\text{ with
  }\;|\theta(\z, Dw, x)|=1\text{ for $|Dw|$-a.e. $x\in \Omega$.}
\end{equation}
The function $\theta(\z, Dw, \cdot)$ is called the Radon--Nikod\'ym
derivative of $(\z,Dw)$ with respect to $|Dw|$.  Moreover, the
following results holds.

\begin{proposition}[{\cite{MR750538}, Chain rule for $(\z,D\cdot)$}]\label{prop:chain-rule}
  Let $\Omega$ be a bounded domain with a Lipschitz-continuous boundary
 $\partial\Omega$ and for $1\le p\le d$ and $p^{\mbox{}_{\prime}}$ given by
  $1=\tfrac{1}{p}+\tfrac{1}{p^{\prime}}$, let $\z \in X_p(\Omega)$ and
$w \in BV(\Omega)_{p^{\mbox{}_{\prime}}}$. Then, for every Lipschitz continuous,
monotonically increasing function $T : \R \rightarrow \R$, one has that
  \begin{equation}\label{E1paring12}
    \theta(\z, D(T \circ w),x) = \theta (\z, Dw, x)\qquad\text{for
      $\abs{Dw}$-a.e. $x\in \Omega$.}
  \end{equation}
\end{proposition}

Further, there is a unique linear extension
$\gamma : X_p(\Omega) \rightarrow L^{\infty}(\partial \Omega)$
satisfying
\begin{equation}
  \label{eq:64}
  \norm{\gamma(\z)}_{\infty} \le \norm{\z}_{\infty}
\end{equation}
and
  \begin{displaymath}
    \gamma(\z)(x) = \z(x)\cdot\nu(x)\text{ for every $x \in
      \partial\Omega$ and $\z \in C^1(\overline{\Omega},\R^d)$.}
  \end{displaymath}

\begin{definition}[{\cite{MR750538}}]\label{def:weak-trace-of-z}
  For every $\z\in X_p(\Omega)$, we write $[\z, \nu]$ for $\gamma(\z)$
  and call $[\z, \nu]$ the \emph{weak trace} of the normal component of $\z$.
\end{definition}

With these preliminaries in mind, we can now state the \emph{generalized integration
  by parts formula} for functions $w\in BV(\Omega)$.

\begin{proposition}[{\cite{MR750538}, Generalized integration by parts}]\label{prop:ibp}
  Let $\Omega$ be a bounded domain with a Lipschitz-continuous boundary
 $\partial\Omega$ and let $1\le p\le d$ and $p^{\mbox{}_{\prime}}$ be given by
  $1=\tfrac{1}{p}+\tfrac{1}{p^{\prime}}$. Then
  \begin{equation}\label{Green0}
    \int_{\Omega} w \,\divi(\z)\, \dx + \int_{\Omega} (\z, Dw) =
    \int_{\partial \Omega} [\z, \nu] w \, \dH^{d-1}.
  \end{equation}
  for every $\z \in X_p(\Omega)$ and $w \in BV(\Omega)_{p^{\mbox{}_{\prime}}}$.
\end{proposition}

The next density result is quite useful.

\begin{lemma}[{\cite[Lemma~1.8 \& Lemma~5.2]{MR750538}}]\label{lem:2}
 Let $\Omega$ be a bounded domain with a Lipschitz-continuous boundary
 $\partial\Omega$. Then for
every $v\in BV(\Omega)_{2}$, there is sequence $(v_{n})_{n\ge 1}$ of
functions $v_{n}\in C^{\infty}(\Omega)\cap BV(\Omega)$ with trace
$\T(v_{n})=\T(v)$,
\begin{align*}
  \lim_{n\to \infty}v_{n}&= v\qquad\text{ in $L^{2}(\Omega)$, and}\\
  \lim_{n\to\infty}\int_{\Omega}(\z,Dv_{n})&= \int_{\Omega}(\z,Dv).
\end{align*}
\end{lemma}

We conclude this section on $BV$-functions with the following
proposition on convergence results.

\begin{proposition}
  \label{prop:3}
  Let $\Omega$ be a bounded domain with a Lipschitz-continuous boundary
 $\partial\Omega$ and for $1\le p\le d$ and $p^{\mbox{}_{\prime}}$ given by
  $1=\tfrac{1}{p}+\tfrac{1}{p^{\mbox{}_{\prime}}}$, suppose $(\z_{n})_{n\ge 1}$
  and $\z$ are elements of $X_{p}(\Omega)$ such that
  \begin{align}
    \label{eq:76}
    \lim_{n\to \infty}\z_{n}=&\;\z\hspace{1.75cm}\text{weakly$\mbox{}^{\ast}$ in
                               $L^{\infty}(\Omega;\R^{d})$, and}\\
    \label{eq:77}
    \lim_{n\to \infty}\divi(\z_{n})=&\;\divi(\z)\qquad\text{weakly in
                               $L^{p}(\Omega)$.}
  \end{align}
  Then, the following statements hold.
  \begin{enumerate}[(1.)]
    \item For every $v\in BV(\Omega)_{p^{\mbox{}_{\prime}}}$,
    \begin{equation}
      \label{eq:78}
      \lim_{n\to \infty}(\z_{n},Dv)=(\z,Dv)\qquad\text{weakly$\mbox{}^{\ast}$ in
                                      $M^{b}(\Omega)$}
    \end{equation}
    \item \label{prop:3-claim2} For $v\in BV(\Omega)_{p^{\mbox{}_{\prime}}}$,
      \eqref{eq:78} implies that
      \begin{equation}
      \label{eq:80}
      \lim_{n\to \infty}\int_{\Omega}(\z_{n},Dv)=\int_{\Omega}(\z,Dv).
    \end{equation}

    \item \label{prop:3-claim3} If, in addition, there is an $C>0$ such that
      \begin{equation}
        \label{eq:127}
        \sup_{n\ge 1}\norm{\divi(\z_{n})}_{\infty}\le C
      \end{equation}
      and if there are $(v_{n})_{n\ge 1}$, $v$ in
      $BV(\Omega)$ such that
      \begin{equation}
        \label{eq:125}
        \lim_{n\to \infty}v_{n}=\;v\qquad\text{weakly$\mbox{}^{\ast}$ in
                               $BV(\Omega)$,}
      \end{equation}
      then
      \begin{equation}
        \label{eq:126}
      \lim_{n\to \infty}(\z_{n},Du_{n})=(\z,Dv)\qquad\text{weakly$\mbox{}^{\ast}$ in
                                      $M^{b}(\Omega)$.}
      \end{equation}
  \end{enumerate}
\end{proposition}

The first limit~\eqref{eq:78} is obtained by a light modification of
the proof of~\cite[Proposition~2.1]{MR750538} and for the proofs
of~\eqref{eq:80}, we were inspired by the proof
of~\cite[Lemma~1.8]{MR750538}. For convenience, we give here the
details.

\begin{proof}
  Let $v\in BV(\Omega)_{p^{\mbox{}_{\prime}}}$. By~\eqref{eq:76}, one has that
  \begin{equation}
    \label{eq:83}
    \sup_{n\ge
      1}\norm{\z_{n}}_{\infty}=:M\qquad\text{is finite and }\norm{\z}_{\infty}\le M.
  \end{equation}
  Applying~\eqref{eq:83} to~\eqref{eq:32},
  one sees that
  \begin{equation}
    \label{eq:86}
    \labs{\int_{\Omega}(\z_{n},Dv)}\le \int_{\Omega}\abs{(\z_{n},Dv)}\le M\,\int_{\Omega}\abs{Dv}.
  \end{equation}
  Thus and by~\eqref{eq:124}, for verifying that~\eqref{eq:78} holds; that is,
  \begin{equation}
    \label{eq:79}
    \lim_{n\to\infty}\langle (\z_{n},Dv),\varphi\rangle
    =\langle (\z,Dv),\varphi\rangle
  \end{equation}
  for every $\varphi\in C_{0}(\Omega)$, it is sufficient to check this
  limit holds for every test functions $\varphi\in
  C_{c}^{\infty}(\Omega)$. But for $\varphi\in
  C_{c}^{\infty}(\Omega)$, \eqref{eq:22} holds, and so
  by~\eqref{eq:76} and~\eqref{eq:77}, one has that
  \begin{align*}
    \langle (\z_{n},Dv),\varphi\rangle&= - \int_{\Omega} v \, \varphi \,
    \divi(\z_{n}) \, \dx - \int_{\Omega} v \, \z_{n} \cdot \nabla \varphi\,
                                        \dx\\
    &\to - \int_{\Omega} v \, \varphi \,
    \divi(\z) \, \dx - \int_{\Omega} v \, \z \cdot \nabla \varphi\,
                                        \dx\\
    &=\langle (\z,Dv),\varphi\rangle
  \end{align*}
  as $n\to\infty$, which proves~\eqref{eq:79}. Next, to see that~\eqref{eq:80} holds, we perform a
  $2\varepsilon$-argument. For this, let $\varepsilon>0$. Since the
  total variational measure $\abs{Dv}$ is a bounded Radon measure on
  $\Omega$, there is a subset $U\Subset \Omega$ such that
  \begin{equation}
    \label{eq:81}
    \int_{\Omega\setminus U}\abs{Dv}\le \frac{\varepsilon}{4 M}
  \end{equation}
  and for every $\varphi\in C_{c}^{\infty}(\Omega)$, there is an $N(\varepsilon,\varphi)\in \N$ such that
  \begin{equation}
    \label{eq:82}
    \abs{\langle (\z_{n},Dv),\varphi\rangle-\langle
                                                          (\z,Dv),\varphi\rangle}<\frac{\varepsilon}{2}
  \end{equation}
  for all $n\ge N(\varepsilon,\varphi)$. Now, we choose a test
  function $\varphi\in C_{c}^{\infty}(\Omega)$ with the properties that
  $\varphi\equiv 1$ on $\overline{U}$ and $0\le \varphi\le 1$ on $\Omega$. Then,
  by~\eqref{eq:32},~\eqref{eq:83}, \eqref{eq:81} and~\eqref{eq:82}, one finds that
  \begin{align*}
    \labs{\int_{\Omega}(\z_{n},Dv)-\int_{\Omega}(\z,Dv)}
      & \le \abs{\langle (\z_{n},Dv),\varphi\rangle-\langle
                                                          (\z,Dv),\varphi\rangle}\\
    &\qquad \int_{\Omega}(1-\varphi)\,\td\abs{(\z_{n},Dv)}+
      \int_{\Omega}(1-\varphi)\,\td\abs{(\z,Dv)}\\
    & \le \frac{\varepsilon}{2}+\int_{\Omega\setminus U}\abs{(\z_{n},Dv)}+
      \int_{\Omega\setminus U}\abs{(\z,Dv)}\\
    & \le \frac{\varepsilon}{2}+2M\int_{\Omega\setminus U}\abs{Dv}\\
    & \le \frac{\varepsilon}{2}+2M\frac{\varepsilon}{4 M}=\varepsilon
  \end{align*}
  for all $n\ge N(\varepsilon,\varphi)$, proving~\eqref{eq:80}.

  To see that the final claim~\eqref{prop:3-claim3} holds, we first
  note that by~\eqref{eq:22}, the limits~\eqref{eq:76},
  \eqref{eq:77},~\eqref{eq:127} and~\eqref{eq:125} yield that
  \begin{align*}
    \langle (\z_{n},Dv_{n}),\varphi \rangle &= - \int_{\Omega} v_{n} \, \varphi \,
{\rm div}(\z_{n}) \, \dx - \int_{\Omega} v_{n} \, \z _{n}\cdot \nabla \varphi
                                              \, \dx\\
    &\to - \int_{\Omega} v \, \varphi \,
{\rm div}(\z) \, \dx - \int_{\Omega} v \, \z\cdot \nabla \varphi
      \, \dx\\
    &=\langle (\z,Dv),\varphi \rangle
  \end{align*}
  for every $\varphi\in C_{0}(\Omega)$, showing that $((\z_{n},Dv_{n}))_{n\ge
    1}$ converges to $(\z,Dv)$ in the distributional sense. But since $(v_{n})_{n\ge 1}$ is
  bounded in $BV(\Omega)$,~\eqref{eq:86} applied to $w=v_{n}$ gives
  that
  \begin{displaymath}
    \labs{\int_{\Omega}(\z_{n},Dv_{n})}\le \int_{\Omega}\abs{(\z_{n},Dv_{n})}\le
    M\,\sup_{n\ge 1}\int_{\Omega}\abs{Dv_{n}}\le M\,C.
  \end{displaymath}
  Thus and by~\eqref{eq:124}, convergence of
  $((\z_{n},Dv_{n}))_{n\ge 1}$ in the distributional sense
  yields~\eqref{eq:126}.
\end{proof}

%
%
%
%

\subsection{A Primer on Nonlinear Semigroups}
\label{sec:semigroups}

Throughout this second part of the preliminary Section~\ref{sec:prelim},
suppose that $X$ is a Banach space with norm $\norm{\cdot}_{X}$,
$X^{\prime}$ its dual space,
$\langle\cdot,\cdot\rangle_{X^{\prime},X}$ the duality brackets on
$X^{\prime}\times X$, and let $I$
denote the \emph{identity} on $X$.\medskip

In this framework, an operator $A$ on $X$ is a possibly nonlinear
and multivalued mapping $A : X\to 2^{X}$. It is standard to identify
an operator $A$ on $X$ with its \emph{graph}
\begin{displaymath}
  A:=\Big\{(u,v)\in X\times X\,\Big\vert\, v\in Au\Big\}\quad
  \text{in $X\times X$}
\end{displaymath}
and so, one sees $A$ as a subset of $X\times X$. The set
$D(A):=\{u\in X\,\vert\,Au\neq \emptyset\}$ is called the
\emph{domain} of $A$, and
$\textrm{Rg}(A):=\bigcup_{u\in D(A)}Au\;\subseteq H$ the \emph{range}
of $A$.  Further, for an operator $A$ on $H$, the \emph{minimal
  section} $A^{\circ}$ of $A$ is given by
\begin{displaymath}
  A^{\circ} := \Big\{(u,v) \in A\,\Big\vert \,
  \norm{v}_{X}=\displaystyle\min_{w\in Au}\norm{w}_{X} \Big\}.
\end{displaymath}

\begin{definition}\label{def:quasi-accretive}
  For $\omega\in \R$, an operator $A$ on $X$ is called
  \emph{$\omega$-quasi $m$-accretive operator} on $X$ if $A+\omega I$
  is \emph{accretive}, that is, for every $(u,v)$,
  $(\hat{u},\hat{v})\in A$ and every $\lambda\ge 0$,
  \begin{displaymath}
    \norm{u-\hat{u}}_{X}\le \norm{u-\hat{u}+\lambda (\omega (u-\hat{u})+ v-\hat{v})}_{X}
  \end{displaymath}
  and if for some (or equiv., all) $\lambda>0$ satisfying
  $\lambda\,\omega<1$, the \emph{range condition}
  \begin{equation}
    \label{eq:29}
    Rg(I+\lambda A)=X
  \end{equation}
  holds.
\end{definition}

It is worth mentioning that in the case $X=H$ is a Hilbert
  space with inner product $(\cdot,\cdot)_{H}$, the notion of $A$
  being accretive is equivalent to $A$ being \emph{monotone}; that is,
 \begin{displaymath}
  (\hat{v}-v,\hat{u}-u)_{H}\ge 0\qquad\text{ for all
    $(u,v)$, $(\hat{u},\hat{v})\in A$.}
\end{displaymath}
If $A$ is a monotone operator on $H$ then $A$ is called \emph{maximal
  monotone} if $A$ is monotone and, in addition, the range condition~\eqref{eq:29}
holds.

Another important class of operators is given by the \emph{sub-differential operator}
\begin{equation}
  \label{eq:9}
  \partial_{X\times X^{\prime}}\phi:=
  \Bigg\{(u,x')\in X\times X^{\prime}\,\Bigg\vert\,
  \begin{array}[c]{c}
\langle
    x',v-u\rangle_{X^{\prime},X}\le \phi(v)-\phi(u)\\
    \hspace{2.4cm}\text{ for all }v\in X
  \end{array}
\Bigg\}
\end{equation}
of a proper, convex and lower semicontinuous function
$\phi : X\to (-\infty,+\infty]$ on Banach space $X$. If $X=H$ is a Hilbert space, then
after identifying the dual space $H^{\prime}$ with $H$, the
sub-differential operator
$\partial_{H\times H^{\prime}}\phi$ becomes a
maximal monotone operator $A$ on $H$. In this setting we simply write
$\partial_{H}\phi$ for the operator $\partial_{H\times H^{\prime}}\phi$. In fact, $\partial_{H}\phi$
satisfies the following stronger type of monotonicity.

\begin{definition}
  An operator $A$ on $H$ is called \emph{cyclically monotone} if for every
  finite sequence $((u_{i},v_{i}))_{i=0}^{n}\subseteq A$, one has that
  \begin{displaymath}
    (u_{0}-u_{n},v_{n})_{H}+\sum_{j=1}^{n}(u_{j}-u_{j-1},v_{j-1})_{H}\le 0.
  \end{displaymath}
\end{definition}

If $A=\partial_{H}\phi$ has a sub-differential structure, then $A$ is cyclically monotone
(cf.,~\cite[(2.1) in Section~2]{MR193549}). We recall this result in
the next theorem.

\begin{theorem}[{Rockafellar~\cite{MR193549}, cf.,\cite[Th\'eor\`eme~2.5 \&
    Corollaire~2.8]{MR0348562}}]
  \label{thm:cycl-monotone}
  Let $A$ be a monotone operator on a Hilbert space $H$. Then, the following
  statements hold.
  \begin{enumerate}[(1.)]
  \item $A$ is cyclically monotone if and only if there is a proper,
    convex and lower semicontinuous function
    $\phi : H\to (-\infty,+\infty]$ such that
    $A\subseteq \partial_{H}\phi$.
  \item\label{thm:cycl-monotone-2} $A$ is maximal cyclically monotone if and
    only if there is a proper, convex and lower semicontinuous
    function $\phi : H\to (-\infty,+\infty]$ such that
    \begin{equation}
      \label{eq:108}
      A= \partial_{H}\phi.
    \end{equation}
    Moreover, the function $\phi$ in~\eqref{eq:108} is unique up to an
    arbitrary additive constant.
  \end{enumerate}
\end{theorem}

One can sharpen the statement~\eqref{thm:cycl-monotone-2} in
Theorem~\ref{thm:cycl-monotone} for homogeneous operators $A$ of
degree $\alpha\in \R$.

\begin{definition}\label{def:homogeneous-operators}
  An operator $A$ on $X$ is called \emph{homogeneous of order
    $\alpha\in \R$} if $(0,0)\in A$ and
  for every $u\in D(A)$ and $\lambda\ge 0$, one has that
  $\lambda u\in D(A)$ and
\begin{equation}
  \label{eq:10}
  A(\lambda u)=\lambda^{\alpha}Au.
\end{equation}
Similarly, we call a functional $\phi : X\to (-\infty,\infty]$
\emph{homogeneous of order $\alpha\in \R$} if $0\in D(\phi)$ with
$\phi(0)=0$ and for every $u\in D(\varphi)$ and $\lambda\ge 0$, one
has that $\lambda u\in D(\phi)$ and
\begin{displaymath}
  \phi(\lambda u)=\lambda^{\alpha}\,\phi(u).
\end{displaymath}
\end{definition}

\begin{theorem}
  \label{thm:cycl-monotone-homogeneous}
  Let $A$ be a homogeneous operator on a Hilbert space $H$ of order $\alpha\in
  \R$. Then $A$ is maximal cyclically monotone if and
    only if \eqref{eq:108} holds for a unique proper, convex, lower semicontinuous
   $\phi : H\to [0,+\infty]$ satisfying
    \begin{equation}
      \label{eq:109}
      \phi(0)=0\quad\text{ and }\quad
      \phi(\lambda u)=\lambda^{\alpha+1}\phi(u)\quad\text{ for
        all $u\in D(A)$.}
    \end{equation}
\end{theorem}

\begin{proof}
  By Theorem~\ref{thm:cycl-monotone}, we have that $A$ is cyclically
  monotone if and only if there is a proper, convex and lower semicontinuous
    function $\phi : H\to [0,+\infty]$ such that \eqref{eq:108}
    holds. Moreover, the functional $\phi$ is given by
  \begin{displaymath}
    \phi(u):=\sup_{n\in \N}\sup_{((u_{i},v_{i}))_{i=0}^{n}\subseteq A}
    \Big\{(u-u_{n},v_{n})_{H}+\sum_{j=1}^{n}(u_{j}-u_{j-1},v_{j-1})_{H}\Big\},
    u\in H,
  \end{displaymath}
  (cf. Rockafellar~\cite{MR193549}). Thus, it remains to
 verify that this functional $\phi$ satisfies
  $\phi\ge 0$ on $H$ and~\eqref{eq:109}. To see this, let
  $(u_{0},v_{0})=(0,0)$ in the definition of $\phi$. Then by the
  cyclic monotonicity of $A$, one has that $\phi(0)=0$. Moreover, since $(0,0)\in
  \partial_{H}\phi$, it follows from the convexity of $\phi$ that
  $\phi\ge 0$ on $H$. It is left to verify that
  \begin{equation}
    \label{eq:14}
    \phi(\lambda u)=\lambda^{\alpha+1}\phi(u)\quad\text{ for
        all $u\in D(A)$.}
  \end{equation}
  Note, since $D(A)\subseteq D(\phi)$, it follows from the
  homogeneity of $A$ that for every $u\in D(A)$ and $\lambda\ge 0$,
  one has $\lambda u\in D(\phi)$. Now, fix $u\in D(A)$ and $\lambda
  >0$. Since for every finite sequence
  $((u_{i},v_{i}))_{i=0}^{n}\subseteq A$, one has that $((\lambda
  u_{i},\lambda^{\alpha}v_{i}))_{i=0}^{n}\subseteq A$, it follows that
  \begin{align*}
    &(\lambda u-\lambda u_{n},\lambda^{\alpha+1}v_{n})_{H}
    +\sum_{j=1}^{n}(\lambda u_{j}-\lambda
      u_{j-1},\lambda^{\alpha+1}v_{j-1})_{H}\\
    &\qquad=\lambda^{\alpha+1}
       \,\Big\{(u-u_{n},v_{n})_{H}+\sum_{j=1}^{n}(u_{j}-u_{j-1},v_{j-1})_{H}\Big\}\\
    &\qquad \le \lambda^{\alpha+1}\,\phi(u)
  \end{align*}
  for every finite sequence
  $((u_{i},v_{i}))_{i=0}^{n}\subseteq A$. Hence, by taking the
  supremum over all $((u_{i},v_{i}))_{i=0}^{n}\subseteq A$ in the
  above inequality yields that
  \begin{displaymath}
    \phi(\lambda u)\le \lambda^{\alpha+1}\,\phi(u).
  \end{displaymath}
  On the other hand, for every finite sequence
  $((u_{i},v_{i}))_{i=0}^{n}\subseteq A$,
  \begin{align*}
    & \lambda^{\alpha+1}\,\Big\{(u-u_{n},v_{n})_{H}+\sum_{j=1}^{n}(u_{j}-u_{j-1},v_{j-1})_{H}\Big\}\\
    &\qquad
    =(\lambda u-\lambda u_{n},\lambda^{\alpha+1}v_{n})_{H}
    +\sum_{j=1}^{n}(\lambda u_{j}-\lambda u_{j-1},\lambda^{\alpha+1}v_{j-1})_{H}\\
    &\qquad \le \phi(\lambda u).
  \end{align*}
  Taking again the supremum over all
  $((u_{i},v_{i}))_{i=0}^{n}\subseteq A$ in this inequality leads to
  the reverse inequality
  $\lambda^{\alpha+1}\,\phi(u)\le \phi(\lambda u)$. The
  uniqueness of a convex, proper, lower semicontinuous functional
  $\phi$ satisfying~\eqref{eq:109} follows from the fact that
  $\phi(0)=0$. This completes
  the proof of this theorem.
\end{proof}

Convex functionals, which are homogeneous of order $\alpha+1$, $\alpha\in \R$,
admit the following important property.

\begin{proposition}
  Let $\phi : X \to [0,+\infty]$ be a convex, proper, and lower
  semicontinuous functional on a Banach space $X$ and suppose, there
  is an $\alpha\in \R$ such that \eqref{eq:109} holds. Then, one has that
  \begin{equation}
    \label{eq:110}
    (\alpha+1)\phi(u)=\langle x',u\rangle_{X^{\prime},X}\qquad
    \text{for every $(u,x^{\prime})\in \partial_{X\times X^{\prime}}\phi$.}
  \end{equation}
\end{proposition}

\begin{proof}
  Let $(u,x')\in \partial_{X\times X^{\prime}}\phi$. Then, by the definition of the
  sub-differential $\partial_{X\times X^{\prime}}\phi$, one has that
  \begin{displaymath}
    \langle x',w-u\rangle_{X^{\prime},X}\le \phi(w)-\phi(u)
  \end{displaymath}
  for every $w\in H$. For $t\in (-1,1]$, let $w=(1+t)u$. Then
  by~\eqref{eq:109}, $w\in D(\phi)$, the previous inequality
  reduces to
  \begin{displaymath}
    t\,\langle x',u\rangle_{X^{\prime},X}\le \Big((1+t)^{\alpha+1}-1\Big)\,\varphi(u).
  \end{displaymath}
  From this, we can deduce that~\eqref{eq:110} holds by first taking $t>
  0$ then dividing by $t$ and subsequently sending $t\to 0+$, and the
  proceed in a similar way for $t<0$.
\end{proof}

Further, we need to introduce the notion of \emph{even} functionals.

\begin{definition}
  \label{def:even}
  We call a proper functional $\phi : X \to (-\infty,+\infty]$ with
  effective domain $D(\phi):=\{x\in X\,\vert\,\phi(x)<\infty\}$ in
  a Banach space $X$ to be \emph{even} if for every $x\in D(\phi)$,
  also $-x\in D(\phi)$ and $\phi(x)=\phi(-x)$.
\end{definition}

If $A$ is $\omega$-quasi $m$-accretive operator on a Banach space $X$, then by the classical
existence theory (see, e.g.,~\cite[Theorem~6.5]{Benilanbook},
or~\cite[Corollary~4.2]{MR2582280}), the first-order Cauchy problem
\begin{equation}
  \label{eq:30}
  \begin{cases}
    \displaystyle\frac{\td u}{\dt}+A(u(t))\ni g(t) &\text{on $(0,T)$,}\\
    \phantom{\displaystyle\frac{\td u}{\dt}+A(}u(0)=u_{0};
  \end{cases}
\end{equation}
is well-posed for every $u_{0}\in \overline{D(A)}^{\mbox{}_{X}}$, and
$g\in L^{1}(0,T;X)$ in the following \emph{mild sense}.

\begin{definition}\label{def:mild}
  For given $u_{0}\in \overline{D(A)}^{\mbox{}_{X}}$ and
  $g\in L^{1}(0,T;X)$, a function $u\in C([0,T];X)$ is called a
\emph{mild solution} of Cauchy problem~\eqref{eq:30} if $u(0)=u_{0}$
and for every
$\varepsilon>0$, there is a \emph{partition}
$\tau_{\varepsilon} : 0=t_{0}<t_{1}<\cdots < t_{N}=T$ and a \emph{step
  function}
  \begin{displaymath}
    u_{\varepsilon,N}(t)=u_{0}\,\mathds{1}_{\{t=0\}}(t)+\sum_{i=1}^{N}u_{i}\,\mathds{1}_{(t_{i-1},t_{i}]}(t)
    \qquad\text{for every $t\in [0,T]$}
  \end{displaymath}
  satisfying
  \begin{align*}
   \bullet\qquad &t_{i}-t_{i-1}<\varepsilon\qquad\text{ for all $i=1,\dots,N$,}\\
    \bullet\qquad &\sum_{N=1}^{N}\int_{t_{i-1}}^{t_{i}}\norm{g(t)-\overline{g}_{i}}\,\dt<\varepsilon\qquad
      \text{where
      $\overline{g}_{i}:=\frac{1}{t_{i}-t_{i-1}}\int_{t_{i-1}}^{t_{i}}g(t)\,\dt$,}\\
   \bullet\qquad & \frac{u_{i}-u_{i-1}}{t_{i}-t_{i-1}}+A u_{i}\ni
                   \overline{g}_{i}
                   \qquad\text{ for all $i=1,\dots,N$,}
  \end{align*}
  and
  \begin{displaymath}
    \sup_{t\in [0,T]}\norm{u(t)-u_{\varepsilon,N}(t)}_{X}<\varepsilon.
  \end{displaymath}
\end{definition}

Mild solutions are limits of step functions which are not necessarily
differentiable in time. This leads to the notion of \emph{strong
  solution} to Cauchy problem~\eqref{eq:30}.

\begin{definition}
  \label{def:strong}
  For given $u_{0}\in \overline{D(A)}^{\mbox{}_{X}}$ and
  $h\in L^{1}(0,T;X)$, a function $u\in C([0,T];X)$ is called a
\emph{strong solution} of Cauchy problem~\eqref{eq:30} if
$u(0)=u_{0}$, and for a.e. $t\in (0,T)$, $u$ is
differentiable at $t$, $u(t)\in D(A)$, and $Au(t)\ni h(t)-\frac{\td u}{\dt}(t)$.
\end{definition}

Further, if $A$ is quasi $m$-accretive, then the family
$\{e^{-t A}\}_{t=0}^{T}$ of mappings $e^{-t A} :
\overline{D(A)}^{\mbox{}_{X}}\times L^{1}(0,T;X)\to
\overline{D(A)}^{\mbox{}_{X}}$ defined by
\begin{equation}
  \label{eq:46}
  e^{-t A}(u_{0},g):=u(t)\qquad\text{for every $t\in [0,T]$, $u_{0}\in
    \overline{D(A)}^{\mbox{}_{X}}$, $g\in L^{1}(0,T;X)$,}
\end{equation}
where $u$ is the unique mild solution of Cauchy problem~\eqref{eq:30},
belongs to the following class.

\begin{definition}\label{def:semigroup}
   Given a subset $C$ of $X$, a family $\{e^{-t A}\}_{t=0}^{T}$ of mapping $e^{-t A} : C\times L^{1}(0,T;X)\to
  C$ is called a \emph{strongly continuous semigroup of
    quasi-contractive mappings $e^{-t A}$} if $\{e^{-t A}\}_{t=0}^{T}$
  satisfies the following three properties:
  \begin{itemize}
  \item (\emph{semigroup property}) for every $(u_{0},f)\in \overline{D(A)}^{\mbox{}_{X}}\times L^{1}(0,T;X)$,
    \begin{displaymath}
      e^{-(t+s) A}(u_{0},h)=e^{-t A}(T_{s}(u_{0},h),h(s+\cdot))
    \end{displaymath}
    for every $t$, $s\in [0,T]$ with $t+s\le T$;
  \item (\emph{strong continuity}) for every
    $(u_{0},h)\in \overline{D(A)}^{\mbox{}_{X}}\times L^{1}(0,T;X)$,
    \begin{displaymath}
      \textrm{$t\mapsto e^{-t A}(u_{0},h)$ belongs to $C([0,T];X)$};
    \end{displaymath}
  \item (\emph{$\omega$-quasi contractivity}) $e^{-t A}$
    satisfies
    \begin{align*}
  &\norm{e^{-t A}(u_{0},g)-e^{-t A}(\hat{u}_{0},\hat{g})}_{X}\le
  e^{\omega t}\norm{u_{0}-\hat{u}_{0}}_{X}\\
  &\hspace{5cm}+\int_{0}^{t}e^{\omega (t-s)}\norm{g(s)-\hat{g}(s)}_{X}\,\ds
\end{align*}
  \end{itemize}
\end{definition}

Taking $g\equiv 0$ and only varying $u_{0}\in
\overline{D(A)}^{\mbox{}_{X}}$, defines by
\begin{equation}
  \label{eq:47}
  e^{-t A}u_{0}=e^{-t A}(u_{0},0)\qquad\text{for every $t\ge 0$,}
\end{equation}
a strongly continuous semigroup  $\{e^{-t A}\}_{t\ge 0}$ on $\overline{D(A)}^{\mbox{}_{X}}$ of
$\omega$-quasi contractions $e^{-t A} : \overline{D(A)}^{\mbox{}_{X}}\to \overline{D(A)}^{\mbox{}_{X}}$.
Given a family $\{e^{-t A}\}_{t\ge 0}$ of $\omega$-quasi contractions
$e^{-t A}$ on $\overline{D(A)}^{\mbox{}_{X}}$, then the operator
\begin{displaymath}
  A_{0}:=\Bigg\{(u_{0},v)\in X\times X\Bigg\vert\;\lim_{h\downarrow
    0}\frac{T_{h}(u_{0},0)-u_{0}}{h}=v\text{ in $X$}\Bigg\}
\end{displaymath}
is an $\omega$-quasi accretive well-defined mapping
$A_{0} : D(A_{0})\to X$ and called the \emph{infinitesimal generator}
of $\{e^{-t A}\}_{t\ge 0}$. If the Banach space $X$ and its dual space
$X^{\ast}$ are both uniformly convex
(see~\cite[Proposition~4.3]{MR2582280}), then one has that
\begin{displaymath}
  -A_{0}=A^{\!\circ},
\end{displaymath}
where $A^{\!\circ}$ is the \emph{minimal selection} of $A$ defined
by
\begin{equation}
  \label{eq:65}
  A^{\! \circ}:=\Big\{(u,v)\in
  A\,\Big\vert\big. \norm{v}_{X}=\inf_{\hat{v}\in Au}\norm{\hat{v}}_{X}\Big\}.
\end{equation}
For simplicity, we ignore the additional geometric
assumptions on the Banach space $X$, and refer to
the two families $\{e^{-t A}\}_{t=0}^{T}$ defined by~\eqref{eq:46} on
$\overline{D(A)}^{\mbox{}_{X}}\times L^{1}(0,T;X)$ and
$\{e^{-t A}\}_{t\ge 0}$ defined by~\eqref{eq:47} on
$\overline{D(A)}^{\mbox{}_{X}}$ as the \emph{semigroup generated by $-A$}.

%
%

\subsubsection{Completely accretive operators}
\label{sec:c-accretive-operators}

Here, we briefly recall the notion of completely accretive operators,
which was introduced by B\'enilan and Crandall~\cite{MR1164641} and
further developed in~\cite{CoulHau2017}.\medskip

We begin by introducing the framework of completely accretive
operators. Let $(\Sigma, \mathcal{B}, \mu)$ be a $\sigma$-finite
measure space, and $M(\Sigma,\mu)$ the space of $\mu$-a.e. equivalent
classes of measurable functions $u : \Sigma\to \R$. For
$u\in M(\Sigma,\mu)$, we write $[u]^+$ to denote $\max\{u,0\}$ and
$[u]^-=-\min\{u,0\}$.  We denote by
$L^{q}(\Sigma,\mu)$, $1\le q\le \infty$, the corresponding standard
Lebesgue space with norm
\begin{displaymath}
  \norm{\cdot}_{q}=
  \begin{cases}
    \displaystyle\left(\int_{\Sigma}\abs{u}^{q}\,\textrm{d}\mu\right)^{1/q} &
    \text{if $1\le q<\infty$,}\\[7pt]
    \inf\Big\{k\in [0,+\infty]\;\Big\vert\;\abs{u}\le k\text{
      $\mu$-a.e. on $\Sigma$}\Big\}
    & \text{if $q=\infty$.}
  \end{cases}
\end{displaymath}
For $1\le q<\infty$, we
identify the dual space $(L^{q}(\Sigma,\mu))'$ with
$L^{q^{\mbox{}_{\prime}}}(\Sigma,\mu)$,
where $q^{\mbox{}_{\prime}}$ is the conjugate exponent of $q$ given by
$1=\tfrac{1}{q}+\tfrac{1}{q^{\mbox{}_{\prime}}}$.

Now, let
  \begin{displaymath}
    J_0:= \Big\{ j : \R \rightarrow
    [0,+\infty]\;\Big\vert\big. \text{$j$ is convex, lower
      semicontinuous, }j(0) = 0 \Big\}.
  \end{displaymath}
  Then, for every $u$, $v\in M(\Sigma,\mu)$, we write
  \begin{displaymath}
  u\ll v \quad \text{if and only if}\quad \int_{\Sigma} j(u)
  \,\td\mu \le \int_{\Sigma} j(v) \, \td\mu\quad\text{for all $j\in J_{0}$.}
\end{displaymath}

With these preliminaries in mind, we can now state the following
definitions.

\begin{definition}
  A mapping $S : D(S)\to M(\Sigma,\mu)$ with domain $D(S)\subseteq
  M(\Sigma,\mu)$ is called a \emph{complete contraction} if
  \begin{displaymath}
    Su-S\hat{u}\ll u-\hat{u}\qquad\text{for every $u$, $\hat{u}\in D(S)$.}
  \end{displaymath}
\end{definition}

Now, we can state the
definition of completely accretive operators.

\begin{definition}
  \label{def:completely-accretive-operators}
  An operator $A$ on $M(\Sigma,\mu)$ is called \emph{completely
    accretive} if for every $\lambda>0$, the resolvent operator
  $J_{\lambda}$ of $A$ is a complete contraction, or equivalently, if
  for every $(u_1,v_1)$, $(u_{2},v_{2}) \in A$ and $\lambda >0$, one
  has that
  \begin{displaymath}
    u_1 - u_2 \ll u_1 - u_2 + \lambda (v_1 - v_2).
  \end{displaymath}
  If $X$ is a linear subspace of $M(\Sigma,\mu)$ and $A$ an operator
  on $X$, then $A$ is \emph{$m$-completely accretive on $X$} if $A$ is
  completely accretive and satisfies the \emph{range
  condition}
\begin{displaymath}
  \textrm{Rg}(I+\lambda A)=X\qquad\text{for some (or equivalently, for
    all) $\lambda>0$.}
\end{displaymath}
Further, for $\omega\in \R$, an operator $A$ on a linear subspace
$X\subseteq M(\Sigma,\mu)$ is called \emph{$\omega$-quasi
  ($m$)-completely accretive} in $X$ if $A+\omega I$ is
($m$)-completely accretive in $X$. Finally, an operator $A$ on a
linear subspace $X\subseteq M(\Sigma,\mu)$ is called \emph{quasi
  $m$-completely accretive} if there is some $\omega\in \R$ such that
$A+\omega I$ is $m$-completely accretive in $X$.
\end{definition}

Before stating a useful characterization of completely accretive
operators, we first need to introducing the following function
spaces. Let
\begin{displaymath}
    L^{1\cap \infty}(\Sigma,\mu):=L^{1}(\Sigma,\mu)\cap
    L^{\infty}(\Sigma,\mu)
    \text{ and }L^{1+\infty}(\Sigma,\mu):= L^1(\Sigma,\mu) +
  L^{\infty}(\Sigma,\mu)
\end{displaymath}
be the \emph{intersection} and the \emph{sum} space of
$L^1(\Sigma,\mu)$ and $L^{\infty}(\Sigma,\mu)$, which respectively
equipped with the norms
\begin{align*}
   \norm{u}_{1 \cap \infty}&:= \max\big\{ \norm{u}_{1},
                             \norm{u}_{\infty}\big\},\\
  \norm{u}_{1+\infty}&:= \inf \left\{ \Vert u_1 \Vert_1 + \Vert
  u_2 \Vert_{\infty} \Big\vert u = u_1 + u_2, \ u_1 \in L^1(\Sigma,\mu), u_2
  \in L^{\infty}(\Sigma,\mu) \right\}
\end{align*}
are Banach spaces. In fact, $L^{1\cap \infty}(\Sigma,\mu)$ and
$L^{1+\infty}(\Sigma,\mu)$ are respectively the smallest and the
largest of the rearrangement-invariant Banach function spaces
(cf~\cite[Chapter~3.1]{MR928802}). If $\mu(\Sigma)$ is finite, then
$L^{1+\infty}(\Sigma,\mu)=L^{1}(\Sigma,\mu)$ with equivalent norms, but
if $\mu(\Sigma)=\infty$ then
\begin{displaymath}
 \bigcup_{1\le q\le\infty}L^{q}(\Sigma,\mu) \subset L^{1+\infty}(\Sigma,\mu).
\end{displaymath}
Further, we will employ
the space
\begin{displaymath}
  L_0(\Sigma,\mu):= \left\{ u \in M(\Sigma,\mu) \;\Big\vert\,\Big.
    \int_{\Sigma} \big[\abs{u}
   - k\big]^+\,\td\mu < \infty \text{ for all $k > 0$} \right\},
\end{displaymath}
which equipped with the $L^{1+\infty}$-norm is a closed subspace of
$L^{1+\infty}(\Sigma,\mu)$. In fact, one has that (cf~\cite{MR1164641})
\begin{math}
  L_0(\Sigma,\mu) = \overline{L^1(\Sigma,\mu) \cap
  L^{\infty}(\Sigma,\mu)}^{\mbox{}_{1+\infty}}.
\end{math}
Since for every $k> 0$, the function
$T_{k}(s):=[\abs{s}-k]^+$ is a Lipschitz mapping $T_{k} : \R\to\R$
satisfying $T_{k}(0)=0$, and
by using Chebyshev's inequality, it is not difficult to see that
$L^{q}(\Sigma,\mu)\hookrightarrow L_0(\Sigma,\mu)$ for every
$1\le q<\infty$ (and $q=\infty$ if the measure $\mu(\Sigma)$ is finite).

%
%

\begin{proposition}[{\cite{MR1164641} for the case $\omega=0$, \cite{CoulHau2017}}]
  \label{prop:completely-accretive}
   Let $P_{0}$ denote the set of all functions $p\in C^{\infty}(\R)$
   satisfying $0\le T'\le 1$,
    $p'$  is compactly supported, and $x=0$ is not contained in the
   support $\textrm{supp}(p)$ of $p$. Then for $\omega\in \R$,
   an operator $A \subseteq L_{0}(\Sigma,\mu)\times
  L_{0}(\Sigma,\mu)$ is $\omega$-quasi completely accretive if and only if
   \begin{displaymath}
     \int_{\Sigma}p(u-\hat{u})(v-\hat{v})\,\dmu+\omega \int_{\Sigma}p(u-\hat{u})(u-\hat{u})\,\dmu\ge 0
   \end{displaymath}
   for every $p\in P_{0}$ and every $(u,v)$, $(\hat{u},\hat{v})\in A$.
 \end{proposition}

 The next proposition is quite useful for characterizing operators.

 \begin{proposition}[{\cite{MR1164641}}]\label{prop:characterization-of-A-L0}
   Let $X\subseteq L_{0}(\Sigma,\mu)$ be a normal Banach space and $A$
   a completely accretive operator in $X$ and let
   $\overline{A}^{\mbox{}_{L_{0}}}$ be the closure of $A$ in $L_{0}(\Sigma,\mu)$. If there is an
   $\lambda_{0}$ such that the range $\Rg(I+\lambda_{0}A)$ is dense in
   $L_{0}(\Sigma,\mu)$, then the operator
   \begin{displaymath}
     A_{X}:=\overline{A}^{\mbox{}_{L_{0}}}\cap (X\times X)
   \end{displaymath}
   is the unique $m$-completely accretive extension of $A$ in
   $X$. Moreover, $A_{X}$ can be characterized by
   \begin{displaymath}
     A_{X}=\Big\{(u,v)\in X\times X\,\Big\vert\, u - \hat{u} \ll u -
     \hat{u} + \lambda (v - \hat{v})\text{ for all
     }(\hat{u},\hat{v})\in A,\;\lambda>0\Big\}.
   \end{displaymath}
 \end{proposition}

%
%
%
%

\section{The Dirichlet Problem for the $1$-Laplace operator}
\label{sec:DP}

In this section, we review the current state of knowledge about
existence and uniqueness to the singular Dirichlet problem
\begin{equation}\label{eq:21N}
    \begin{cases}
      \displaystyle -\divi\Big(\frac{Du}{|Du|}\Big)=0 & \text{in $\Omega$,}\\
      \phantom{-\divi\Big(\frac{Du}{|Du|}\Big)}  u= h & \text{on $\partial\Omega$,}
    \end{cases}
\end{equation}
for given boundary data $h\in L^{1}(\partial\Omega)$. As mentioned at
the beginning of this paper, we always assume, if nothing else is
said, that $\Omega$ is a bounded Lipschitz domain in $\R^{d}$,
$d\ge 2$.

In order to obtain existence of solutions to Dirichlet problem~\eqref{eq:21N},
it is natural, to study the existence of a minimizer of the famous
 \emph{least gradient problem}
\begin{equation}
  \label{eq:20}
  \inf\Big\{\int_{\Omega}\abs{D v}\,\Big\vert\,v\in
  BV(\Omega),\;v=h\text{ on $\partial\Omega$}\Big\}.
\end{equation}

Existence of solutions to the minimizing problem~\eqref{eq:20} was
obtained by Parks \cite{MR0458304,MR831890} under the hypotheses
$\Omega$ is strictly convex and the boundary data $h$ satisfies the
bounded slope condition. Sternberg, Williams and
Ziemer~\cite{MR1172906} improved this result by establishing existence
and uniqueness of a minimizer
$u\in BV(\Omega)\cap C(\overline{\Omega})$ of~\eqref{eq:20} for
boundary data $u\in C(\partial\Omega)$ on bounded domains $\Omega$
with a Lip\-schitz boundary $\partial\Omega$ of non-negative mean
curvature (in the weak sense) and not being locally
area-minimizing.

On $BV(\Omega)$, there is a continuous trace operator
$\T : BV(\Omega)\to L^{1}(\partial\Omega)$ available
(see~Proposition~\ref{propo:continuity-of-trace}). Thus Sternberg,
Williams and Ziemer called in~\cite{MR1246349} a function
$u\in BV(\Omega)$ to be of \emph{least gradient} if
\begin{displaymath}
  \int_{\Omega}\abs{D u}=\min\Big\{\int_{\Omega}\abs{D v}\,\Big\vert\,v\in
  BV(\Omega),\; \T(u)=\T(v)\Big\}.
\end{displaymath}
Since for given $h\in L^{1}(\partial\Omega)$, there is a
$H\in BV(\Omega)$ satisfying $\T(H)=h$, a function $u\in BV(\Omega)$
satisfies the boundary constrain
\begin{equation}
  \label{eq:5}
  u=h\quad \text{ on $\partial\Omega$}
\end{equation}
in the \emph{traces sense} if $\T(u)=\T(H)$. In many elliptic
boundary-value problems (as for example, the Dirichlet problem
associated with the $p$-Laplace operator, see,
e.g.,~\cite{MR3369257}), it is standard that the solution attains the boundary
condition~\eqref{eq:5} merely in the sense of traces. However, by using
this weak notion of attaining the boundary condition~\eqref{eq:5},
a function $u\in BV(\Omega)$ is a minimizer of~\eqref{eq:20} if $u$ minimizes
the total variation $\int_{\Omega}\abs{Dv}$ on the affine space
$\T(H)+BV_{0}(\Omega)$ (cf.~\cite[Theorem~2.2]{MR1246349}), where
$BV_{0}(\Omega)$ is the closure of the $BV$-norm of the set of test
functions $C^{\infty}_{c}(\Omega)$. But this last problem has the two
challenges that the trace operator $\T$ is only continuous with
respect to the strict topology and of missing compactness results on $BV(\Omega)$.
Thus, to establish existence and uniqueness of a minimizer
to~\eqref{eq:20} and related problems, the continuity condition on the boundary data $h$
was used by many authors, including
Miranda~\cite{MR0222735}, Parks and Ziemer~\cite{MR808421}, Bombieri,
De Giorgi, Giusti~\cite{MR0250205}, or more recently, Jerrard,
Moradifam, and Nachman~\cite{MR3739314}.

Recently, Spradlin and Tamasan~\cite{MR3298723} constructed an
essentially bounded boundary function $h$ on the unit circle $S^{1}$
in $\R^{2}$ for which the minimizing problem~\eqref{eq:20} has no
solution $u\in BV(\Omega)$ satisfy~\eqref{eq:5} in the sense of
traces. If the set of discontinuities is countable, then in the planar
case, G\'orny~\cite{MR3813249} (see also
\cite{MR3596671,2018arXiv181111138G}, and Rybka and
Sabra~\cite{MR3596671}) could establish existence of a minimizer to
problem~\eqref{eq:20}.

This suggests that for discontinuous boundary data
$h\in L^{1}(\partial\Omega)$, the notion of traces for the boundary
condition~\eqref{eq:5} might not be the right one for establishing
existence of a minimizer to problem~\eqref{eq:20}. Thus, Rossi, Segura
and the second author~\cite{MR3263922} studied for given
$h\in L^{1}(\partial\Omega)$, the following \emph{relaxed} functional
\begin{math}
  \Phi_h: L^{\frac d{d-1}}(\Omega)
  \rightarrow (-\infty,+\infty]
\end{math}
given by
\begin{equation}
  \label{eq:101}
  \Phi_h(v) =
    \begin{cases} \displaystyle
      \int_{\Omega}\abs{Dv} + \int_{\partial \Omega} \abs{h-v}
      \, \dH^{d-1} & \text{if $v\in BV(\Omega)$,}\\[7pt]
      +\infty & \text{if $v \in  L^{\frac{d}{d-1}}(\Omega) \setminus BV(\Omega)$}.
\end{cases}
\end{equation}
The functional $\Phi_h$ is convex, lower semicontinuous on
$L^{\frac d{d-1}}(\Omega)$, and thanks to the Sobolev
inequality~\eqref{eq:7}, $\Phi_h$ is coercive. Thus, there is a
$u\in BV(\Omega)$ solving the variational problem
\begin{equation}
  \label{eq:8}
  \min_{v \in BV(\Omega)} \Phi_h(v).
\end{equation}

One easily verifies that if $u\in BV(\Omega)$ is a function of least
gradient satisfying the boundary condition~\eqref{eq:5} in the sense
of traces, then $u$ is a minimizer of problem~\eqref{eq:8}.  Moreover,
every minimizer $u$ of~\eqref{eq:8} satisfies the following inclusion of the
first variation
\begin{equation}
  \label{eq:23}
  0\in \partial_{L^{\frac{d}{d-1}}\times L^{d}(\Omega)}\Phi_h(u)\qquad\text{in
    $L^{\frac{d}{d-1}}(\Omega)\times L^{d}(\Omega)$,}
\end{equation}
which is directly related to notion of \emph{weak solutions} to
Dirichlet problem~\eqref{eq:21N}.

By characterizing the sub-differential
$\partial_{L^{\frac{d}{d-1}}\times L^{d}(\Omega)}\Phi_h$, Rossi,
Segura and the second author~\cite{MR3263922} discovered that for
boundary data $h\in L^{1}(\partial\Omega)$, a minimizer
$u_{h}$ of~\eqref{eq:8} satisfies the Dirichlet boundary
condition~\eqref{eq:5} in problem~\eqref{eq:21N} merely in the following
\emph{weaker sense}: there is a divergence free vector field $\z_{h}\in L^\infty(\Omega;\R^d)$
such that $\norm{\z_{h}}_{\infty}\le 1$ and
\begin{equation}
    \label{z42}
    [\z_{h},\nu] \in \sign(h - \T(u))\quad
    \text{$\mathcal{H}^{d-1}$-a.e. on }\partial\Omega,
  \end{equation}
where $[\z_{h},\nu]$ denotes Anzellotti's generalized \emph{normal derivative},
$\nu$ the outward-pointing unit normal vector (see
Section~\ref{subsec:bv}), and $\sign(\cdot)$ is the accretive graph in
$\R^2$ of the \emph{signum} given by
\begin{displaymath}
  \sign(r):=
  \begin{cases}
    1 & \text{if $r>0$,}\\
    [-1,1] & \text{if $r=0$,}\\
    -1 & \text{if $r<0$.}
  \end{cases}
\end{displaymath}
More precisely, they obtained the following one.

\begin{proposition}[{\cite[Theorem~2.5]{MR3263922}}]\label{prop:6}
  For $h\in L^{1}(\partial\Omega)$ and $u \in BV(\Omega)$, the
  following statements are equivalent:
\begin{itemize}
\item[(i)] $0\in \partial_{L^{\frac{d}{d-1}}\times L^{d}(\Omega)}\Phi_h(u)$.

\item[(ii)] 
  there exists a vector field
  $\z_{h}\in L^\infty(\Omega;\R^d)$ satisfying~\eqref{z42},
  \begin{align}
    \label{z21} \norm{\z_h}_\infty&\le 1,\\
    \label{z22} -\divi(\z_h)&=0\qquad \text{in
                            $\mathcal{D}^\prime(\Omega)$, and}\\
    \label{z32} (\z_h,Du)&=|Du|\qquad \text{as Radon measures.}
  \end{align}
\end{itemize}
\end{proposition}

Having this characterization in mind, every solution
$u\in BV(\Omega)$ of the constrained least gradient
problem~\eqref{eq:20} is a \emph{weak solution}
to the Dirichlet problem~\eqref{eq:21N}, and vice versa.

\begin{definition}\label{def:sols-DP}
  For given $h\in L^{1}(\partial\Omega)$, we call a function
  $u\in BV(\Omega)$ a \emph{weak solution} to Dirichlet
  problem~\eqref{eq:21N} if there exists a vector field
  $\z_{h}\in L^\infty(\Omega;\R^d)$ satisfying~\eqref{z42}--\eqref{z32}.
\end{definition}

By using Definition~\ref{def:sols-DP}, further examples could be constructed
showing the phenomenon of non-unique\-ness in Dirichlet problem~\eqref{eq:21N}.

\begin{example}\label{rem:uniqueness-of-DP}
  In~\cite{MR3263922}, the following counter example to the uniqueness of solutions to
  Dirichlet problem~\eqref{eq:21N} on the unit ball $\Omega=\{(x,y)\in \R^2 : x^2+y^2<1\}$
  was constructed for discontinuous boundary data. Let the boundary function
  $h\in L^{\infty}(\partial\Omega)$ be given (in polar coordinates) by
 \begin{displaymath}
   h(\theta):= \left\{
     \begin{array}{ll}
       \cos (2 \theta)+1, \quad &\hbox{if} \ \cos (2 \theta) >0;\\
       \cos (2 \theta)-1, \quad &\hbox{if} \ \cos (2 \theta) <0;
     \end{array}\right.
\end{displaymath}
for every $ \theta\in (-\pi,\pi]$. Now, for every $-1\le\lambda\le1$, let
$u^\lambda : \overline{\Omega}\to \R$ be given by
\begin{equation*}
    u^\lambda(x,y)=\left\{
      \begin{array}{ll}
        \displaystyle  2x^2\,,&\hbox{if }|x|> \frac{\sqrt{2}}{2}\,,
                                 |y|<\frac{\sqrt{2}}{2}\,;\\[6pt]
        \lambda\,,&\hbox{if }|x|<\frac{\sqrt{2}}{2}\,,
                     |y|<\frac{\sqrt{2}}{2}\,;\\[6pt]
        \displaystyle -2y^2\,,&\hbox{if }|x|<\frac{\sqrt{2}}{2}\,,
                   |y|>\frac{\sqrt{2}}{2}\,.
    \end{array}\right.
\end{equation*}
Then, each $u^\lambda$ is a weak solution of Dirichlet
problem~\eqref{eq:21N} satisfying the boundary conditions~\eqref{eq:5}
in the weaker sense~\eqref{z42} with $h$.
\end{example}

Example~\ref{rem:uniqueness-of-DP} and the one given in~\cite{MR3852558} demonstrate
well that smoothness of the boundary $\partial\Omega$ and other nice geometric
properties of $\Omega$ (as, for instance, convexity of $\Omega$) are
not sufficient to establish uniqueness of solutions to the Dirichlet
problem~\eqref{eq:21N} for discontinuous boundary data
$h\in L^{\infty}(\partial\Omega)$. This justifies the notation
of differential inclusion used in~\eqref{eq:23}. But, in particular,
shows that the Dirichlet-to-Neumann operator $\Lambda$ might be multi-valued.\medskip

Next, we turn to the following observation (cf.,~\cite[Remark~2.8]{MR3263922}).

\begin{theorem}\label{thm:9}
  For given $h\in L^{1}(\partial\Omega)$, let $u$ and $\hat{u}$ be two
  weak solutions of Dirichlet problem~\eqref{eq:21N} for the same
  boundary data $h$. If the vector
  field $\z\in L^\infty(\Omega;\R^d)$ satisfies \eqref{z42}--\eqref{z32}
  with respect to $u$ and $\hat{\z}_h\in L^\infty(\Omega;\R^d)$
  satisfies \eqref{z42}--\eqref{z32} with respect to $\hat{u}$, then
  $\hat{\z}_h$ also satisfies \eqref{z42}--\eqref{z32} with respect to
  $u$ and $\z_h$ satisfies \eqref{z42}--\eqref{z32} with respect to
  $\hat{u}$.
\end{theorem}

From Theorem~\ref{thm:9}, by the fact that the minimization problem~\eqref{eq:8}
always admits a weak solution, and by Proposition~\ref{prop:6}, we can
conclude the following consequence
(cf.,~\cite[Theorem~1.2]{MR3820240}).

\begin{corollary}
  \label{cor:1}
  For given boundary data $h\in L^{1}(\partial\Omega)$, there is a
  divergence-free vector field $\z_{h}\in L^\infty(\Omega;\R^d)$
  satisfying $\norm{\z_{h}}_{\infty}\le 1$ such that every weak
  solution $u$ of Dirichlet problem~\eqref{eq:21N} satisfies
  \begin{equation}
    \label{eq:95}
    \left(\z_{h},\frac{Du}{\abs{Du}}\right)=1\qquad\text{as
      Radon measures,}
  \end{equation}
  and~\eqref{z42}.
\end{corollary}

Corollary~\ref{cor:1} says that there is a divergence-free vector
field $\z_{h}$, which determines the level set of all weak solutions $u$ of Dirichlet
problem~\eqref{eq:21N}. More precisely, since $\norm{\z_{h}}_{\infty}\le
1$ on $\Omega$, \eqref{eq:32} yields that
\begin{equation}
  \label{eq:96}
  \int_{\Omega}(\z_{h},Dw)\le 1
\end{equation}
for every $w\in BV(\Omega)$ with $\abs{Dw}=1$. Note, that~\eqref{eq:96}
can well be interpreted as a \emph{Radon-measure version} of the
point-wise inequality
\begin{displaymath}
  z_{h}\cdot\xi\le 1\qquad\text{a.e. on $\Omega$}
\end{displaymath}
holding for any vector fields $\xi\in S^{d-1}$. Thus, \eqref{eq:95} says that
for every weak solutions $u$ of Dirichlet
problem~\eqref{eq:21N}, the vector field $Dw=Du/\abs{Du}$
maximizes~\eqref{eq:96} in the sense of Radon measures. Recall, for
given vector fields $\z\in \R^{d}$ with $\abs{\z}\le 1$ and $\xi\in S^{d-1}$, the equality
$z\cdot\xi=1$ implies that $\z$ and $\xi$ are parallel to each other
and $\abs{\z}=1$. Thus, and since $\abs{Dw}(\Omega)=1$, \eqref{eq:95}
one be can understood as a condition implying that the two vector
fields $\z_{h}$ and $Dw$ are parallel to each other in some weak sense.\medskip

Further, as outlined in~\cite{MR3820240}, \eqref{z42} describes the
set of possible jumps on the boundary $\partial\Omega$ of a weak
solution $u$ of~\eqref{eq:21N}. More precisely, it follows from \eqref{z42}
that up to a set of $\mathcal{H}^{d-1}$-measure zero, one has that
\begin{align*}
  &\{x\in \partial\Omega\,\vert\,\T(u)(x)>h(x)\}\subseteq
    \{x\in \partial\Omega\,\vert\, [\z_{h},\nu]=-1\},\\
  &\{x\in \partial\Omega\,\vert\,\T(u)(x)<h(x)\}\subseteq
    \{x\in \partial\Omega\,\vert\, [\z_{h},\nu]=1\},
\end{align*}
and
\begin{displaymath}
  \{x\in \partial\Omega\,\vert\,\T(u)(x)=h(x)\}\subseteq
    \{x\in \partial\Omega\,\vert\, -1\le [\z_{h},\nu]\le 1\}.
\end{displaymath}

We now turn to the Proof of Theorem~\ref{thm:9}.

\begin{proof}[Proof of Theorem~\ref{thm:9}]
  Let $u$ and $\hat{u}$ be two solutions of Dirichlet problem
  \eqref{eq:21N} for the same given boundary function
  $h\in L^{1}(\partial\Omega)$. Further, let $\z_{h}$ and
  $\hat{\z}_{h}\in L^\infty(\Omega;\R^d)$ be two vector fields satisfying
  \eqref{z42}--\eqref{z32} with respect to $u$ and $\hat{u}$,
  respectively.

  Note, by~\eqref{z22}, the two vector fields $\z_{h}$ and $\hat{\z}_{h}$
  belong to $X_{d}(\Omega)$ and by Sobolev-inequality~\eqref{eq:7},
  $(u- \hat{u})\in BV _{d/(d-1)}(\Omega)$. Thus, the
  generalized integration by parts formula~\eqref{Green0} yields
  \begin{displaymath}
    \int_{\Omega} (\z_{h}, D(u - \hat{u}))
    -  \int_{\partial\Omega} [\z_{h},\nu] \big(\T(u) - \T(\hat{u})\big)\, \dH^{N-1} = 0
  \end{displaymath}
  and
  \begin{displaymath}
    \int_{\Omega} (\hat{\z}_{h}, D(u - \hat{u}))
    - \int_{\partial\Omega} [\hat{\z}_{h},\nu] \big(\T(u) - \T(\hat{u})\big)\, \dH^{N-1} = 0\,.
  \end{displaymath}
   Subtracting these two equations from each other and using the fact
   that the pairing $(\z_{h},Dw)$ is bilinear yields
   \begin{equation}
      \label{eq:31}
     \begin{split}
       &\int_{\Omega} (\z_{h} - \hat{\z}_{h}, D(u - \hat{u})) \\
       &\qquad+ \int_{\partial\Omega}\big([\z_{h},\nu] - [\hat{\z}_{h},\nu]\big) \big(h-
       \T(u) - (h-\T(\hat{u}))\big)\, \td\mathcal H^{d-1} = 0\,.
     \end{split}
   \end{equation}
   Since $\z_{h}$ and $\hat{z}_{h}$ satisfy $\norm{\z_{h}}_{\infty}\le 1$,
   $\norm{\hat{\z}_{h}}_{\infty}\le 1$, it follows
   from~\eqref{eq:32} that
   \begin{displaymath}
     \labs{\int_{\Omega} (\z_{h},D\hat{u})}\le \labs{D\hat{u}}(\Omega)\quad
     \text{ and }\quad  \labs{\int_{\Omega}
       (\hat{\z}_{h},Du)}\le \abs{Du}(\Omega).
   \end{displaymath}
   Thus, the bilinearity of the pairing $(\cdot,D\cdot)$ yields
   \begin{equation}\label{uni2}
     \begin{split}
       \int_{\Omega}(\z_{h}-\hat{\z}_{h}, D(u - \hat{u})) &=
       \abs{D\hat{u}}(\Omega) -\int_{\Omega}(\z_{h},D\hat{u})\\
       &\qquad \qquad+\abs{Du}(\Omega) - \int_{\Omega}(\hat{\z}_{h},Du)\ge 0.
     \end{split}
      \end{equation}
   Further, by the monotonicity of the $\sign$-graph in $\R^2$, and since $\z_{h}$
   and $\hat{\z}_{h}$ satisfy~\eqref{z42}, one has that
   \begin{displaymath}
     \big([\z_{h},\nu] - [\hat{\z}_{h},\nu]\big)\,\big((h - \T(u))- (h-\T(\hat{u}))\big)\ge 0
     \quad\text{$\mathcal{H}^{d-1}$-a.e. on $\partial\Omega$}
   \end{displaymath}
   and so,
   \begin{displaymath}
     \int_{\partial\Omega}\big([\z_{h},\nu] - [\hat{\z}_{h},\nu]\big)
     \big((h - \T(u))- (h-\T(\hat{u}))\big)\, \td\mathcal H^{d-1}\ge 0
   \end{displaymath}
   Thus, \eqref{eq:31} implies that
   \begin{displaymath}
     \int_{\Omega}(\z_{h}-\hat{\z}_{h}, D(u - \hat{u})) =0
   \end{displaymath}
   or, equivalently,
   \begin{equation}
     \label{eq:33}
     \int_{\Omega}(\z_{h},D\hat{u})=|D\hat{u}|(\Omega)\quad\text{ and }
     \quad \int_{\Omega} (\hat{\z}_{h},Du)=|Du|(\Omega),
   \end{equation}
   and
   \begin{displaymath}
     \big([\z_{h},\nu] - [\hat{\z}_{h},\nu]\big)\,\big((h - \T(u))- (h-\T(\hat{u}))\big)=0
     \quad\text{$\mathcal{H}^{d-1}$-a.e. on $\partial\Omega$.}
   \end{displaymath}
   Then,
   \begin{align*}
       0&=\big([\z_{h},\nu] - [\hat{\z}_{h},\nu]\big)\, (h - \T(u))- (h-\T(\hat{u}))\\
       &= \abs{h - \hat{u}} -
       [\z_{h},\nu](h-\T(\hat{u}))+\abs{h - \T(u)} - [\hat{\z}_{h},\nu](h-\T(u))
   \end{align*}
   $\mathcal{H}^{d-1}$-a.e. on $\partial\Omega$. Since
   $\norm{[\z_{h},\nu]}_{\infty}\le 1$ and
   $\norm{[\hat{\z}_{h},\nu]}_{\infty}\le 1$, the previous equation yields that
   \begin{displaymath}
     [\hat{\z}_{h},\nu](h-\T(u)) = \abs{h- \T(u)}\quad\text{ and }\quad
     [\z_{h},\nu](h-\T(\hat{u})) = \abs{h - \T(\hat{u})}
   \end{displaymath}
   $\mathcal{H}^{d-1}$-a.e. on $\partial\Omega$. From this, we can conclude that
   \begin{displaymath}
     [\hat{\z}_{h},\nu] \in \sign(h - \T(u))\quad\text{ and }\quad
     [\z_{h},\nu] \in \sign(h-\T(\hat{u}))\quad\text{$\mathcal{H}^{d-1}$-a.e. on
   $\partial\Omega$.}
 \end{displaymath}
Further, by recalling~\eqref{value-DN-derivative}, there are
Radon-Nikod\'ym derivatives
\begin{displaymath}
\theta(\z_{h},D\hat{u},\cdot)=\frac{\td (\z_{h},
D\hat{u})}{\td \abs{D\hat{u}}}\quad\text{ and }\quad
\theta(\hat{z}_{h},Du,\cdot) =\frac{\td (\hat{\z}_{h},Du)}{\td\abs{Du}}
\end{displaymath}
satisfying $\abs{\theta(\z_{h},D\hat{u},x)}=1$
for $\abs{D\hat{u}}$-a.e. $x\in \Omega$ and $\abs{\theta(\hat{\z}_{h},Du,x)}=1$
for $\abs{Du}$-a.e. $x\in \Omega$. Applying this to~\eqref{eq:33}
yields that
\begin{displaymath}
  \int_{\Omega}\theta(\z_{h},D\hat{u},\cdot)\,\td
  \abs{D\hat{u}}=|D\hat{u}|(\Omega) \quad\text{ and }\quad
  \int_{\Omega}\theta(\hat{\z}_{h},Du,\cdot)\,\td
  \abs{D\hat{u}}=|Du|(\Omega),
\end{displaymath}
implying that $\theta(\z_{h},D\hat{u},\cdot)=1$ and
$\theta(\hat{\z}_{h},Du,\cdot)=1$ a.e. on $\Omega$. This shows that
\begin{displaymath}
  (\z_{h},D\hat{u})=|D\hat{u}|\quad\text{ and }\quad (\hat{\z}_{h}, Du)=|Du|
\end{displaymath}
as Radon measures. This completes the proof of showing that $\hat{\z}_{h}$
satisfies \eqref{z42}-\eqref{z32} with respect to $u$ and $\z_{h}$
satisfies \eqref{z42}-\eqref{z32} with respect to $\hat{u}$.
\end{proof}

Even though that there might be infinitely many divergence-free vector
fields $\z_{h}\in L^{\infty}(\Omega;\R^{d})$ to a given boundary data
$h\in L^{1}(\partial\Omega)$, the value of the integral
\begin{displaymath}
  \int_{\partial\Omega}[\z_{h},\nu]\,h\,\dH^{d-1}
\end{displaymath}
remains the same for all vector fields $\hat{\z}_{h}$
satisfying \eqref{z42}-\eqref{z32} for some $u\in BV(\Omega)$.

\begin{theorem}
  \label{thm:8}
  For every given boundary data $h\in L^{1}(\partial\Omega)$, one has that
  \begin{equation}
    \label{eq:106}
    \int_{\partial\Omega}[\z_{h},\nu]\,h\,\dH^{d-1}=\min_{v\in BV(\Omega)}\Phi_{h}(v).
  \end{equation}
  for every vector fields $\z_{h}\in L^{\infty}(\Omega;\R^{d})$
  satisfying \eqref{z42}-\eqref{z32} for some $u\in BV(\Omega)$.
\end{theorem}

\begin{proof}[Proof of Theorem~\ref{thm:8}]
  Let $h\in L^{1}(\partial\Omega)$ and $\z_{h}\in L^{\infty}(\Omega;\R^{d})$
  satisfy \eqref{z42}-\eqref{z32} for some $u\in BV(\Omega)$. Then, $u$ is a weak
  solution of Dirichlet problem~\eqref{eq:21N} and so,
  Proposition~\ref{prop:6} says that $u$ satisfies
  \begin{equation}
    \label{eq:103}
    \min_{v\in BV(\Omega)}\Phi_{h}(v)=\int_{\Omega}\abs{Du}+\int_{\partial\Omega}\abs{h-\T(u)}\,\dH^{d-1}.
  \end{equation}

  On the other hand, by~\eqref{z32}, the generalized integration by parts
  formula~\eqref{Green0}, \eqref{z22}, and~\eqref{z42}, one sees that
\begin{align*}
  &\int_{\partial\Omega}[\z_{h},\nu]\,h\,\dH^{d-1}
    -\int_\Omega \abs{Du}\\
   &\qquad =\int_{\partial\Omega}[\z_{h},\nu]\,h\,\dH^{d-1}
    -\int_{\Omega}(\z_{h},Du)\\
  &\qquad =\int_{\partial\Omega}[\z_{h},\nu]\,h\,\dH^{d-1}
    -\int_{\partial\Omega}[\z_{h},\nu]\,\T(u)\,\dH^{d-1}\\
  &\qquad
    =\int_{\partial\Omega}[\z_{n},\nu]\,(h-\T(u))\,\dH^{d-1}\\
  &\qquad
    =\int_{\partial\Omega}\abs{h-\T(u)}\,\dH^{d-1}
\end{align*}
and so,
\begin{equation}
  \label{eq:105}
  \int_{\partial\Omega}[\z_{h},\nu]\,h\,\dH^{d-1}=
  \int_\Omega \abs{Du}+\int_{\partial\Omega}\abs{h-\T(u)}\,\dH^{d-1}.
\end{equation}
 Clearly, \eqref{eq:106} follows from combining~\eqref{eq:103} with~\eqref{eq:105}.
\end{proof}

%
%
%
%

\section{A Robin-type problem for the $1$-Laplace operator}
\label{sec:Robin}

In order to show that the Dirichlet-to-Neumann operator $\Lambda$
associated with the $1$-Laplace operator $\Delta_{1}$satisfies the
range condition~\eqref{eq:29}, we recall some recent
results obtain by the second author with collaborators~\cite{MR3328128} on the
following inhomogeneous \emph{Robin-type boundary-value problem}
\begin{equation}\label{eq:63}
    \begin{cases}
      \hspace{0.35cm}\displaystyle -\Delta_1 u =0 & \text{in $\Omega$,}\\
      \displaystyle\frac{Du}{|Du|} \cdot \nu = T_{1}(g-\alpha u) & \text{on $\partial\Omega$,}
    \end{cases}
 \end{equation}
 for the $1$-Laplace operator $\Delta_{1}$, for given $\alpha>0$ and $g\in
 L^{2}(\partial\Omega)$.\medskip

In the boundary condition of problem~\eqref{eq:63}, the function $T_{1} :
\R\to \R$ given
by $T_{1}(s)=s$ if $\abs{s}\le 1$ and $T_{1}(s)=\sign(s)$ if
$\abs{s}\ge 1$,  denotes the \emph{truncator operator}, which is
necessary to add in~\eqref{eq:63}, since it preserves the condition
\begin{equation}
  \label{eq:70}
\lnorm{\frac{Du}{|Du|} \cdot \nu}_{\infty}\le 1
\end{equation}
satisfied by every solution $u$ of problem~\eqref{eq:63} (cf.,
\eqref{eq:64} and the fact that every vector field $\z$ associated with a weak
solution $u$ of~\eqref{eq:63} satisfies $\norm{\z}_{\infty}\le 1$). We emphasize
that the use of a truncator $T_{1}$ in the Robin-type boundary
condition~\eqref{eq:63} is a phenomenon, which is exclusively
generated by the structure of the $1$-Laplace
operator $\Delta_{1}$ (and its co-normal derivative).\medskip

Another reason supporting the use of the
truncator $T_{1}$ in the singular boundary-value problem~\eqref{eq:63}
is provided by studying the correct
associated energy functional; intuitively, the natural functional
associated with problem~\eqref{eq:63} (without $T_{1}$) is given by
\begin{displaymath}
  I_{\alpha,g}(u):=\int_{\Omega}\abs{Du}+\int_{\partial\Omega}
  \Big[\tfrac{\alpha}{2}\abs{\T(u)}^2-g\,\T(u)\Big]\,\dH^{d-1}, \qquad
  u\in V_{2}(\Omega),
\end{displaymath}
where the space $V_{2}(\Omega)$ is given by
\begin{displaymath}
  V_{2}(\Omega)=\Big\{u\in BV(\Omega)\,\vert\,\T(u)\in L^{2}(\partial\Omega)\Big\}.
\end{displaymath}
But the functional $I_{\alpha,g}$ is, in general, not lower
semicontinuous with respect to the $L^{1}(\Omega)$-topology
(cf.,~\cite{MR921549}). Thus, one employs instead the $L^{1}$-lower
semicontinuous envelope
\begin{equation}
  \label{eq:69}
  \Theta_{\alpha,g}(u):=\int_{\Omega}\abs{Du}+\int_{\partial\Omega}\Gamma_{g}(x,\T(u))\,\dH^{d-1},
\end{equation}
$u\in V_{2}(\Omega)$, where $\Gamma_{g}  : \partial\Omega\times\R\to \R$ is a Borel
function, which is convex and contractive with respect to the second
variable, uniformly with respect to the first one, and satisfies
$\tfrac{\partial}{\partial u}\Gamma_{g}(x,u)=T_{1}(g(x)-\alpha
u)$. Here, for the $L^1$-lower semicontinuity of the functional
$\Theta_{g}$, the contractivity property of the mapping $u\mapsto
\Gamma_{g}(x,u)$ is crucial (cf., Proposition~\ref{prop:4}).\medskip

To find the correct notion and the existence of \emph{weak solutions}
$u$ to the inhomogeneous Robin-type problem~\eqref{eq:63}, the authors
of~\cite{MR3328128}
start from the more regular Robin-type problem associated with the
$p$-Laplace operator (for $p>1$)
\begin{equation}
  \label{eq:65}
  \begin{cases}
      \hspace{1.35cm}\displaystyle -\Delta_p u_{p} =0 & \text{in $\Omega$,}\\
      \displaystyle\abs{\nabla u_{p}}^{p-2}\nabla u_{p}\cdot \nu
      = T_{1}(g-\alpha u) & \text{on $\partial\Omega$.}
    \end{cases}
\end{equation}
It is not hard to see that for every given
$g\in L^{2}(\partial\Omega)$ and $\alpha>0$, problem~\eqref{eq:65}
admits a unique weak solution $u_{p}\in W^{1,p}(\Omega)$. After
deriving \emph{a priori}-estimates for $p\in (1,2)$, they establish
in \cite[Theorem 1.1.]{MR3328128} the existence of the following type
of solutions.

\begin{definition}\label{def:Robin}
  For given $g \in L^2(\partial \Omega)$ and $\alpha>0$, we say that
  $u\in V_{2}(\Omega)$ is a \emph{weak solution} to the inhomogeneous
  Robin-type problem~\eqref{eq:63} for the $1$-Laplace operator if for
  $u$, there is a vector field $\z\in L^\infty(\Omega;\R^d)$
  satisfying~\eqref{z21}--\eqref{z32}, and
\begin{equation}
  \label{eq:66}
  [\z,\nu]=T_1(g-\alpha u)\qquad\text{$\mathcal{H}^{d-1}$-a.e. on $\partial\Omega$.}
\end{equation}
\end{definition}

For later use, we restate the existence result \cite[Theorem
1.1.]{MR3328128} with more details to the convergence by the
approximate problem~\eqref{eq:65}.

\begin{theorem}[{\cite[Theorem~1.1.]{MR3328128}}]\label{thm:6}
  Let $\Omega$ be a bounded domain with a boundary
  $\partial\Omega$ of class $C^1$. Then, for every $g \in L^2(\partial \Omega)$ and $\alpha>0$, there is
  a weak solution $u\in V_{2}(\Omega)$ of the inhomogeneous
  Robin-type problem~\eqref{eq:63} for the $1$-Laplace operator. Moreover, for every sequence
  $(p_{n})_{n\ge 1}$ in $(1,2)$ converging to $1$, there is a
  subsequence $(p_{k_{n}})_{n\ge 1}$ and a weak solution $u\in
  V_{2}(\Omega)$ the inhomogeneous
  Robin-type problem~\eqref{eq:63} for the $1$-Laplace operator such
  that
  \begin{displaymath}
    \lim_{n\to \infty}u_{p_{k_{n}}}=u\qquad\text{in $L^{q}(\Omega)$ for
      all $1\le q<\frac{d}{d-1}$,}
  \end{displaymath}
  where $u_{p_{k_{n}}}$ is the unique solution of the Robin-type
  problem~\eqref{eq:65} associated with the $p_{k_{n}}$-Laplace operator.
\end{theorem}

Further, the following relation between the the inhomogeneous
Robin-type problem~\eqref{eq:63} and the Dirichlet
problem~\eqref{eq:21N} was obtained in~\cite{MR3328128}.

\begin{proposition}[{\cite[Proposition~2.13]{MR3328128}}]
  \label{prop:2}
 Let $g$, $h\in L^2(\partial\Omega)$, $\alpha>0$, and $u\in BV(\Omega)$. Then the
 following statements hold.
   If $u$ is a weak solution to the inhomogeneous Robin-type problem~\eqref{eq:63}, then
   $u$ is a weak solution to the Dirichlet problem~\eqref{eq:21N} with
   Dirichlet boundary data
   \begin{equation}
     \label{eq:90}
     h=g-\alpha \,[\z,\nu]\qquad\text{on $\partial\Omega$,}
   \end{equation}
   in the weak sense~\eqref{z42}, where $\z\in
   L^{\infty}(\Omega;\R^{d})$ is some vector field associated with
   $u$ via the conditions~\eqref{z42}-\eqref{z32}.
\end{proposition}

%
%
%
%
\section{Proofs of the main results}
\label{sec:proofs}

This section is dedicated to outline the proofs of our main results
Theorem~\ref{thm:main1}, Theorem~\ref{thm:main2}, and
Theorem~\ref{thm:stability}. The proofs of these results are
obtained in several steps, which we fix respectively in a
separate proposition. We begin by introducing the Dirichlet-to-Neumann
operator $\Lambda$ associated with the $1$-Laplace operator
$\Delta_{1}$ as an operator in $L^{1}(\partial\Omega)$.

%
%
%
%

\subsection{The Dirichlet-to-Neumann operator in $L^{1}$}
\label{sec:DtN}
We start this subsection with the following definition.

\begin{definition}\label{def:LambdaL1}
  We define the \emph{Dirichlet-to-Neumann operator $\Lambda$ in
    $L^{1}(\partial\Omega)$ associated with the $1$-Laplace operator}
  $\Delta_{1}$ by the set of all pairs
  $(h,g)\in L^{1}(\partial\Omega)\times L^{1}(\partial\Omega)$ with
  the property that there is a weak solution $u\in BV(\Omega)$ of
  Dirichlet problem~\eqref{eq:21N} with Dirichlet data $h$ and there
  is a vector field $\z\in L^{\infty}(\Omega;\R^{d})$ associated with
  $u$ (through~\eqref{z42}-\eqref{z32}) such that
    \begin{equation}
      \label{eq:92}
      g=[\z,\nu]\qquad\text{$\mathcal{H}^{d-1}$-a.e. on $\partial\Omega$.}
    \end{equation}
\end{definition}

\begin{remark}\label{rem:DtN-in-L1}
  (a) Since the minimization problem~\eqref{eq:8} admits a solution
  for every boundary data $h\in L^{1}(\partial\Omega)$ and by
  Proposition~\ref{prop:6}, the effective domain $D(\Lambda)$ of the
  Dirichlet-to-Neumann operator $\Lambda$ associated with $\Delta_{1}$
  satisfies
  \begin{displaymath}
    D(\Lambda)=L^{1}(\partial\Omega).
  \end{displaymath}

 (b) The Dirichlet-to-Neumann operator
  $\Lambda$ associated with $\Delta_{1}$ satisfies
  \begin{equation}
    \label{eq:72}
    \Lambda\subseteq L^{1}(\partial\Omega)\times L^{\infty}(\partial\Omega)
  \end{equation}
  since for every pair $(h,g)\in\Lambda$, one has that
  \begin{displaymath}
    \norm{g}_{\infty}=\norm{[\z,\nu]}_{\infty}\le 1,
  \end{displaymath}
  where $\z\in L^{\infty}(\Omega;\R^{d})$ is any associated vector
  field to some weak solution $u\in BV(\Omega)$ of Dirichlet
  problem~\eqref{eq:21N} with Dirichlet data $h$.
\end{remark}

We come to the first property of the Dirichlet-to-Neumann operator
$\Lambda$.

\begin{proposition}\label{prop:Lambda-completely-accretive}
  The Dirichlet-to-Neumann operator $\Lambda$ associated with the
  $1$-Laplace operator $\Delta_{1}$ is completely accretive in
  $L^{1}(\partial\Omega)$.
\end{proposition}

\begin{proof}
We aim to show that
\begin{equation}
  \label{eq:71}
    \int_{\partial\Omega} (g - \hat{g})\, p(h - \hat{h}) \dH^{d-1} \ge 0
  \end{equation}
  for every $(h,g)$, $(\hat{h},\hat{g})\in \Lambda$ and
  $p \in P_0$. Note, even if the truncator $p$ in~\eqref{eq:71} would
  be the identity on $\R$, the integral in~\eqref{eq:71} would
  exist due to~\eqref{eq:72}. Now, let $(h,g)$,
  $(\hat{h},\hat{g})\in \Lambda$ and $p \in P_0$. Then by the
  definition of $\Lambda$, for each pair $(h,g)$, $(\hat{h},\hat{g})$, there
  are weak solutions $u$, $\hat{u}$ of Dirichlet
problem~\eqref{eq:21N} with Dirichlet data $h$ and $\hat{h}$,
respectively, and associated vector fields $\z$, $\hat{\z}\in L^{\infty}(\Omega;\R^{d})$ satisfying
\eqref{z42}-\eqref{z32}. By the chain rule for
$BV$-functions (Theorem~\ref{thm:BV-chain-rule}), the function
$p(u-\hat{u})$ belongs to
  $BV(\Omega)$. Thus, applying the generalized integration by parts
formula~\eqref{Green0} to $w=p(u-\hat{u})$ and to the two vector
fields $\z$ and $\hat{\z}$, respectively, and by using~\eqref{z22}, gives
\begin{displaymath}
  \int_{\Omega}(\z,D(p(u-\hat{u})))
  =\int_{\partial\Omega}[\z,\nu]p(\T(u)-\T(\hat{u}))\,\dH^{d-1}
\end{displaymath}
and
\begin{displaymath}
  \int_{\Omega}(\hat{\z},D(p(u-\hat{u})))
  =\int_{\partial\Omega}[\hat{\z},\nu]p(\T(u)-\T(\hat{u}))\,\dH^{d-1}.
\end{displaymath}
Since $g=[\z,\nu]$ and $\hat{g}=[\hat{\z},\nu]$, we can conclude from
these two integral equations that
\begin{equation}
  \label{eq:73}
  \int_{\partial\Omega} (g - \hat{g})\, p(\T(u)-\T(\hat{u})) \dH^{d-1}
  =\int_{\Omega}(\z-\hat{\z},D(p(u-\hat{u}))).
\end{equation}
Since $p$ is Lipschitz continuous and monotonically increasing, the chain rule for $(\z,D\cdot)$
(Proposition~\ref{prop:chain-rule}) yields that for the Radon-Nikod\'ym
derivative $\theta(\z, D(p(u-\hat{u})),x)$ of $(\z, D(p(u-\hat{u})))$ with
respect to the total variational measure $\abs{D(p(u-\hat{u}))}$, one
has that
\begin{displaymath}
    \theta(\z, D(p(u-\hat{u})),x) = \theta (\z, D(u-\hat{u}), x)\qquad\text{for
      $\abs{D(u-\hat{u})}$-a.e. $x\in \Omega$.}
\end{displaymath}
Moreover, the Radon-Nikod\'ym derivative
$\theta(\z- \hat{\z}, D(u- \hat{u}),x)$ of $(\z- \hat{\z}, D(u - \hat{u}))$
is positive since by the bilinearity of the pairing
$(\cdot,D\cdot)$ and by~\eqref{z32}, \eqref{z21} and~\eqref{eq:32},
one has that
\begin{align*}
  \int_\Omega (\z-\hat{\z},D(u-\hat{u}))
    &=\int_\Omega \abs{Du}-\int_\Omega (\hat{\z},Du)\\
  & \hspace{2cm}+ \int_\Omega \abs{D\hat{u}}-\int_\Omega (\hat{\z},Du)\ge 0.
\end{align*}
Thus,
 \begin{align*}
    & \int_\Omega (\z-\hat{\z},D(p(u-\hat{u})))\\
    & \hspace{2cm} = \int_\Omega \theta(\z - \hat{\z}, D(p(u - \hat{u})),x)\,\abs{D(p(u-\hat{u}))}\\
    &  \hspace{2cm} = \int_\Omega \theta(\z - \hat{\z}, D((u -
      \hat{u}),x)\,\abs{D(p(u-\hat{u}))}\ge 0.
 \end{align*}
Applying this to~\eqref{eq:73}, shows that
\begin{equation}
  \label{eq:74}
    \int_{\partial\Omega} (g - \hat{g})\, p(\T(u)-\T(\hat{u}))\, \dH^{d-1}\ge 0,
 \end{equation}
 which would complete the proof of this proposition if the weak solutions
 $u$ and $\hat{u}$ of Dirichlet problem~\eqref{eq:21N} would satisfy
 the Dirichlet boundary condition~\eqref{eq:5} in the sense of
 traces. But our notion of solutions to Dirichlet
 problem~\eqref{eq:21N} assumes only that $u$ and $\hat{u}$ satisfy
 Dirichlet boundary condition~\eqref{eq:5} in the weak
 sense~\eqref{z42}. Thus, we still need to provide an argument,
 why~\eqref{eq:74} implies the desired inequality~\eqref{eq:71}. Now,
 by~\eqref{eq:74},
 \begin{equation}
   \label{eq:75}
\begin{split}
    &\int_{\partial\Omega} (g - \hat{g})\, p(h - \hat{h}) \dH^{d-1}\\
    &\qquad \ge \int_{\partial\Omega} (g-\hat{g})\, \Big(p(h - \hat{h})-p(\T(u)-\T(\hat{u}))\Big)\,
      \dH^{d-1}\\
    &\qquad =\int_{\partial\Omega} (g-\hat{g})\,\int_{0}^{1} p'\Big(s (h
    - \hat{h})+(1-s)(\T(u)-\T(\hat{u}))\Big)\,\ds\times\\
    &\hspace{5.1cm} \times \Big[(h - \hat{h})-(\T(u)-\T(\hat{u}))\Big]\dH^{d-1}.\\
  \end{split}
\end{equation}
Using again that $g=[\z,\nu]$ and $\hat{g}=[\hat{\z},\nu]$, and since $u$ and
$\hat{u}$ satisfy the Dirichlet boundary condition~\eqref{eq:5} in the weak
 sense~\eqref{z42}, one finds that
\begin{align*}
  &\big(g - \hat{g}\big)\,\big((h- \T(u)) - (\hat{h}-\T(\hat{u}))\big)\\
  &\qquad = \abs{h - \T(u)} + \abs{\hat{h}- \T(\hat{u})} -
    [\z,\nu](\hat{h}-\T(\hat{u})) - [\hat{\z},\nu](h-\T(u)) \ge 0
\end{align*}
for $\mathcal{H}^{d-1}$-a.e. on $\partial\Omega$. Moreover, since
$p'\ge 0$, the integral
\begin{displaymath}
\int_{0}^{1} p'\Big(s (h - \hat{h})+(1-s) (\T(u)-\T(\hat{u}))\Big)\,\ds\ge 0
\end{displaymath}
$\mathcal{H}^{d-1}$-a.e. on $\partial\Omega$, the last integral on the
right hand-side in~\eqref{eq:75} is positive, implying
that~\eqref{eq:71} holds.
\end{proof}

\begin{proposition}\label{prop:Lambda-homogeneity}
  The Dirichlet-to-Neumann operator $\Lambda$ associated with the
  $1$-Laplace operator $\Delta_{1}$ is homogeneous of order zero.
\end{proposition}

\begin{proof}
  Let $\lambda>0$ and $(h,g)\in \Lambda$. Then, there are $u\in
  BV(\Omega)$ and a vector field $\z\in L^{\infty}(\Omega;\R^{d})$
  satisfying~\eqref{z42}-\eqref{z32}. First, we show that for the
  boundary data $\lambda h$, the function $\lambda u$ is a weak solution of
  Dirichlet problem~\eqref{eq:21N}. To see this, we begin by applying the
  bilinearity of $(\cdot,D\cdot)$ and homogeneity of the total
  variational measure $\abs{D\cdot}$. Then, since $u$
  satisfies~\eqref{z32}, one sees that
  \begin{displaymath}
    (\z,D(\lambda u))=\lambda (\z,Du)=\lambda\,\abs{Du}=\abs{D(\lambda
      u)}.
  \end{displaymath}
  Further, for the same vector field $\z$, which
  satisfies~\eqref{z42}-\eqref{z22}, one has that
  \begin{displaymath}
    [\z,\nu]\in \sign\Big(h-\T(u)\Big)=\tfrac{1}{\lambda}\sign\Big(\lambda h-\lambda \T(u)\Big),
  \end{displaymath}
  implying that $g=[\z,\nu]$ satisfies
  \begin{displaymath}
    [\z,\nu]\in \sign\Big(\lambda h-\lambda \T(u)\Big).
  \end{displaymath}
  Thus, $\lambda u$ is a weak solution of Dirichlet
  problem~\eqref{eq:21N} for the boundary data $\lambda h$ with the
  same value $g$ for the generalized Neumann derivative $[\z,\nu]$
  associated with the weak solution $u$ of Dirichlet
  problem~\eqref{eq:21N}. Since $(h,g)\in \Lambda$ were arbitrary, we
  thereby have shown that $\Lambda (\lambda h)=\Lambda h$ for all
  $h\in D(\Lambda)$, establishing the claim of this proposition.
\end{proof}

One of our aims is to relate the closure
$\overline{\Lambda}^{\mbox{}_{L^{1}\times L^{\infty}_{\sigma}}}$ in
$L^{1}\times L^{\infty}_{\sigma}(\partial\Omega)$ of the Dirichlet-to-Neumann
operator $\Lambda$ with a sub-differential structure $\partial \varphi$
in $L^{1}\times L^{\infty}(\partial\Omega)$. Here, we write
$L^{\infty}_{\sigma}(\partial\Omega)$ to denote
$L^{\infty}(\partial\Omega)$ equipped with the
weak$\mbox{}^{\ast}$-topology
$\sigma(L^{\infty}(\partial\Omega),L^{1}(\partial\Omega))$. For this,
we introduce the following potential candidate of a convex function.

\begin{proposition}\label{prop:DtN-functional}
  The functional $\varphi : L^{1}(\partial\Omega)\to [0,\infty)$ given
  by
  \begin{equation}
    \label{eq:107}
    \varphi(h)=\int_{\partial\Omega}
    [\z_{h},\nu]\,h\,\dH^{d-1}\qquad\text{for every
      $h\in L^{1}(\partial\Omega)$,}
  \end{equation}
  where $\z_{h}$ is any vector field
  $\z_{h}\in L^{\infty}(\Omega;\R^{d})$
  satisfying~\eqref{z42}-\eqref{z32} for some $u\in BV(\Omega)$ with
  boundary data $h$, is a well-defined convex and continuous functional on
  $L^{1}(\partial\Omega)$, which is homogeneous of order one and even.
\end{proposition}

\begin{proof}
  First, we note that thanks to Theorem~\ref{thm:8}, for given boundary value $h\in
  L^{1}(\partial\Omega)$, the value $\varphi(h)$ given
  by~\eqref{eq:107} is independent of the vector field
  $\z_{h}\in L^{\infty}(\Omega;\R^{d})$
  satisfying~\eqref{z42}-\eqref{z32} for
 some $u\in BV(\Omega)$ with boundary data $h$, and given by
 \begin{displaymath}
\varphi(h)=\min_{v\in  BV(\Omega)}\Phi_{h}(v),
\end{displaymath}
where $\Phi_{h}$ is defined by~\eqref{eq:101}. Therefore, $\varphi$ is
a well-defined, proper mapping. Next, we show that $\varphi$ is
homogeneous of order one. For this, let $h\in L^{1}(\partial\Omega)$
and $u \in BV(\Omega)$ a weak solution of Dirichlet
problem~\eqref{eq:21N} with Dirichlet data $h$ and corresponding
vector field $\z_{h}\in L^{\infty}(\Omega;\R^{d})$
satisfying~\eqref{z42}-\eqref{z32} with respect to $u$ and $h$. Then
by Proposition~\ref{prop:6},
\begin{displaymath}
   \lambda\varphi(h)=\lambda\min_{v\in
     BV(\Omega)}\Phi_{h}(v)=\lambda \Phi_{h}(u)
   = \Phi_{\lambda h}(\lambda u)
 \end{displaymath}
for every $\lambda\ge 0$. Since
 \begin{displaymath}
   \sign(h-\T(u))=\sign(\lambda (h-\T(u)))
   =\sign(\lambda h-\T(\lambda u))
 \end{displaymath}
 for every $\lambda\ge 0$, and since
 $\z_{h}\in L^{\infty}(\Omega;\R^{d})$
 satisfies~\eqref{z42}-\eqref{z22} with respect to $u$ and $h$, it
 follows that the
 same vector field $\z_{h}$ satisfies~\eqref{z42}-\eqref{z22} with
 respect to $\lambda u$ and Dirichlet data $\lambda h$. Moreover, by the linearity of the
 mappings $(\z_{h},D\cdot)$ and $\abs{D\cdot}$, \eqref{z32} yields that
 \begin{displaymath}
   (\z_{h},D(\lambda u)=\lambda\,
   (\z_{h},D(u)=\lambda\,\abs{Du}=\abs{D(\lambda u)},
 \end{displaymath}
 showing that $\z_{h}$ and $\lambda u$ also satisfy
 \eqref{z32}. Therefore, for every $\lambda>0$, one has that
 $\z_{h}=\z_{\lambda h}$; more precisely, $z_{h}$ satisfies \eqref{z42}-\eqref{z32}
 with respect to $\lambda u$, implying that
 \begin{equation}
   \label{eq:111}
   \lambda\varphi(h)=\varphi(\lambda h)
 \end{equation}
for every $\lambda>0$. Note, if $\lambda=0$, then $u\equiv 0$ is certainly a minimizer of
 $\Phi_{0}$ over $BV(\Omega)$ and a weak solution of Dirichlet problem~\eqref{eq:21N}
 with Dirichlet data $h=0$ with corresponding vector field
 $\z_{0}\equiv 0\in \R^{d}$. Therefore, $\varphi$ also
 satisfies~\eqref{eq:111} for $\lambda=0$, completing the proof of
 homogeneity of $\varphi$.

 To see that $\varphi$ is convex, let $h_{1}$,
 $h_{2}\in L^{1}(\partial\Omega)$, and $\lambda\in (0,1)$. Then, there are weak solutions
 $u_{1}$, $u_{2}\in BV(\Omega)$ of Dirichlet problem~\eqref{eq:21N}
 with Dirichlet data $h_{1}$, $h_{2}$ and corresponding vector fields
 $\z_{h_{1}}$, $\z_{h_{2}}\in L^{\infty}(\Omega;\R^{d})$
 satisfying~\eqref{z42}-\eqref{z32} with respect to $u_{1}$ and
 $u_{2}$. Then by the homogeneity of $\varphi$, we have that
 \begin{displaymath}
   \lambda\,\varphi(h_{1})=\Phi_{\lambda
     h_{1}}(\lambda u_{1})\quad\text{ and }\quad
 (1-\lambda) \, \varphi(h_{2}) =\Phi_{(1-\lambda) h_{2}}((1-\lambda) \, u_{2})
 \end{displaymath}
 and so, by the convexity of $\Phi$, and by Theorem~\ref{thm:8},
 \begin{align*}
   &\lambda\,\varphi(h_{1})+(1-\lambda) \, \varphi(h_{2})\\
   &\qquad = \lambda\,\Phi_{\lambda h_{1}}(\lambda u_{1})+(1-\lambda)
     \, \Phi_{(1-\lambda) h_{2}}((1-\lambda) \, u_{2})\\
   &\qquad\ge \Phi_{\lambda h_{1}+(1-\lambda) h_{2}}(\lambda u_{1}
     +(1-\lambda) \, u_{2})\\
   &\qquad\ge \min_{v\in BV(\Omega)}\Phi_{\lambda h_{1}+(1-\lambda) h_{2}}(v)=\varphi(\lambda h_{1}+(1-\lambda) h_{2}).
 \end{align*}

To see that the convex, proper functional $\varphi$ given
by~\eqref{eq:107} is continuous on $L^{1}(\partial\Omega)$, it is
sufficient to show (cf., \cite[Lemma~7.1]{MR1422252}) that
for every $h\in L^{1}(\partial\Omega)$, $\varphi$ is bounded on a
neighborhood of $h$. For this, let $h\in L^{1}(\partial\Omega)$, $r>0$
and $\hat{h}$ an element of the open ball $B_{L^{1}}(h,r)$ in
$L^{1}(\partial\Omega)$ centered at $h$ of radius $r$. Then, as for
every vector field $\z_{\hat{h}}\in
L^{\infty}(\Omega,\R^{d})$ related to a weak solution $u_{\hat{h}}$ of
Dirichlet problem~\eqref{eq:21N} with Dirichlet data $\hat{h}$, one
has that $\norm{[\z_{\hat{h}},\nu]}_{\infty}\le 1$, it follows that
\begin{displaymath}
  \varphi(\hat{h})\le \norm{\hat{h}}_{1}\le r+\norm{h}_{1}.
\end{displaymath}

Finally, we show that $\varphi$ is even. For this let $h\in
L^{1}(\partial\Omega)$. Since the effective domain $D(\varphi)$ is the
whole space $L^{1}(\partial\Omega)$, we also have that $-h\in
D(\varphi)$. Moreover, let $u_{h}\in BV(\Omega)$ and $\z_{h}\in L^{\infty}(\Omega;\R^{d})$
 satisfy~\eqref{z42}-\eqref{z22} with respect to $u_{h}$ and
 $h$, and set
 \begin{displaymath}
   u_{-h}:=-u_{h}\quad\text{ and }\quad \z_{-h}:=-\z_{h}.
\end{displaymath}
Then,
 obviously, $\z_{-h}$ satisfies $\norm{\z_{-h}}_{\infty}\le 1$,
 $-\divi(\z_{-h})=0$ in $\mathcal{D}'(\Omega)$ and by the bilinearity
 of the measure $(\cdot,D\cdot)$, it follows that
 \begin{displaymath}
   (\z_{-h},Du_{-h})=(-\z_{h},D(-u_{h}))=(\z_{h},Du_{h})=\abs{Du_{h}}=\abs{D(-u_{h})}=\abs{Du_{-h}}
 \end{displaymath}
 as Radon measures. In addition, by~\eqref{z42}, one has that
 \begin{displaymath}
   [\z_{-h},\nu]=-[\z_{h},\nu]\in -\sign(h - \T(u_{h}))=\sign(-h - \T(u_{-h}))
  \end{displaymath}
  $\mathcal{H}^{d-1}$-a.e. on $\partial\Omega$. Hence, we have shown that the pair $(u_{-h},\z_{-h})$
  satisfy~\eqref{z42}-\eqref{z22}. Thus, and by the linearity of the
  weak trace $\z\mapsto [\z,\nu]$, one sees that
  \begin{displaymath}
    \varphi(-h)=\int_{\partial\Omega}[\z_{-h},\nu](-h)\,\dH^{d-1}
    = int_{\partial\Omega}[\z_{h},\nu]\,h\,\dH^{d-1}=\varphi(h).
  \end{displaymath}
  This completes the proof of this proposition.
\end{proof}

Next, we turn to the relation of the closure
\begin{displaymath}
  \overline{\Lambda}^{\mbox{}_{L^{1}\times L^{\infty}_{\sigma}}}=\Bigg\{(h,g)\in
  L^{1}\times L^{\infty}(\partial\Omega)\,\Bigg\vert\,
  \begin{array}[c]{c}
  \text{there exists }\,((h_{n},g_{n}))_{n\ge 1}\subseteq \Lambda\text{
    s.t. }\\
    \displaystyle\lim_{n\to \infty} (h_{n},g_{n})= (h,g)\text{ in }L^{1}\times L^{\infty}_{\sigma}(\partial\Omega)
  \end{array}
\Bigg\}
\end{displaymath}
of the
Dirichlet-to-Neumann operator $\Lambda$ in
$L^{1}(\partial\Omega)\times L^{\infty}_{\sigma}(\partial\Omega)$ with
the sub-differential operator $\partial_{L^{1}\times L^{\infty}(\partial\Omega)}\varphi$ in
$L^{1}\times L^{\infty}(\partial\Omega)$.

\begin{proposition}\label{prop:7}
  For the closure
  $\overline{\Lambda}^{\mbox{}_{L^{1}\times L^{\infty}_{\sigma}}}$ in
  $L^{1}\times L^{\infty}_{\sigma}(\partial\Omega)$ of
  the Dirichlet-to-Neumann operator $\Lambda$ associated with the $1$-Laplace operator,
  one has that
  \begin{displaymath}
    \overline{\Lambda}^{\mbox{}_{L^{1}\times L^{\infty}_{\sigma}}}\subseteq \partial_{L^{1}\times L^{\infty}(\partial\Omega)}\varphi,
  \end{displaymath}
  where $\partial\varphi$ denotes the sub-differential operator in
  $L^{1}(\partial\Omega)\times L^{\infty}(\partial\Omega)$ of the
  functional $\varphi : L^{1}(\partial\Omega)\to [0,\infty)$ given
  by~\eqref{eq:107}.
\end{proposition}

\begin{proof}
  We begin by taking $(h,g)\in \Lambda$ and
  $\hat{h}\in L^{1}(\partial\Omega)$. By definition of $\Lambda$ and
  since the variational problem~\eqref{eq:8} for Dirichlet data $\hat{h}$ admits a
  solution which is characterized by Proposition~\ref{prop:6}, there
  are $u_{\hat{h}}\in BV(\Omega)$ and $\z_{h}$, $\z_{\hat{h}}\in
  L^{\infty}(\Omega;\R^{d})$ such that $(u_{\hat{h}},\z_{\hat{h}})$ satisfies
  \eqref{z42}-\eqref{z32}, and, in addition, $g$ and $\z_{h}$
  satisfy~\eqref{eq:92}. Then, multiply $g$ by $(\hat{h}-h)$ and
  integrating over $\partial\Omega$. Then by~\eqref{eq:92}, the
  definition of $\varphi$, since $\norm{g}_{\infty}\le 1$,
  and by the generalized integration by parts formula~\eqref{Green0}, one sees that
  \begin{align*}
    &\int_{\partial\Omega}g(\hat{h}-h)\,\dH^{d-1}\\
    &\qquad =\int_{\partial\Omega}g\,\hat{h}\,\dH^{d-1}-\varphi(h)\\
    &\qquad
      =\int_{\partial\Omega}g\,(\hat{h}-\T(u_{\hat{h}}))\,\dH^{d-1}+\int_{\partial\Omega}g\,
      \T(u_{\hat{h}}) \,\dH^{d-1}-\varphi(h)\\
    &\qquad
      \le
      \int_{\partial\Omega}\abs{\hat{h}-\T(u_{\hat{h}})}\,\dH^{d-1}+\int_{\Omega}(\z_{\hat{h}},D
      u_{\hat{h}}) \,\dx-\varphi(h)\\
    &\qquad
      \le
      \int_{\partial\Omega}\abs{\hat{h}-\T(u_{\hat{h}})}\,\dH^{d-1}+\int_{\Omega}\abs{D
      u_{\hat{h}}}-\varphi(h)\\
    &\qquad =\varphi(\hat{h}) -\varphi(h).
  \end{align*}
  Therefore, one has that $(h,g)\in \partial_{L^{1}\times L^{\infty}(\partial\Omega)}\varphi$, showing that
  $\Lambda \subseteq \partial_{L^{1}\times
    L^{\infty}(\partial\Omega)}\varphi$.

  Next, let
  $(h,g)\in \overline{\Lambda}^{\mbox{}_{L^{1}\times
      L^{\infty}_{\sigma}}}$, $((h_{n},g_{n}))_{n\ge 1}\subseteq
  \Lambda$ such that $h_{n}\to h$ in $L^{1}(\partial\Omega)$
  and $g_{n}\to g$ in $L^{\infty}_{\sigma}(\partial\Omega)$. Then by
  the first part of this proof, we have that each $(h_{n},g_{n})\in
  \partial_{L^{1}\times L^{\infty}(\partial\Omega)}\varphi$ and hence,
  \begin{displaymath}
    \varphi(\hat{h})-\int_{\partial\Omega}g_{n}(\hat{h}-h_{n})\,\dH^{d-1}\ge \varphi(h_{n})
  \end{displaymath}
  for every $\hat{h}\in L^{1}(\partial\Omega)$ and every $n\ge
  1$. Taking the limit inferior as $n\to \infty$ on both sides of this
  inequality and using that $\varphi$ is lower
  semicontinuous in $L^{1}(\partial\Omega)$, one finds that
  \begin{displaymath}
    \varphi(\hat{h})-\int_{\partial\Omega}g(\hat{h}-h)\,\dH^{d-1}\ge \varphi(h),
  \end{displaymath}
  showing that $(h,g)\in \partial_{L^{1}\times L^{\infty}(\partial\Omega)}\varphi$ and thereby, completing the
  proof of this proposition.
\end{proof}

With the help of the functional $\varphi$, we can now show that the
Dirichlet-to-Neumann operator $\Lambda$ is closed in $L^1 \times L^1(\partial \Omega)$.

\begin{proposition}\label{prop:charact-closure-Lambda}
  The Dirichlet-to-Neumann operator $\Lambda$ is closed in
  $L^1 \times L^{\infty}_{\sigma}(\partial \Omega)$.
\end{proposition}

\begin{proof}[Proof of Proof~\ref{prop:charact-closure-Lambda}]
Let $(h,g)\in \overline{\Lambda}^{\mbox{}_{L^{1}\times L^{\infty}_{\sigma}}}$ the closure of the
Dirichlet-to-Neumann operator $\Lambda$ in $L^1 \times L^{\infty}_{\sigma}(\partial \Omega)$. Then, there is
a sequence
\begin{displaymath}
((h_{n},g_{n}))_{n\ge 1}\subseteq \Lambda\quad\text{ such that }\quad
(h_n, g_n) \to (h,g)\text{ in $L^1(\partial \Omega) \times L^{\infty}_{\sigma}(\partial
\Omega)$.}
\end{displaymath}
By definition of $\Lambda$, for every pair $(h_{n},g_{n})$, there
are $u_n \in BV(\Omega)$ and $\z_n\in L^\infty(\Omega;\R^d)$ satisfying
\begin{align}
    \label{Nz21} \norm{\z_{n}}_\infty&\le 1,\\
    \label{Nz22} -\divi(\z_{n})&=0\qquad \text{in
                            $\mathcal{D}^\prime(\Omega)$, and}\\
    \label{Nz32} (\z_{n},D u_{n})&=|D u_{n}|\qquad \text{as Radon
                                   measures,}\\
    \label{Nz52} g_{n}&=[\z_{n},\nu] \qquad\text{$\mathcal{H}^{d-1}$-a.e. on $\partial\Omega$,}
  \end{align}
and
\begin{equation}
  \label{Nz42}
  [\z_n,\nu] \in \sign\Big(h_n -\T(u_n)\Big)\qquad\text{$\mathcal{H}^{d-1}$-a.e. on $\partial\Omega$.}
\end{equation}
Now, \eqref{Nz21} yields that there is a vector field $\z_{g}\in L^\infty(\Omega;\R^d)$ satisfying
$\norm{\z_{g}}_\infty\le 1$ and, after possibly passing to a
subsequence of $((h_{n},g_{n}))_{n\ge 1}$, one has that
\begin{equation}
  \label{eq:102}
  \z_n\rightharpoonup \z_{g}\qquad\text{ weakly$\mbox{}^{\ast}$ in
$L^{\infty}(\Omega,\R^{d})$.}
\end{equation}
Therefore and by~\eqref{Nz22}, it follows that also the vector field $\z_{g}$
satisfies
\begin{equation}
  \label{eq:115}
  -\divi(\z_{g})=0\qquad\text{ in $\mathcal{D}^\prime(\Omega)$.}
\end{equation}
Thanks to~\eqref{Nz52}, \eqref{Nz21}, and since
$g_{n} \to g$ weakly$\mbox{}^{\ast}$ in $L^{\infty}(\partial\Omega)$, we can pass to a
subsequence, if necessary, to conclude that
\begin{displaymath}
  \norm{g}_{L^{\infty}(\partial\Omega)}\le \liminf_{n\to\infty}\;\norm{g_{n}}_{L^{\infty}(\partial\Omega)}\le 1.
\end{displaymath}
Now, by~\eqref{Nz22}, since $\z_{g}$ satisfies \eqref{eq:115},
and~\eqref{eq:102}, it follows from Proposition~\ref{prop:3} that
\begin{equation}
  \label{eq:67}
  \lim_{n \to \infty} \int_{ \Omega}(\z_n,Dw) = \int_{ \Omega}(\z_{g}, Dw).
\end{equation}
for every $w\in BV(\Omega)$. Thus, if $\xi \in L^1(\partial \Omega)$ and
$w \in BV(\Omega)$ such that $\T(w) = \xi$, then by the
  generalized integration by parts formula~\eqref{Green0},
  by~\eqref{Nz52}, since $\z_{g}$ satisfies \eqref{eq:115}, and by~\eqref{eq:67}, one sees that
\begin{align*}
  \int_{\partial \Omega} g\, \xi\,\dH^{d-1}
  &= \lim_{n \to \infty} \int_{\partial
    \Omega} g_n\,\xi\,\dH^{d-1}\\
  &= \lim_{n \to \infty} \int_{\partial \Omega} [\z_n, \nu]\,\xi\,\dH^{d-1}\\
  &= \lim_{n \to \infty} \int_{ \Omega}(\z_n,Dw) + \int_{ \Omega}
  \divi(\z_n)\,w\,\dx\\
  &= \int_{ \Omega}(\z_{g}, Dw)\\
  &= \int_{\partial \Omega} [\z_{g},\nu]\, \xi\,\dH^{d-1}.
\end{align*}
Since $\xi \in L^1(\partial \Omega)$ was arbitrary, we have thereby shown
that
\begin{equation}
\label{eq:116}
 g=[\z_{g},\nu]\qquad\text{$\mathcal{H}^{d-1}$-a.e. on $\partial\Omega$}
\end{equation}
and
\begin{equation}
  \label{eq:12}
  \int_{ \Omega}(\z_{g}, Dw)
  = \int_{\partial \Omega} [\z_{g},\nu]\, \T(w)\,\dH^{d-1}\qquad\text{for every $w\in BV(\Omega)$.}
\end{equation}

On the other hand, since each $u_{n}$ is a weak solution of Dirichlet
problem~\eqref{eq:21N} with boundary data $h_{n}$,
Proposition~\ref{prop:6} yields that
\begin{equation}
  \label{eq:100}
 \begin{split}
     &\int_\Omega \abs{Du_{n}} + \int_{\partial \Omega}\abs{\T(u_{n}) - h_{n}}\,\dH^{d-1}\\
     &\qquad\le \int_\Omega \abs{Dw} + \int_{\partial \Omega}\abs{\T(w) - h_{n}}\,\dH^{d-1}
   \end{split}
 \end{equation}
 for every $w \in BV(\Omega)$. Combining this estimate for some fixed
 $w \in BV(\Omega)$ together with the
 triangle inequality and the fact that $(h_{n})_{n\ge 1}$ is bounded
 in $L^{1}(\partial\Omega)$, one finds a constant $M$ such that
 \begin{displaymath}
     \int_\Omega \abs{Du_{n}} + \int_{\partial
       \Omega}\abs{\T(u_{n})}\,\dH^{d-1}\le M
     \qquad\text{for all $n\ge 1$.}
 \end{displaymath}
Therefore and by the Maz'ya inequality~\eqref{eq:7}, the
 sequence $(u_{n})_{n\ge 1}$ is boun\-ded in $BV(\Omega)$. Hence, there
 is a $u_{h}\in BV(\Omega)$ such that after possibly passing to a
 subsequence, $u_{n}\to u_{h}$ weakly$\mbox{}^{\ast}$ in $BV(\Omega)$.
 Now, let $w\in BV(\Omega)$. Then by~\eqref{eq:100},
 \begin{align*}
     &\int_\Omega \abs{Du_{n}} + \int_{\partial \Omega}\abs{\T(u_{n}) -
       h}\,\dH^{d-1}\\
     &\quad\le \int_\Omega \abs{Du_{n}} + \int_{\partial \Omega}\abs{\T(u_{n}) -
       h_{n}}\,\dH^{d-1}+\int_{\partial \Omega}\abs{h_{n}-h}\,\dH^{d-1}\\
     &\quad \le \int_\Omega \abs{Dw} + \int_{\partial
       \Omega}\abs{\T(w) - h_{n}}\,\dH^{d-1}
       +\int_{\partial \Omega}\abs{h_{n}-h}\,\dH^{d-1}
 \end{align*}
 and so, by the limit $h_{n}\to h$ in $L^{1}(\partial\Omega)$ and
 by Modica's convergence result (Proposition~\ref{prop:4}), one gets that
\begin{displaymath}
     \int_\Omega \abs{Du_{h}} + \int_{\partial \Omega}\abs{\T(u_{h})-
       h}\,\dH^{d-1}
      \le \int_\Omega \abs{Dw} + \int_{\partial
       \Omega}\abs{\T(w) - h}\,\dH^{d-1}
\end{displaymath}
for every $w\in BV(\Omega)$, showing that
$u_{h}$ is a minimizer of the relaxed functional
$\Phi_{h}$ given by~\eqref{eq:101}. Thus, by Proposition~\ref{prop:6},
$u_{h}$ is a weak solution of the Dirichlet problem~\eqref{eq:21N} with
boundary data $h$. Hence, there is a vector field $\z_{h}\in
L^\infty(\Omega;\R^d)$ satisfying~\eqref{z42}-\eqref{z32} with respect
to $u_{h}$.\medskip


Now, by~\eqref{Nz22}-\eqref{Nz42} and the generalized integration by parts
formula, one sees that
\begin{align*}
  &\int_{\partial\Omega}[\z_{n},\nu]\,h_{n}\,\dH^{d-1}
    -\int_\Omega \abs{Du_{n}}\\
  &\qquad =\int_{\partial\Omega}[\z_{n},\nu]\,h_{n}\,\dH^{d-1}
    -\int_{\partial\Omega}[\z_{n},\nu]\,\T(u_{n})\,\dH^{d-1}\\
  &\qquad
    =\int_{\partial\Omega}[\z_{n},\nu]\,(h_{n}-\T(u_{n}))\,\dH^{d-1}\\
  &\qquad
    =\int_{\partial\Omega}\abs{h_{n}-\T(u_{n})}\,\dH^{d-1},
\end{align*}
or, equivalently,
\begin{displaymath}
  \int_{\partial\Omega}[\z_{n},\nu]\,h_{n}\,\dH^{d-1}=
  \int_\Omega \abs{Du_{n}}+\int_{\partial\Omega}\abs{h_{n}-\T(u_{n})}\,\dH^{d-1}.
\end{displaymath}
Note that the left-hand side in the above equation is $\varphi(h_{n})$
for the functional $\varphi$ given by~\eqref{eq:107}. Since
$(h_n, g_n) \to (h,g)$ in
$L^1(\partial \Omega) \times L^{\infty}_{\sigma}(\partial\Omega)$, and
since by Proposition~\ref{prop:DtN-functional}, $\varphi$ is
continuous, one has that

 \begin{align*}
   \int_{\partial\Omega}g\,h\,\dH^{d-1}
   &=\lim_{n\to
     \infty}\int_{\partial\Omega}g_{n}\,h_{n}\,\dH^{d-1}\\
   &=\lim_{n\to
     \infty}\int_{\partial\Omega}[\z_{n},\nu]\,h_{n}\,\dH^{d-1}\\
  &= \lim_{n\to\infty}\varphi(h_{n})
   =\varphi(h)
   =\int_{\partial\Omega}[\z_{h},\nu]\,h\,\dH^{d-1}.
 \end{align*}
Hence, we have shown that
\begin{equation}
  \label{eq:117}
  \int_{\partial\Omega}[\z_{g},\nu]\,h\,\dH^{d-1}
  =\int_{\partial\Omega}g\,h\,\dH^{d-1}=\int_{\partial\Omega}[\z_{h},\nu]\,h\,\dH^{d-1}.
\end{equation}

Next, we intend to show that
\begin{equation}
  \label{eq:85}
  (\z_{g},Du_{h})=\abs{Du_{h}}\qquad\text{as Radon measures.}
\end{equation}
To see this, recall that by~\eqref{eq:117} and since the pair
$(\z_{h},u_{h})$ satisfy~\eqref{z42} and \eqref{z32}, one has that
\begin{align*}
  \int_{\partial\Omega}[\z_{g},\nu]\,h\,\dH^{d-1}
  &=\int_{\partial\Omega}[\z_{h},\nu]\,h\,\dH^{d-1}\\
  &= \int_{\Omega}\abs{Du_{h}}+\int_{\partial\Omega}\abs{h-\T(u_{h})}\,\dH^{d-1}.
\end{align*}
On the other hand, an integration by parts gives
\begin{align*}
  \int_{\partial\Omega}[\z_{g},\nu] h \dH^{d-1}
  &=\int_{\partial\Omega}[\z_{g},\nu]
    \T(u_{h}) \dH^{d-1}+\int_{\partial\Omega}[\z_{g},\nu]\Big(h-\T(u_{h})\Big)\,\dH^{d-1}\\
  &=\int_{\Omega}(\z_{g},Du_{h}) +\int_{\partial\Omega}[\z_{g},\nu]\Big(h-\T(u_{h})\Big)\,\dH^{d-1}
\end{align*}
Combining those two equations, one finds that
\begin{align*}
  &\int_{\Omega}(\z_{g},Du_{h})
    +\int_{\partial\Omega}[\z_{g},\nu]\Big(h-\T(u_{h})\Big)\,\dH^{d-1}\\
  &\qquad =\int_{\Omega}\abs{Du_{h}}+\int_{\partial\Omega}\abs{h-\T(u_{h})}\,\dH^{d-1}
\end{align*}
or, equivalently,
\begin{equation}
  \label{eq:19}
\begin{split}
  &\int_{\Omega}(\z_{g},Du_{h})-\int_{\Omega}\abs{Du_{h}}\\
  &\qquad =\int_{\partial\Omega}\abs{h-\T(u_{h})}-[\z_{g},\nu]\Big(h-\T(u_{h})\Big)\,\dH^{d-1}.
\end{split}
\end{equation}
Now,
\begin{equation}
\label{eq:37}
  [\z_{g},\nu]\Big(h-\T(u_{h})\Big)\le \abs{h-\T(u_{h})}\qquad\text{$\mathcal{H}^{d-1}$-a.e. on $\partial\Omega$}
\end{equation}
and by~\eqref{eq:32} and $\norm{\z_{g}}_\infty\le 1$, one has that
\begin{displaymath}
  \labs{\int_{\Omega}(\z_{g},Du_{h})}\le \norm{\z_{g}}_{\infty}\,
  \int_{\Omega}\abs{Du_{h}}\le \int_{\Omega}\abs{Du_{h}}.
\end{displaymath}
Thus at both sides in~\eqref{eq:19}, one has that
\begin{align*}
 0\ge &\int_{\Omega}(\z_{g},Du_{h})-\int_{\Omega}\abs{Du_{h}}\\
  &\qquad
    =\int_{\partial\Omega}\abs{h-\T(u_{h})}-[\z_{g},\nu]\Big(h-\T(u_{h})\Big)\,\dH^{d-1}\ge
    0,
\end{align*}
which implies that
\begin{equation}
\label{eq:42}
  \int_{\Omega}(\z_{g},Du_{h})=\int_{\Omega}\abs{Du_{h}}
\end{equation}
and
\begin{equation}
  \label{eq:128}
  \int_{\partial\Omega}\abs{h-\T(u_{h})}-[\z_{g},\nu]\Big(h-\T(u_{h})\Big)\,\dH^{d-1}=0.
\end{equation}

Since $(\z_{g},Du)$ is absolutely continuous w.r.t. $\abs{Du_{h}}$,
\eqref{eq:42} implies that~\eqref{eq:85} holds.

Further, by~\eqref{eq:37}, \eqref{eq:128} implies that
\begin{displaymath}
  [\z_{g},\nu](h-\T(u_{h}))-\abs{h-\T(u_{h})}=0
  \qquad\text{$\mathcal{H}^{d-1}$-a.e. on $\partial\Omega$.}
\end{displaymath}
Since $[\z_{h},\nu]$ and $u_{h}$ satisfy \eqref{z42}, this means that
\begin{displaymath}
  [\z_{g},\nu]=[\z_{h},\nu]\qquad\text{$\mathcal{H}^{d-1}$-a.e. on
    $\{h\neq \T(u_{h})\}$.}
\end{displaymath}
Since $\norm{[\z_{g},\nu]}_{\infty}\le 1$, we have thereby shown that
\begin{equation}
  \label{eq:129}
  [\z_{g},\nu]\in \sign(h-\T(u_{h}))\qquad\text{$\mathcal{H}^{d-1}$-a.e. on $\partial\Omega$.}
\end{equation}
Summarizing, we have shown that for every
$(h,g)\in \overline{\Lambda}^{\mbox{}_{L^{1}\times
    L^{\infty}_{\sigma}}}$, there are $u_{h}\in BV(\Omega)$ and
$\z_{g}\in L^{\infty}(\Omega;\R^d)$ satisfying
$\norm{\z_{g}}_{\infty}\le 1$, \eqref{eq:115}, \eqref{eq:116},
\eqref{eq:85}, and~\eqref{eq:129}, proving that $(h,g)\in \Lambda$.
\end{proof}

To complete the proof of Theorem~\ref{thm:main1}, it remains to show
that the Dirichlet-to-Neumann
operator $\Lambda$ associated with the $1$-Laplace operator satisfies
the \emph{range condition}
\begin{equation}
  \label{eq:rang-in-L1}
    \Rg\left(I + \lambda\,\Lambda\right) = L^1(\partial \Omega)
\end{equation}
for some (or, equivalently, for all) $\lambda>0$. But for this, we
later use that the \emph{restriction $\Lambda_{L^{2}}$ of $\Lambda$ on
  $L^{2}(\partial\Omega)\times L^{\infty}(\partial\Omega)$} is
$m$-accretive in $L^{2}(\partial\Omega)$. This property of
$\Lambda_{\vert L^2}$ and the proof of its sub-differential structure is outlined in the
following subsection.


%
%
%
%

\subsection{The Dirichlet-to-Neumann operator in $L^2$}
\label{subsec:DtN-in-L2}

In this subsection, we focus on the Dirichlet-to-Neumann
operator $\Lambda_{\vert L^2}$ in $L^2(\partial\Omega)$.

\begin{definition}\label{def:DtN-in-L2-precise}
  We define the \emph{Dirichlet-to-Neumann operator $\Lambda_{\vert L^2}$ in
      $L^{2}(\partial\Omega)$ associated with the
      $1$-Laplace operator} $\Delta_{1}$ by
    \begin{displaymath}
      \Lambda_{\vert L^2}=\Lambda\cap \Big(L^{2}(\partial\Omega)\times L^{2}(\partial\Omega)\Big);
    \end{displaymath}
    or equivalent, by the set of all pairs $(h,g)\in L^{2}(\partial\Omega)\times
    L^{2}(\partial\Omega)$ with the property that there is a weak
    solution $u\in BV(\Omega)$ of Dirichlet problem~\eqref{eq:21N}
    with Dirichlet data $h$ and there is a vector field $\z\in
    L^{\infty}(\Omega;\R^{d})$ associated with $u$
    (satisfying~\eqref{z42}-\eqref{z32}) and
    \begin{displaymath}
    g=[\z,\nu]\qquad\text{$\mathcal{H}^{d-1}$-a.e. on $\partial\Omega$.}
  \end{displaymath}
\end{definition}

\begin{remark}\label{rem:DtN-in-L2}
  Since $L^{2}(\partial\Omega)\subseteq L^{1}(\partial\Omega)$, it
  follows from Remark~\ref{rem:DtN-in-L1} that the effective domain
  $D(\Lambda_{\vert L^{2}})$ of the Dirichlet-to-Neumann operator
  $\Lambda_{\vert L^{2}}$
  associated with $\Delta_{1}$ satisfies
  \begin{displaymath}
    D(\Lambda)=L^{2}(\partial\Omega)
  \end{displaymath}
  and the operator
  \begin{equation}
    \label{eq:130}
    \Lambda_{\vert L^{2}}\subseteq L^{2}(\partial\Omega)\times \overline{B}_{L^{\infty}(\partial\Omega)}.
  \end{equation}
\end{remark}

It is clear that $\Lambda_{\vert L^2}$ is completely accretive in
$L^2(\partial\Omega)$ since $\Lambda$ admits this property in
$L^1(\partial\Omega)$. We can say more about $\Lambda_{\vert L^2}$.

\begin{proposition}
  \label{prop:5}
  The Dirichlet-to-Neumann operator $\Lambda_{\vert L^2}$ in
    $L^{2}(\partial\Omega)$ associated with the $1$-Laplace operator
  $\Delta_{1}$ is cyclically monotone.
\end{proposition}

\begin{proof}
  Let $(h_{j})_{j=0}^{n}\subseteq D(\Lambda_{\vert L^2})$ be a finite cyclic
  sequence with $h_{0}=h_{n}$ and $(g_{j})_{j=0}^{n}$ a corresponding
  sequence of elements $g_{j}\in \Lambda_{\vert L^2}h_{j}$. Then, for every
  $j=0,\dots, n$, there is a weak solution $u_{j}\in BV(\Omega)$ of
  Dirichlet problem~\eqref{eq:21N} with Dirichlet data $h_{j}$,
  (w.l.g., we may assume $u_{0}=u_{n}$), and there is a vector field
  $\z_{j}\in L^{\infty}(\Omega;\R^{d})$ associated with $u_{j}$
  (satisfying~\eqref{z42}-\eqref{z32}) and
  \begin{equation}
    \label{eq:89}
    g_{j}=[\z_{j},\nu]\qquad\text{$\mathcal{H}^{d-1}$-a.e. on $\partial\Omega$.}
  \end{equation}
  By applying the generalized integration by parts
  formula~\eqref{Green0} to $w=(u_{j}-u_{j-1})$ and the vector field
  $\z_{j}$, and by using~\eqref{z22}, gives
  \begin{displaymath}
    \int_{\Omega}(\z_{j},D(u_{j}-u_{j-1}))=\int_{\partial\Omega}[\z_{j},\nu]\Big(\T(u_{j})-\T(u_{j-1})
    \Big)\,\dH^{d-1}
\end{displaymath}
for $j=1,\dots, n$. Therefore, by~\eqref{eq:89}, the bilinearity of the pairing
$(\cdot,D\cdot)$, and since $u_{0}=u_{n}$, it follows from the last integral equation that
\begin{align*}
  \sum_{j=1}^{n}\int_{\Omega}(\z_{j},D(u_{j}-u_{j-1}))
  &    =\sum_{j=1}^{n}\left[\int_{\Omega}\abs{Du_{j}}-\int_{\Omega}(\z_{j},Du_{j-1})\right]\\
  &
    =\sum_{j=1}^{n-1}\left[\int_{\Omega}\abs{Du_{j}}-\int_{\Omega}(\z_{j+1},Du_{j})\right]\\
  &\hspace{2cm}+\int_{\Omega}\abs{Du_{n}}-\int_{\Omega}(\z_{1},Du_{0}).
\end{align*}
By~\eqref{eq:32}, \eqref{z32}, and since $u_{0}=u_{n}$, we can
conclude that the right hand-side in the last equation is
non-negative and hence, we have shown that
\begin{displaymath}
  \sum_{j=1}^{n}\int_{\partial\Omega} g_{j}\,
  \Big(\T(u_{j})-\T(u_{j-1})\Big) \dH^{d-1}\ge 0.
\end{displaymath}
By using now this inequality, one sees that
\begin{equation}
  \label{eq:91}
\begin{split}
  &\sum_{j=1}^{n}\int_{\partial\Omega} g_{j}\, \Big(h_{j}-h_{j-1}\Big) \dH^{d-1}\\
  &\qquad \ge \int_{\partial\Omega} \sum_{j=1}^{n} g_{j}\,
    \Big((h_{j}-\T(u_{j}))-(h_{j-1}-\T(u_{j-1}))\Big) \dH^{d-1}.
\end{split}
\end{equation}
By~\eqref{eq:89}, since $u_{j}$ satisfies the Dirichlet boundary
condition~\eqref{eq:5} in the weak sense~\eqref{z42} with $h=h_{j}$,
and since $h_{0}=h_{n}$ and $u_{0}=u_{n}$, one finds that
\begin{align*}
   &\sum_{j=1}^{n} g_{j}\,
     \Big((h_{j}-\T(u_{j}))-(h_{j-1}-\T(u_{j-1}))\Big)\\
  &\qquad =\sum_{j=1}^{n} g_{j}\,
    (h_{j}-\T(u_{j}))-\sum_{j=1}^{n} g_{j}\,(h_{j-1}-\T(u_{j-1}))\\
  &\qquad =\sum_{j=1}^{n} g_{j}\,
    (h_{j}-\T(u_{j}))-\sum_{j=0}^{n-1} g_{j+1}\,(h_{j}-\T(u_{j}))\\
  &\qquad = \sum_{j=1}^{n-1}\abs{h_{j}-\T(u_{j})} -
    [\z_{j+1},\nu](h_{j}-\T(u_{j}))\\
  &\hspace{3cm} + \abs{h_{n}-\T(u_{n})} -[\z_{1},\nu](h_{0}-\T(u_{0}))\ge 0
\end{align*}
for $\mathcal{H}^{d-1}$-a.e. on $\partial\Omega$. Applying this
to~\eqref{eq:91} yields that
\begin{displaymath}
  \sum_{j=1}^{n}\int_{\partial\Omega} g_{j}\, \Big(h_{j}-h_{j-1}\Big)
  \dH^{d-1}\ge 0,
\end{displaymath}
and since the cyclic sequence $(h_{j})_{j=0}^{n}\subseteq
D(\Lambda_{\vert L^2})$ was arbitrary, we have thereby shown that
$\Lambda_{\vert L^{2}}$ is
cyclically monotone.
\end{proof}

\begin{proposition}\label{prop:DtN-range-cond-in-L2}
  The Dirichlet-to-Neumann operator $\Lambda_{\vert L^{2}}$ associated with
  the $1$-Laplace operator satisfies the range
  condition~\eqref{eq:29} for $X=L^{2}(\partial\Omega)$.
\end{proposition}

\begin{proof}
  Let $g \in L^{2}(\partial\Omega)$ and $\lambda>0$. Then, our aim is
  to find a  boundary function $h\in L^2(\partial\Omega)$ such that the inclusion
  \begin{equation}
    \label{eq:93}
    h+\lambda \Lambda_{\vert L^2}h \ni g
  \end{equation}
  holds. By the definition of $\Lambda_{\vert L^2}$,
  inclusion~\eqref{eq:93} is equivalent to the fact that there is a
  vector field $\z\in L^{\infty}(\Omega;\R^{d})$ and a
  weak solution $u\in BV(\Omega)$ of Dirichlet problem~\eqref{eq:21N}
  with Dirichlet data $h$ related through~\eqref{z42}-\eqref{z32} and the weak trace $[\z,\nu]$
  of the normal component of $\z$ is \emph{uniquely} given by
  \begin{displaymath}
    \dfrac{g-h}{\lambda}=[\z,\nu]\qquad\text{$\mathcal{H}^{d-1}$-a.e. on $\partial\Omega$.}
  \end{displaymath}
  Since $\norm{[\z,\nu]}_{\infty}\le 1$, it is natural to impose on the vector field $\z$ the
  condition
  \begin{equation}
    \label{eq:113}
    [\z,\nu]=T_{1}\left(\dfrac{g-h}{\lambda}\right)\qquad\text{$\mathcal{H}^{d-1}$-a.e. on $\partial\Omega$,}
  \end{equation}
  where $T_{1}$ denotes the truncator introduced in
  Section~\ref{sec:Robin}. Thus, if we find a boundary function $h$
  such that there is a \emph{weak solution} $u$ to the elliptic problem
  \begin{equation}\label{eq:112}
  \begin{cases}
    \quad-\Delta_{1}u=0 &\qquad\text{in $\Omega$,}\\
    \quad\phantom{-\Delta_{1}}u=h &\qquad\text{on $\partial\Omega$,}\\
    \dfrac{Du}{\abs{Du}}\cdot\nu=T_{1}\left(\dfrac{g-h}{\lambda}\right) &\qquad\text{on $\partial\Omega$,}
   \end{cases}
 \end{equation}
 then $h$ is a solution to the inclusion~\eqref{eq:93}.

 \begin{definition}
   For given $g \in L^{2}(\partial\Omega)$ and $\lambda>0$, we call a
   function $u\in BV(\Omega)$ a \emph{weak solution} of boundary
   problem~\eqref{eq:112} if there is a vector field $\z\in
   L^{\infty}(\Omega;\R^{d})$ satisfying \eqref{z21}-\eqref{z32} and
   the weak trace $[\z,\nu]$ satisfies \eqref{eq:113} and
   \begin{equation}
     \label{eq:114}
    \frac{g-h}{\lambda} \in \sign(h - \T(u))\qquad\text{$\mathcal{H}^{d-1}$-a.e. on $\partial\Omega$.}
  \end{equation}
 \end{definition}

  According to Theorem~\ref{thm:6}, there is a weak solution $u$
  of the Robin-type boundary-value problem~\eqref{eq:63} with $\alpha=\lambda$; that is,
  $u\in BV(\Omega)$ with trace $\T(u)\in L^{2}(\partial\Omega)$, and
  there is a vector field $\z\in L^{\infty}(\Omega;\R^{d})$
  satisfying~\eqref{z21}-\eqref{z32}, and
  \begin{displaymath}
     [\z,\nu]=T_1(g-\lambda u)\qquad\text{$\mathcal{H}^{d-1}$-a.e. on $\partial\Omega$.}
  \end{displaymath}
  Now, according to  Proposition~\ref{prop:2}, $u$ is also a weak solution of Dirichlet
  problem~\eqref{eq:21N} for Dirichlet data
  \begin{displaymath}
    h:=g-\lambda \,[\z,\nu].
  \end{displaymath}
  Since for this choice of $h$, one trivially has that
  $\dfrac{g-h}{\lambda}=[\z,\nu]$, one easily verifies that $u$, $h$
  and $\z$ satisfy \eqref{eq:113} and~\eqref{eq:114}. Moreover, since
  $h\in L^{2}(\partial\Omega)$, we have thereby shown that there is a
  $h$ satisfying the inclusion~\eqref{eq:93}, completing the proof
  of Proposition~\ref{prop:DtN-range-cond-in-L2}.
\end{proof}

We conclude this section with the following characterization of the
Dirich\-let-to-Neumann operator $\Lambda_{\vert L^{2}}$ on
$L^{2}(\partial\Omega)$.

\begin{proposition}\label{prop:sub-diff-L2}
 The Dirichlet-to-Neumann operator $\Lambda_{\vert L^{2}}$ in
  $L^{2}(\partial\Omega)$ can be characterized as the sub-differential operator
  $\partial_{L^2(\partial\Omega)}\varphi_{\vert L^{2}}$ in $L^2(\partial\Omega)$;
  that is,
  \begin{displaymath}
    \Lambda_{\vert L^2}=\partial_{L^2(\partial\Omega)}\varphi_{\vert L^{2}}
  \end{displaymath}
  where $\varphi_{\vert L^{2}}$ denotes the restriction on
  $L^2(\partial\Omega)$ of the functional $\varphi$ given by~\eqref{eq:107}.
\end{proposition}

We prove Proposition~\ref{prop:sub-diff-L2} in two different ways.

\begin{proof}[$1^{\textrm{st}}$ Proof of Proposition~\ref{prop:sub-diff-L2}]
  By Proposition~\ref{prop:5} and
  Proposition~\ref{prop:DtN-range-cond-in-L2}, the
  Dirichlet-to-Neumann operator $\Lambda_{L^{2}}$ is a maximal
  cyclically monotone operator in $L^{2}(\partial\Omega)$. Moreover,
  by Proposition~\ref{prop:Lambda-homogeneity} and since
  $\Lambda_{L^{2}}\subseteq \Lambda$, we have that $\Lambda_{L^{2}}$
  is homogeneous of order zero. Therefore by
  Theorem~\ref{thm:cycl-monotone-homogeneous}, there is a unique
  proper, convex, lower semicontinuous functional $\phi$ on
  $L^{2}(\partial\Omega)$, which is homogeneous of order one
  satisfying $\Lambda_{L^{2}}=\partial_{L^{2}(\partial\Omega)}\phi$. Since $\phi$ is
  homogeneous of order one, it follows from
  Theorem~\ref{thm:cycl-monotone-homogeneous} that $\phi$ satisfies
  \begin{displaymath}
    \phi(h)=\int_{\partial\Omega}g\,h\, \dH^{d-1}\qquad\text{for every
      $(h,g)\in \Lambda_{L^{2}}$.}
  \end{displaymath}
  By definition of $\Lambda_{L^{2}}$, for every $(h,g)\in
  \Lambda_{L^{2}}$ there is a vector field $\z_{h}\in
  L^{\infty}(\Omega;\R^{d})$ and $u_{h}\in BV(\Omega)$ satisfying
  \eqref{z42}-\eqref{z32}, and $g=[\z,\nu]$. From this, we can
  conclude that
  \begin{displaymath}
    \phi(h)=\int_{\partial\Omega}[\z,\nu] \,h\, \dH^{d-1}=\varphi(h)
  \end{displaymath}
  for every $h\in L^{2}(\partial\Omega)$, which identifies the
  functional $\phi$ with $\varphi$.
\end{proof}

The argument of our second proof of Proposition~\ref{prop:sub-diff-L2} is a bit
shorter.

 \begin{proof}[$2^{\textrm{nd}}$ Proof of Proposition~\ref{prop:sub-diff-L2}]
   On the other hand, the restriction $\varphi_{\vert L^2}$ of the
   functional $\varphi$ given by~\eqref{eq:107} on
   $L^{2}(\partial\Omega)$ is by
   Proposition~\ref{prop:DtN-functional}, convex, proper, lower
   semicontinuous on $L^{2}(\partial\Omega)$ and homogeneous of order
   one. Moreover, by following the same argument as in the proof of
   Proposition~\ref{prop:7}, one easily sees that
   $\Lambda_{L^{2}}\subseteq \partial_{L^2(\partial\Omega)}\varphi_{\vert L^2}$, which
   means that $\partial_{L^2(\partial\Omega)}\varphi_{\vert L^2}$ is a monotone
   extension of $\Lambda_{L^{2}}$. But since $\Lambda_{L^{2}}$ is
   maximal monotone, this is only possible if
   $\Lambda_{L^{2}}=\partial_{L^2(\partial\Omega)}\varphi_{\vert L^2}$
   (see~\cite[Proposition~2.2]{MR0348562}), which proves the claim of
   Proposition~\ref{prop:sub-diff-L2}.
\end{proof}

%
%
%
%

\subsection{The Dirichlet-to-Neumann operator in $L^1$ (continued)}
\label{subsec:DtN-in-L1-continued}

With this in mind, we can now complete the proof of
Theorem~\ref{thm:main1}.

\begin{proof}[Proof of Theorem~\ref{thm:main1}.]
  We only show that $\Lambda$ satisfies the range
  condition~\eqref{eq:rang-in-L1} since then the
  characterization~\eqref{eq:104} follows from
  Proposition~\ref{prop:characterization-of-A-L0} and the other
  statements of this theorem were proved in the beforegoing
  propositions. By Proposition~\ref{prop:DtN-range-cond-in-L2}, the
  restriction $\Lambda_{\vert L^2}$ of $\Lambda$ on
  $L^2(\partial\Omega)$ satisfies the range condition~\eqref{eq:29}
  for $X=L^2(\partial\Omega)$. Thus and since
  $\Lambda_{\vert L^2}\subseteq \Lambda$, we have that $\Lambda$
  satisfies the range condition~\eqref{eq:93} for every
  $g\in L^2(\partial\Omega)$. Now, let $g\in L^1(\partial\Omega)$ and
  choose a sequence $(g_{n})_{n\ge 1}$ in $L^{2}(\partial\Omega)$ such
  that $g_{n}\to g$ in $L^{1}(\partial\Omega)$. Then, for every
  $n\ge 1$, there is an $h_{n}\in L^{2}(\partial\Omega)$
  satisfying~\eqref{eq:93} with right hand-side $g_{n}$. By
  Proposition~\ref{prop:Lambda-completely-accretive},
  $(h_{n})_{n\ge 1}$ is a Cauchy sequence in
  $L^{1}(\partial\Omega)$. Hence, there is an
  $h\in L^{1}(\partial\Omega)$ such that $h_{n}\to h$ in
  $L^{1}(\partial\Omega)$. Now,
  \begin{displaymath}
    \Lambda_{\vert L^2} h_{n}\ni \frac{g_{n}-h_{n}}{\lambda}\to
    \frac{g-h}{\lambda}\qquad\text{in $L^{1}(\partial\Omega)$ as $n\to
      \infty$.}
  \end{displaymath}
  Note, that the sequence $((g_{n}-h_{n})/\lambda)_{n\ge 1}$ is also
  bounded in $L^{\infty}(\partial\Omega)$. Thus, after passing to a
  subsequence, we also have that $(g_{n}-h_{n})/\lambda\to (g-h)/\lambda$
  weakly$\mbox{}^{\ast}$ in $L^{\infty}(\partial\Omega)$. 
Since by Proposition~\ref{prop:charact-closure-Lambda}, $\Lambda$ is
closed in $L^{1}\times L^{\infty}_{\sigma}(\partial\Omega)$, we have
thereby shown that 
\begin{displaymath}
  \Lambda h=\frac{g-h}{\lambda},
\end{displaymath}
which is equivalent to the range condition~\eqref{eq:29} for
$X=L^1(\partial\Omega)$. This completes the proof of
this theorem.
\end{proof}

Next, we outline the proof of Theorem~\ref{thm:main2}.

\begin{proof}[Proof of Theorem~\ref{thm:main2}]
  By Proposition~\ref{prop:sub-diff-L2} and the Hilbert space theory
  on maximal monotone operators (see~\cite{MR0348562}), for every
  $h_{0}\in L^{2}(\partial\Omega)$ and $g\in L^{2}(0,T; L^{2}(\partial\Omega))$, there is a unique strong solution
  \begin{displaymath}
    h\in W^{1,2}([\delta,T);L^{2}(\partial\Omega))\cap
    C([0,\infty);L^{2}(\partial\Omega)),\; \delta\in (0,T),
  \end{displaymath}
  of Cauchy problem (in $L^{2}(\partial\Omega)$)
  \begin{equation}
    \label{eq:6}
    \begin{cases}
    \dfrac{\td h}{\dt}(t)+\Lambda h(t)+F(h(t))\ni g(t)&
    \quad\text{ for $t\in(0,T)$,}\\
    \hspace{3,4cm} h(0)=h_{0}.& \quad\text{on $\partial\Omega$.}
  \end{cases}
  \end{equation}
  Multiplying the differential inclusion
  \begin{displaymath}
    \dfrac{\td h}{\dt}(t)+
    \Lambda h(t)+F(h(t))\ni g(t)
  \end{displaymath}
  by $\dfrac{\td h}{\dt}(t)$ with respect to the $L^{2}$-inner
  product, and subsequently applying the chain
  rule~\cite[Lemme~3.3]{MR0348562}, gives
  \begin{equation}
        \label{eq:34}
    \lnorm{\dfrac{\td h}{\dt}(t)}_{2}^{2}+\dfrac{\td
    }{\dt}\E(h(t))=\Big(g(t), \dfrac{\td h}{\dt}(t)\Big)_{L^{2}(\partial\Omega)}\qquad\text{for a.e. $t>0$.}
  \end{equation}
  But since $\E$ is defined for all $h\in L^{2}(\partial\Omega)$, we can
  integrate the latter equation over the whole interval $(0,t)$ for
  any $t\in (0,T)$. Then, we find
  \begin{displaymath}
    \int_{0}^{t}\lnorm{\dfrac{\td
        h}{\ds}(s)}_{2}^{2}\,\ds+\E(h(t))=\E(h_{0})+\int_{0}^{t}\Big(g(s),
    \dfrac{\td h}{\ds}(s)\Big)_{L^{2}(\partial\Omega)}\ds.
  \end{displaymath}
  Now, applying Young's inequality to compensate the term $\dfrac{\td
      h}{\ds}(s)$ on the right-hand side, we arrive to the global
    estimate~\eqref{eq:1}. This proves
    statement~\eqref{thm:main2-claim2} of
    Theorem~\ref{thm:main2} and that  the global
    inequality~\eqref{eq:1} holds for initial data $h_{0}\in L^{2}(\partial\Omega)$. Next, let $g\in L^{2}(0,T;
    L^{2}(\partial\Omega))$, $h_{0}\in L^{1}(\partial\Omega)$, and
    $(h_{0,n})_{n\ge 1}$ a sequence in $L^{2}(\partial\Omega)$
    converging to $h_{0}$ in $L^{1}(\partial\Omega)$. Since
    $\partial\Omega$ has finite measure, each strong
    solutions $h_{n}$ of Cauchy problem (in $L^{2}(\partial\Omega)$)
    \begin{equation}
      \label{eq:3}
     \begin{cases}
     \dfrac{\td h_{n}}{\dt}(t)+\Lambda h_{n}(t)+F(h_{n}(t))\ni g(t)&
     \quad\text{for $t\in(0,T)$,}\\
     \hspace{3.78cm} h_{n}(0)=h_{0,n}.& \quad\text{on $\partial\Omega$.}
    \end{cases}
    \end{equation}
    is also a strong solution in $L^{1}(\partial\Omega)$
    of~\eqref{eq:3}. Moreover, by Corollary~\ref{coro:comparision}, we
    have that
    \begin{equation}
      \label{eq:15}
      \lim_{n\to \infty}h_{n}=h\qquad\text{ in $C([0,T];L^{1}(\partial\Omega))$}
    \end{equation}
    and $h$ is the unique mild solution of Cauchy problem~\eqref{eq:6} in
    $L^{1}(\partial\Omega)$. Applying $h_{n}$ to the global
    inequality~\eqref{eq:1}, by using the fact that the functional $\E$ is continuous
    on $L^{1}(\partial\Omega)$, one finds that $(\td h_{n}/\dt)_{n\ge 1}$ is bounded in
    $L^{2}(0,T;L^{2}(\partial\Omega))$. Hence, there is $\chi\in
    L^{2}(0,T;L^{2}(\partial\Omega))$ and a subsequence of
    $(h_{n})_{n\ge 1}$, which, for simplicity, we denote again by
    $(h_{n})_{n\ge 1}$, such that
    \begin{equation}
      \label{eq:13}
      \lim_{n\to \infty}\dfrac{\td
      h_{n}}{\dt}=\chi\qquad\text{weakly in
    $L^{2}(0,T;L^{2}(\partial\Omega))$.}
   \end{equation}
   Let $\xi\in C^{\infty}_{c}(0,T)$ and $v\in
   L^{2}(\partial\Omega)$. Since $\dfrac{\td
      h_{n}}{\dt}$ is the weak derivative of $h_{n}$ in
    $L^{2}(0,T;L^{2}(\partial\Omega))$, one has that
    \begin{displaymath}
      \int_{0}^{T}(\dfrac{\td
        h_{n}}{\dt},v)_{L^{2}(\partial\Omega))}\xi(t)\,\dt=-\int_{0}^{T}(h_{n},v)_{L^{2}(\partial\Omega))}
      \dfrac{\td }{\dt}\xi(t)\,\dt.
    \end{displaymath}
    By~\eqref{eq:15} and~\eqref{eq:13}, sending $n\to \infty$ in the
    last equation gives that
     \begin{displaymath}
      \int_{0}^{T}(\chi(t),v)_{L^{2}(\partial\Omega))}\xi(t)\,\dt=-\int_{0}^{T}(h,v)_{L^{2}(\partial\Omega))}
      \dfrac{\td }{\dt}\xi(t)\,\dt.
    \end{displaymath}
    Since $\xi\in C^{\infty}_{c}(0,T)$ and
    $v\in L^{2}(\partial\Omega)$ were arbitrary, this proves that
    $\chi$ is the weak derivative of $h$ in
    $L^{2}(0,T;L^{2}(\partial\Omega))$.  Coming back to
    inequality~\eqref{eq:1} applied to $h_{n}$, if one takes the limit
    inferior in this inequality, and uses that $\E$ is continuous on
    $L^{1}(\partial\Omega)$, then one sees that~\eqref{eq:1} also
    holds for initial data $h_{0}\in L^{1}(\partial\Omega)$,
    completing the proof of statement~\eqref{thm:main2-claim1} of this
    theorem. Statement~\eqref{thm:main2-claim2} follows
    from~\cite{MR0348562} and inequality~\eqref{eq:1}. Statement~\eqref{thm:main2-claim3} of
    Theorem~\ref{thm:main2} follows immediately from~\cite{MR4200826},
    which completes the proof of Theorem~\ref{thm:main2}.
\end{proof}

%
%
%
%

\subsection{Long-time Stability}
\label{subsec:long-time-stability}

In this section, we give the proof of Theorem~\ref{thm:stability} on
the long-time stability of the semigroup $\{e^{-t
  (\Lambda_{\vert L^{q}}+F)}\}_{t\ge 0}$ generated by
$-(\Lambda_{\vert L^{q}(\partial\Omega)}+F)$ on $L^{q}(\partial\Omega)$.\medskip

We begin by the following proposition.

\begin{proposition}
  \label{prop:energy-decreasing}
  Let $F$ be given by~\eqref{eq:28} with $f$ satisfying~\eqref{eq:27},
  and $\varphi$ be the functional given by~\eqref{eq:107}.  Then the
  following statements hold.
  \begin{enumerate}
  \item \label{prop:energy-decreasing-claim1} For every $h_{0}\in L^{1}(\partial\Omega)$, the energy
    functional $\E : L^{1}(\partial\Omega)\to \R$ defined by
    \begin{displaymath}
      \E(h):=\varphi(h)+\int_{0}^{h}
      f(\cdot,r)\,\dr,\qquad h\in L^{1}(\partial\Omega),
    \end{displaymath}
    decreases monotonically along the trajectory
    $\{e^{-t (\Lambda+F)}h_{0}\,\vert\,t\ge 0\}$;
    \item \label{prop:energy-decreasing-claim2} If $F\equiv 0$, then one has that
    \begin{equation}
      \label{eq:26}
      \langle \dfrac{\td }{\dt}e^{-t \Lambda}h_{0},
      e^{-t \Lambda}h_{0} \rangle_{L^{\infty},L^{1}}=-\varphi(e^{-t \Lambda}h_{0})
    \end{equation}
    for a.e. $t>0$ and every $h_{0}\in L^{1}(\partial\Omega)$.
    \item \label{prop:energy-decreasing-claim3} If $F\equiv 0$, then
      for every positive $h_{0}\in L^{1}(\partial\Omega)$, $e^{-t
        \Lambda}h_{0}\in L^{\infty}(\partial\Omega)$, one has
      that
      \begin{equation}
        \label{eq:43}
        \varphi(e^{-t \Lambda}h_{0})\le -\frac{1}{t}\norm{e^{-t
            \Lambda}h_{0}}^{2}_{2}\le 0
      \end{equation}
  \end{enumerate}
\end{proposition}

\begin{proof}
  By taking $g\equiv 0$ in~\eqref{eq:34}, and subsequently integrating
  over $(s,t)$ for any $0\le s\le t$, one finds
  \begin{displaymath}
    \E(e^{-t (\Lambda+F)}h_{0})\le \E(e^{-s (\Lambda+F)}h_{0}),
  \end{displaymath}
  showing that $\E$ is decreasing along $\{e^{-t
    (\Lambda+F)}h_{0}\,\vert\,t\ge 0\}$ provided the initial data
  $h_{0}\in L^{2}(\partial\Omega)$. But for given $h_{0}\in
  L^{1}(\partial\Omega)$, there is a sequence $(h_{0,n})_{n\ge 1}$ in
  $L^{2}(\partial\Omega)$ converging to $h$ in
  $L^{1}(\partial\Omega)$. By Corollary~\ref{coro:comparision}, $e^{-t
    (\Lambda+F)}h_{0,n}\to e^{-t (\Lambda+F)}h_{0}$ in
  $C([0,T];L^{1}(\partial\Omega))$ for every $T>0$. Thus and by the
  continuity of $\E$ on $L^{1}(\partial\Omega)$, for given $0\le s\le
  t$, we can send $n\to \infty$ in
  \begin{displaymath}
    \E(e^{-t (\Lambda+F)}h_{0,n})\le \E(e^{-s (\Lambda+F)}h_{0,n})
  \end{displaymath}
   and find that $\E$ is decreasing along $\{e^{-t
    (\Lambda+F)}h_{0}\,\vert\,t\ge 0\}$ for any initial data $h_{0}\in
  L^{1}(\partial\Omega)$. It remains to show that~\eqref{eq:26}
  holds. For this, we note that by Theorem~\ref{thm:main2}, for every
  $h_{0}\in L^{1}(\partial\Omega)$, $h(t):=e^{-t (\Lambda+F)}h_{0}$ is
  a strong solution of the Cauchy problem (in $L^{1}(\partial\Omega)$)
  \begin{displaymath}
    \begin{cases}
    \dfrac{\td h}{\dt}(t)+\Lambda h(t)\ni 0&
    \quad\text{ for $t\in(0,T)$,}\\
    \hspace{1.4cm} h(0)=h_{0}.& \quad\text{on $\partial\Omega$.}
  \end{cases}
  \end{displaymath}
  Hence, multiplying by $h(t)$ and using that $\varphi$ is homogeneous
  of order one, yields~\eqref{eq:26}.

\end{proof}

For the rest of this section, we focus on the case
$F\equiv 0$. Then, we have the following.

\begin{proposition}
  \label{prop:conservation-of-mass}
  The semigroup $\{e^{-t \Lambda}\}_{t\ge 0}$ generated
  by the negative Dirichlet-to-Neumann operator
$-\Lambda$ on $L^{1}(\partial\Omega)$
conserves mass; in other words, one has that
\begin{equation}
  \label{eq:38}
  \int_{\partial\Omega}h_{0}\,\dH^{d-1}=\int_{\partial\Omega}e^{-t \Lambda} h_{0}\,\dH^{d-1}
\end{equation}
for all $t\ge 0$ and $h_{0}\in L^{1}(\partial\Omega)$.
\end{proposition}

\begin{proof}
  Recall that by Theorem~\ref{thm:main2}, for every
  $h_{0}\in L^{1}(\partial\Omega)$, $h(t):=e^{-t \Lambda}h_{0}$ is
  a strong solution of the Cauchy problem (in $L^{1}(\partial\Omega)$)
  \begin{equation}
    \label{eq:35}
    \begin{cases}
    \dfrac{\td h}{\dt}(t)+\Lambda h(t)\ni 0&
    \quad\text{ for $t\in(0,T)$,}\\
    \hspace{1.7cm} h(0)=h_{0}.& \quad\text{on $\partial\Omega$.}
  \end{cases}
 \end{equation}
  Hence, for a.e. $t>0$, there is a weak solution $u_{h(t)}\in
  BV(\Omega)$ of Dirichlet problem~\eqref{eq:21N} and a vector field
  $\z_{h(t)}\in L^{\infty}(\Omega;\R^{d})$
  satisfying~\eqref{z42}-\eqref{z32} with boundary data $h(t)$, and
  the generalized co-normal derivative
  \begin{equation}
    \label{eq:36}
    [\z_{h(t)},\nu]=-\dfrac{\td
      h}{\dt}(t)\qquad\text{$\mathcal{H}^{d-1}$-a.e. on $\partial\Omega$.}
  \end{equation}
  Let $\mathds{1}_{\overline{\Omega}}$ denote the constant $1$
  function on $\overline{\Omega}$. Multiplying~\eqref{eq:36} by
  $\T(\mathds{1}_{\overline{\Omega}})=\mathds{1}_{\partial\Omega}$
  with respect to the $L^{2}$-inner product and then, integrating by
  parts (Proposition~\ref{prop:ibp}) yields that
  \begin{align*}
    -\dfrac{\td }{\dt}\int_{\partial\Omega}h(t)\,
    \mathds{1}\,\dH^{d-1}
   &= -\int_{\partial\Omega}\dfrac{\td h}{\dt}(t)\,
    \mathds{1}\,\dH^{d-1}\\
    &=\int_{\partial\Omega}[\z_{h(t)},\nu]\, \mathds{1}\,\dH^{d-1}\\
    &=\int_{\Omega}(\z_{h(t)},D \mathds{1})=0.
  \end{align*}
  Hence, integrating this equation over $(0,t)$ for any $t>0$, shows
  that~\eqref{eq:38} holds.
\end{proof}

Next, we establish the long-time convergence in $L^{q}(\partial\Omega)$ of the
semigroup $\{e^{-t \Lambda}\}_{t\ge 0}$.

\begin{proposition}
  \label{propo:convergence}
  Let $1\le q<\infty$, $\varphi$ given by \eqref{eq:107}, and $h_{0}\in
  L^{q}(\partial\Omega)$. Then, the following statements hold.
  \begin{enumerate}
  \item \label{propo:convergence-claim1} One has that
    \begin{equation}
      \label{eq:39}
      \lim_{t\to \infty}e^{-t  \Lambda}h_{0}=\overline{h_{0}}:=\tfrac{1}{\mathcal{H}^{d-1}(\partial\Omega)}
      \int_{\partial\Omega}e^{-t \Lambda} h_{0}\,\dH^{d-1}
      \qquad\text{in $L^{q}(\partial\Omega)$;}
    \end{equation}
   \item \label{propo:convergence-claim2} One has that
    \begin{equation}
      \label{eq:40}
      \lim_{t\to \infty}\varphi(e^{-t  \Lambda}h_{0})=\varphi(\overline{h_{0}})=0;
    \end{equation}
    \item \label{propo:convergence-claim5} (Entropy-Transport inequalities)
    There is a $C>0$ such that
    \begin{displaymath}
      \norm{e^{-t  \Lambda}h_{0}-\overline{h_{0}}}_{1}\le C\,
      \varphi(e^{-t  \Lambda}h_{0})\qquad\text{for all $t>0$;}
    \end{displaymath}
    Moreover, for every $1<q<r\le \infty$ and $h_{0}\in
    L^{r}(\partial\Omega)$, one has that
    \begin{displaymath}
      \norm{e^{-t  \Lambda}h_{0}-\overline{h_{0}}}_{q}
      \le \norm{h_{0}-\overline{h_{0}}}_{r}^{\frac{(q-1)r}{q(r-1)}}C\,\left(\varphi(e^{-t  \Lambda}h_{0})\right)^{\frac{r-q}{q(r-1)}}
    \end{displaymath}

    \item \label{propo:convergence-claim3} For every $h_{0}\in
      L^{2}(\partial\Omega)$, one has that
      \begin{equation}
        \label{eq:50}
      \varphi(e^{-t  \Lambda}h_{0})\le
      2\dfrac{\norm{h_{0}}_{2}^{2}}{t}
      \qquad\text{for all $t>0$.}
    \end{equation}
  \end{enumerate}
\end{proposition}

\begin{proof}
  We first establish~\eqref{eq:39} for
  $h_{0}\in L^{2}(\partial\Omega)$. Since the functional
  $\varphi$ given by~\eqref{eq:107} is even, and since the
  Dirichlet-to-Neumann operator $\Lambda_{\vert L^2}$ is the
  sub-differential operator of the restriction
  $\varphi_{\vert L^{2}}$ of $\varphi$ on $L^{2}(\partial\Omega)$, the
  limit~\eqref{eq:39} follows from a classic result due to
  Bruck~\cite[Theorem~5]{MR377609} in the Hilbert space
  theory. Moreover, by the continuity of $\varphi_{\vert L^{2}}$ on
  $L^{2}(\partial\Omega)$, it follows that~\eqref{eq:40} holds.

  Next, let $h_{0}\in L^{q}(\partial\Omega)$ for a given
  $1\le q\le \infty$.  One can always construct a sequence
  $(h_{0,n})_{n\ge 1}$ in
  $L^{\infty}(\partial\Omega)$ such that
  $h_{0,n}\to h_{0}$ in $L^{q}(\partial\Omega)$ and by the continuous
  embedding from $L^{q}(\partial\Omega)$ into $L^{1}(\partial\Omega)$,
  one also has that the mean-values
  $\overline{h_{0,n}}\to \overline{h_{0}}$ in $\R$ as $n\to
  \infty$. If $q=\infty$, then one simply choose the sequence
  $(h_{0,n})_{n\ge 1}$ given by $h_{0,n}\equiv h_{0}$ for all $n\ge
  1$. Then, for given $\varepsilon>0$, there is a
  $n_{0}=n_{0}(\varepsilon)\in \N$ large enough such that
  \begin{displaymath}
    \norm{h_{0,n}-h_{0}}_{q}<\dfrac{\varepsilon}{3}\quad\text{ and
    }\quad
    \norm{\overline{h_{0,\varepsilon}}-\overline{h_{0}}}_{q}<\dfrac{\varepsilon}{3}.
  \end{displaymath}
  Since each $e^{-t  \Lambda}$ is a contractive in $L^{q}(\partial\Omega)$,
  one has that
   \begin{align*}
     \norm{e^{-t  \Lambda}h_{0}-\overline{h_{0}}}_{q}
     &\le \norm{e^{-t  \Lambda}h_{0}-e^{-t\Lambda}h_{0,n_{0}}}_{q}
       +\norm{e^{-t
       \Lambda}h_{0, n_{0}}-\overline{h_{0, n_{0}}}}_{q}\\
     &\qquad +
       \norm{\overline{h_{0,n_{0}}}-\overline{h_{0}}}_{q}\\
     &\le \norm{h_{0}-h_{0, n_{0}}}_{q}
       +\norm{e^{-t
       \Lambda}h_{0, n_{0}}-\overline{h_{0,n_{0}}}}_{q}\\
     &\qquad +
       \norm{\overline{h_{0,n_{0}}}-\overline{h_{0}}}_{q}\\
     &\le 2\dfrac{\varepsilon}{3}
       +\norm{e^{-t \Lambda}h_{0, n_{0}}-\overline{h_{0,n_{0}}}}_{q}.
   \end{align*}
  Thus, in order to prove~\eqref{eq:39} in $L^{q}(\partial\Omega)$ for
  general $h_{0}\in L^{q}(\partial\Omega)$, it is sufficient to
  establish~\eqref{eq:39} for
  $h_{0}\in L^{\infty}(\partial\Omega)$.  So, let
  $h_{0}\in L^{\infty}(\partial\Omega)$. Since $h_{0}$ also belongs to
  $L^{2}(\partial\Omega)$, the first part of this proof implies that
  $e^{-t \Lambda}h_{0}\to \overline{h_{0}}$ in $L^{2}(\partial\Omega)$
  as $t\to \infty$. If $q<2$, then by the continuous embedding of
  $L^{2}(\partial\Omega)$ into $L^{q}(\partial\Omega)$, one has
  that~\eqref{eq:39} needs to be true also in this case. Thus, let's focus
  now on the case $2<q\le \infty$. Since $h_{0}\in
  L^{\infty}(\partial\Omega)$, by the contractivity property of $e^{-t
    \Lambda}$ in $L^{\infty}(\partial\Omega)$, and by the fact that
  \begin{displaymath}
    e^{-t \Lambda}(c\mathds{1}_{\partial\Omega})=c\mathds{1}_{\partial\Omega}
    \qquad\text{$\mathcal{H}^{d-1}$-a.e. on $\partial\Omega$ for all
      $t\ge 0$,}
  \end{displaymath}
  one sees that
  \begin{align*}
    \norm{e^{-t \Lambda}h_{0}-\overline{h_{0}}}_{q}
    &\le \norm{e^{-t
      \Lambda}h_{0}-\overline{h_{0}}}_{\infty}^{\frac{q-2}{q}}\,
      \norm{e^{-t \Lambda}h_{0}-\overline{h_{0}}}_{2}^{\frac{2}{q}}\\
    &\le \norm{h_{0}-\overline{h_{0}}}_{\infty}^{\frac{q-2}{q}}\,
      \norm{e^{-t \Lambda}h_{0}-\overline{h_{0}}}_{2}^{\frac{2}{q}}\to
    0
  \end{align*}
  as $t\to \infty$. This completes the proof of
  statement~\eqref{propo:convergence-claim1} and by using the
  continuity of the functional $\varphi$, it follows that
  \eqref{propo:convergence-claim1} implies
  \eqref{propo:convergence-claim2}. To see that
  \eqref{propo:convergence-claim5} holds, we recall that by Theorem~\ref{thm:8},
  \begin{equation}
    \label{eq:48}
    \varphi(e^{-t \Lambda}h_{0})=
    \int_{\Omega}\abs{Du_{e^{-t \Lambda}h_{0}}}+
    \int_{\partial\Omega}\abs{e^{-t \Lambda}h_{0}-\T(u_{e^{-t \Lambda}h_{0}})}\,\dH^{d-1}
  \end{equation}
  for every $t>0$, where $Du_{e^{-t \Lambda}h_{0}}\in BV(\Omega)$ denotes a
  weak solution of Dirichlet problem~\eqref{eq:21N} with boundary data
  $e^{-t \Lambda}h_{0}$. 
  Further, by the Poincar\'e trace-inequality~\eqref{eq:61} for
  $BV$-functions, there is a constant $C_{p}>0$ such that
  \begin{equation}
    \label{eq:49}
    \norm{\T(u_{e^{-t \Lambda}h_{0}})-\overline{\T(u_{e^{-t
            \Lambda}h_{0}})}}_{1}\le C_{p}\,\int_{\Omega}\abs{Du_{e^{-t
          \Lambda}h_{0}}}.
  \end{equation}
  Now, by using \eqref{eq:48}, \eqref{eq:49}, and \eqref{eq:38}, then one finds that
  \begin{align*}
    \int_{\partial\Omega}\abs{e^{-t\Lambda}h_{0}-\overline{h_{0}}}\,\dH^{d-1}
     &\le \int_{\partial\Omega}\abs{e^{-t \Lambda}h_{0}-\T(u_{e^{-t\Lambda}h_{0}})}\,\dH^{d-1}\\
     &\qquad   +\int_{\partial\Omega}\abs{\T(u_{e^{-t\Lambda}h_{0}})
       -\overline{\T(u_{e^{-t \Lambda}h_{0}})}}\,\dH^{d-1}\\
     &\qquad\qquad
       +\int_{\partial\Omega}\abs{\overline{\T(u_{e^{-t\Lambda}h_{0}})}-\overline{h_{0}}}\,\dH^{d-1}\\
    &\le \int_{\partial\Omega}\abs{e^{-t \Lambda}h_{0}-\T(u_{e^{-t\Lambda}h_{0}})}\,\dH^{d-1}\\
     &\qquad   + C_{p}\,\int_{\Omega}\abs{Du_{e^{-t
          \Lambda}h_{0}}}\\
     &\qquad\qquad
       +\labs{\int_{\partial\Omega}\big(\T(u_{e^{-t\Lambda}h_{0}})-h_{0}\big)\dH^{d-1}}\\
    &=\int_{\partial\Omega}\abs{e^{-t \Lambda}h_{0}-\T(u_{e^{-t\Lambda}h_{0}})}\,\dH^{d-1}\\
     &\qquad   + C_{p}\,\int_{\Omega}\abs{Du_{e^{-t
          \Lambda}h_{0}}}\\
     &\qquad\qquad
       +\labs{\int_{\partial\Omega}\big(\T(u_{e^{-t\Lambda}h_{0}})-e^{-t\Lambda}h_{0}\big)\dH^{d-1}}\\
    &\le 2\,\int_{\partial\Omega}\abs{e^{-t \Lambda}h_{0}-\T(u_{e^{-t\Lambda}h_{0}})}\,\dH^{d-1}\\
     &\qquad   + C_{p}\,\int_{\Omega}\abs{Du_{e^{-t\Lambda}h_{0}}}\\
    &\le (2+C_{p})\,\varphi(e^{-t\Lambda}h_{0})
  \end{align*}
  for all $t\ge 0$, proving~\eqref{propo:convergence-claim5}. Finally,
  to see that \eqref{eq:50} holds, one simply applies~\eqref{eq:45}
  to
  \begin{displaymath}
    [\z_{e^{-t\Lambda}h_{0}},\nu]=-\dfrac{\td
      h}{\dt}_{\! +}(t)\qquad\text{$\mathcal{H}^{d-1}$-a.e. on $\partial\Omega$,}
  \end{displaymath}
  where the vector field
  $\z_{e^{-t\Lambda}h_{0}}\in L^{\infty}(\Omega;\R^{d})$ is such that
  $[\z_{e^{-t\Lambda}h_{0}},\nu]=\Lambda^{\!\circ}(e^{-t\Lambda}h_{0})$.
  Then one finds that
  \begin{align*}
    \varphi(e^{-t\Lambda}h_{0})&=\int_{\partial\Omega}[\z_{e^{-t\Lambda}h_{0}},\nu]e^{-t\Lambda}h_{0}\,\dH^{d-1}\\
    &\le
      \int_{\partial\Omega}\abs{[\z_{e^{-t\Lambda}h_{0}},\nu]}\,\abs{e^{-t\Lambda}h_{0}}\,\dH^{d-1}\\
    &\le 2\int_{\partial\Omega}\dfrac{\abs{e^{-t\Lambda}h_{0}}^2}{t}\,\dH^{d-1}
  \end{align*}
  for all $t>0$. This completes the proof of this proposition.

\end{proof}

%
%


\begin{thebibliography}{10}

\bibitem{MR1857292}
L.~Ambrosio, N.~Fusco, and D.~Pallara, \emph{Functions of bounded variation and
  free discontinuity problems}, Oxford Mathematical Monographs, The Clarendon
  Press, Oxford University Press, New York, 2000.

\bibitem{MR2294196}
K.~Ammar, F.~Andreu, and J.~Toledo, \emph{Quasi-linear elliptic problems in
  {$L^1$} with non homogeneous boundary conditions}, Rend. Mat. Appl. (7)
  \textbf{26} (2006), 291--314.

\bibitem{MR2033382}
F.~Andreu-Vaillo, V.~Caselles, and J.~M. Maz\'on, \emph{Parabolic quasilinear
  equations minimizing linear growth functionals}, Progress in Mathematics,
  vol. 223, Birkh\"auser Verlag, Basel, 2004.
  \href{http://dx.doi.org/10.1007/978-3-0348-7928-6}{\nolinkurl{doi:10.1007/978-3-0348-7928-6}}

\bibitem{MR750538}
G.~Anzellotti, \emph{Pairings between measures and bounded functions and
  compensated compactness}, Ann. Mat. Pura Appl. (4) \textbf{135} (1983),
  293--318 (1984).
  \href{http://dx.doi.org/10.1007/BF01781073}{\nolinkurl{doi:10.1007/BF01781073}}

\bibitem{MR2823661}
W.~Arendt and A.~F.~M. ter Elst, \emph{The {D}irichlet-to-{N}eumann operator on
  rough domains}, J. Differential Equations \textbf{251} (2011), 2100--2124.
  \href{http://dx.doi.org/10.1016/j.jde.2011.06.017}{\nolinkurl{doi:10.1016/j.jde.2011.06.017}}
  Available at
  \href{https://doi.org/10.1016/j.jde.2011.06.017}{\nolinkurl{https://doi.org/10.1016/j.jde.2011.06.017}}

\bibitem{MR2798103}
W.~Arendt, C.~J.~K. Batty, M.~Hieber, and F.~Neubrander, \emph{Vector-valued
  {L}aplace transforms and {C}auchy problems}, second ed., Monographs in
  Mathematics, vol.~96, Birkh\"auser/Springer Basel AG, Basel, 2011.
  \href{http://dx.doi.org/10.1007/978-3-0348-0087-7}{\nolinkurl{doi:10.1007/978-3-0348-0087-7}}

\bibitem{MR4041276}
W.~Arendt and D.~Hauer, \emph{Maximal {$L^2$}-regularity in nonlinear gradient
  systems and perturbations of sublinear growth}, Pure Appl. Anal. \textbf{2}
  (2020), 23--34.
  \href{http://dx.doi.org/10.2140/paa.2020.2.23}{\nolinkurl{doi:10.2140/paa.2020.2.23}}

\bibitem{MR2582280}
V.~Barbu, \emph{Nonlinear differential equations of monotone types in {B}anach
  spaces}, Springer Monographs in Mathematics, Springer, New York, 2010.
  \href{http://dx.doi.org/10.1007/978-1-4419-5542-5}{\nolinkurl{doi:10.1007/978-1-4419-5542-5}}

\bibitem{MR1164641}
P.~B\'enilan and M.~G. Crandall, \emph{Completely accretive operators},
  Semigroup theory and evolution equations ({D}elft, 1989), Lecture Notes in
  Pure and Appl. Math., vol. 135, Dekker, New York, 1991, pp.~41--75.

\bibitem{MR648452}
P.~B{\'e}nilan and M.~G. Crandall, \emph{Regularizing effects of homogeneous
  evolution equations}, Contributions to analysis and geometry ({B}altimore,
  {M}d., 1980), Johns Hopkins Univ. Press, Baltimore, Md., 1981, pp.~23--39.

\bibitem{Benilanbook}
P.~B{\'e}nilan, M.~G. Crandall, and A.~Pazy, \emph{Evolution problems governed
  by accretive operators}, book in preparation, 1994.

\bibitem{MR928802}
C.~Bennett and R.~Sharpley, \emph{Interpolation of operators}, Pure and Applied
  Mathematics, vol. 129, Academic Press, Inc., Boston, MA, 1988.

\bibitem{MR0250205}
E.~Bombieri, E.~De~Giorgi, and E.~Giusti, \emph{Minimal cones and the
  {B}ernstein problem}, Invent. Math. \textbf{7} (1969), 243--268.
  \href{http://dx.doi.org/10.1007/BF01404309}{\nolinkurl{doi:10.1007/BF01404309}}

\bibitem{MR3415587}
T.~Brander, \emph{Calder\'{o}n problem for the {$p$}-{L}aplacian: first order
  derivative of conductivity on the boundary}, Proc. Amer. Math. Soc.
  \textbf{144} (2016), 177--189.
  \href{http://dx.doi.org/10.1090/proc/12681}{\nolinkurl{doi:10.1090/proc/12681}}
  Available at
  \href{https://doi.org/10.1090/proc/12681}{\nolinkurl{https://doi.org/10.1090/proc/12681}}

\bibitem{MR0348562} H.~Br\'ezis, \emph{Op\'erateurs maximaux monotones
    et semi-groupes de contractions dans les espaces de {H}ilbert},
  North-Holland Mathematics Studies, No. 5. Notas de Matem\'atica
  (50), North-Holland Publishing Co., 1973.

\bibitem{MR377609}
R.~E. Bruck, Jr., \emph{Asymptotic convergence of nonlinear contraction
  semigroups in {H}ilbert space}, J. Functional Analysis \textbf{18} (1975),
  15--26.
  \href{http://dx.doi.org/10.1016/0022-1236(75)90027-0}{\nolinkurl{doi:10.1016/0022-1236(75)90027-0}}

\bibitem{MR590275}
A.-P. Calder\'{o}n, \emph{On an inverse boundary value problem}, Seminar on
  {N}umerical {A}nalysis and its {A}pplications to {C}ontinuum {P}hysics ({R}io
  de {J}aneiro, 1980), Soc. Brasil. Mat., Rio de Janeiro, 1980, pp.~65--73.

\bibitem{MR3465809}
R.~Chill, D.~Hauer, and J.~Kennedy, \emph{Nonlinear semigroups generated by
  {$j$}-elliptic functionals}, J. Math. Pures Appl. (9) \textbf{105} (2016),
  415--450.
  \href{http://dx.doi.org/10.1016/j.matpur.2015.11.005}{\nolinkurl{doi:10.1016/j.matpur.2015.11.005}}

\bibitem{CoulHau2017}
T.~Coulhon and D.~Hauer, \emph{Functional inequalities and regularizing effect
  of nonlinear semigroups - {T}heory and {A}pplication}, SMAI - Math\'ematiques
  et Applications, 2021.

\bibitem{MR1036731}
R.~Dautray and J.-L. Lions, \emph{Mathematical analysis and numerical methods
  for science and technology. {V}ol. 1}, Springer-Verlag, Berlin, 1990.

\bibitem{MR938995}
J.~I. D\'{\i}az and R.~F. Jim{\'e}nez, \emph{Boundary behaviour of solutions of
  the {S}ignorini problem. {I}. {T}he elliptic case}, Boll. Un. Mat. Ital. B
  (7) \textbf{2} (1988), 127--139.

\bibitem{MR3713816}
M.~Dos~Santos, \emph{Characteristic functions on the boundary of a planar
  domain need not be traces of least gradient functions}, Confluentes Math.
  \textbf{9} (2017), 65--93.
  \href{http://dx.doi.org/10.5802/cml.36}{\nolinkurl{doi:10.5802/cml.36}}

\bibitem{MR1158660}
L.~C. Evans and R.~F. Gariepy, \emph{Measure theory and fine properties of
  functions}, Studies in Advanced Mathematics, CRC Press, Boca Raton, FL, 1992.

\bibitem{2018arXiv181111138G}
W.~{G{\'o}rny}, \emph{{Existence of minimisers in the least gradient problem
  for general boundary data}}, arXiv e-prints :
  \href{https://arxiv.org/abs/1811.11138}{1811.11138} (2018), arXiv:1811.11138.

\bibitem{MR3852558}
W.~G\'{o}rny, \emph{({N}on)uniqueness of minimizers in the least gradient
  problem}, J. Math. Anal. Appl. \textbf{468} (2018), 913--938.
  \href{http://dx.doi.org/10.1016/j.jmaa.2018.08.038}{\nolinkurl{doi:10.1016/j.jmaa.2018.08.038}}

\bibitem{MR3813249}
\bysame, \emph{Planar least gradient problem: existence, regularity and
  anisotropic case}, Calc. Var. Partial Differential Equations \textbf{57}
  (2018), Art. 98, 27.
  \href{http://dx.doi.org/10.1007/s00526-018-1378-y}{\nolinkurl{doi:10.1007/s00526-018-1378-y}}

\bibitem{MR3596671}
W.~G\'{o}rny, P.~Rybka, and A.~Sabra, \emph{Special cases of the planar least
  gradient problem}, Nonlinear Anal. \textbf{151} (2017), 66--95.
  \href{http://dx.doi.org/10.1016/j.na.2016.11.020}{\nolinkurl{doi:10.1016/j.na.2016.11.020}}

\bibitem{MR3369257}
D.~Hauer, \emph{The {$p$}-{D}irichlet-to-{N}eumann operator with applications
  to elliptic and parabolic problems}, J. Differential Equations \textbf{259}
  (2015), 3615--3655.
  \href{http://dx.doi.org/10.1016/j.jde.2015.04.030}{\nolinkurl{doi:10.1016/j.jde.2015.04.030}}

\bibitem{MR4200826}
\bysame, \emph{Regularizing effect of homogeneous evolution equations with
  perturbation}, Nonlinear Anal. \textbf{206} (2021), 112245, 34.
  \href{http://dx.doi.org/10.1016/j.na.2021.112245}{\nolinkurl{doi:10.1016/j.na.2021.112245}}
  Available at
  \href{https://doi.org/10.1016/j.na.2021.112245}{\nolinkurl{https://doi.org/10.1016/j.na.2021.112245}}

\bibitem{hauer2020perturbation}
\bysame, \emph{Regularizing effect of homogeneous evolution equations with
  perturbation}, arXiv e-prints:
  \href{https://arxiv.org/abs/2004.00483}{2004.00483} (2021), 1--39.

\bibitem{MR4031770}
D.~Hauer and J.~M. Maz\'{o}n, \emph{Regularizing effects of homogeneous
  evolution equations: the case of homogeneity order zero}, J. Evol. Equ.
  \textbf{19} (2019), 965--996.
  \href{http://dx.doi.org/10.1007/s00028-019-00502-y}{\nolinkurl{doi:10.1007/s00028-019-00502-y}}

\bibitem{MR3739314}
R.~L. Jerrard, A.~Moradifam, and A.~I. Nachman, \emph{Existence and uniqueness
  of minimizers of general least gradient problems}, J. Reine Angew. Math.
  \textbf{734} (2018), 71--97.
  \href{http://dx.doi.org/10.1515/crelle-2014-0151}{\nolinkurl{doi:10.1515/crelle-2014-0151}}

\bibitem{MR3798643}
M.~Latorre and S.~Segura~de Le\'{o}n, \emph{Elliptic 1-{L}aplacian equations
  with dynamical boundary conditions}, J. Math. Anal. Appl. \textbf{464}
  (2018), 1051--1081.
  \href{http://dx.doi.org/10.1016/j.jmaa.2018.02.006}{\nolinkurl{doi:10.1016/j.jmaa.2018.02.006}}
  Available at
  \href{https://doi.org/10.1016/j.jmaa.2018.02.006}{\nolinkurl{https://doi.org/10.1016/j.jmaa.2018.02.006}}

\bibitem{MR797538}
H.~P. Lotz, \emph{Uniform convergence of operators on {$L^\infty$} and similar
  spaces}, Math. Z. \textbf{190} (1985), 207--220.
  \href{http://dx.doi.org/10.1007/BF01160459}{\nolinkurl{doi:10.1007/BF01160459}}
  Available at
  \href{https://doi.org/10.1007/BF01160459}{\nolinkurl{https://doi.org/10.1007/BF01160459}}

\bibitem{MR3263922}
J.~M. Maz\'on, J.~D. Rossi, and S.~Segura~de Le\'on, \emph{Functions of least
  gradient and 1-harmonic functions}, Indiana Univ. Math. J. \textbf{63}
  (2014), 1067--1084.
  \href{http://dx.doi.org/10.1512/iumj.2014.63.5327}{\nolinkurl{doi:10.1512/iumj.2014.63.5327}}
  Available at
  \href{https://doi.org/10.1512/iumj.2014.63.5327}{\nolinkurl{https://doi.org/10.1512/iumj.2014.63.5327}}

\bibitem{MR3328128}
J.~M. Maz\'{o}n, J.~D. Rossi, and S.~Segura~de Le\'{o}n, \emph{The
  1-{L}aplacian elliptic equation with inhomogeneous {R}obin boundary
  conditions}, Differential Integral Equations \textbf{28} (2015), 409--430.
  Available at
  \href{http://projecteuclid.org/euclid.die/1427744095}{\nolinkurl{http://projecteuclid.org/euclid.die/1427744095}}

\bibitem{MR2777530}
V.~Maz'ya, \emph{Sobolev spaces with applications to elliptic partial
  differential equations}, augmented ed., Grundlehren der Mathematischen
  Wissenschaften [Fundamental Principles of Mathematical Sciences], vol. 342,
  Springer, Heidelberg, 2011.
  \href{http://dx.doi.org/10.1007/978-3-642-15564-2}{\nolinkurl{doi:10.1007/978-3-642-15564-2}}
  Available at
  \href{https://doi.org/10.1007/978-3-642-15564-2}{\nolinkurl{https://doi.org/10.1007/978-3-642-15564-2}}

\bibitem{MR0222735}
M.~Miranda, \emph{Comportamento delle successioni convergenti di frontiere
  minimali}, Rend. Sem. Mat. Univ. Padova \textbf{38} (1967), 238--257.

\bibitem{MR921549}
L.~Modica, \emph{Gradient theory of phase transitions with boundary contact
  energy}, Ann. Inst. H. Poincar\'e Anal. Non Lin\'eaire \textbf{4} (1987),
  487--512. Available at
  \href{http://www.numdam.org/item?id=AIHPC_1987__4_5_487_0}{\nolinkurl{http://www.numdam.org/item?id=AIHPC_1987__4_5_487_0}}

\bibitem{MR3820240}
A.~Moradifam, \emph{Existence and structure of minimizers of least gradient
  problems}, Indiana Univ. Math. J. \textbf{67} (2018), 1025--1037.
  \href{http://dx.doi.org/10.1512/iumj.2018.67.7360}{\nolinkurl{doi:10.1512/iumj.2018.67.7360}}

\bibitem{MR0458304}
H.~Parks, \emph{Explicit determination of area minimizing hypersurfacess}, Duke
  Math. J. \textbf{44} (1977), 519--534. Available at
  \href{http://projecteuclid.org/euclid.dmj/1077312385}{\nolinkurl{http://projecteuclid.org/euclid.dmj/1077312385}}

\bibitem{MR831890}
H.~R. Parks, \emph{Explicit determination of area minimizing hypersurfaces.
  {II}}, Mem. Amer. Math. Soc. \textbf{60} (1986), iv+90.
  \href{http://dx.doi.org/10.1090/memo/0342}{\nolinkurl{doi:10.1090/memo/0342}}

\bibitem{MR808421}
H.~R. Parks and W.~P. Ziemer, \emph{Jacobi fields and regularity of functions
  of least gradient}, Ann. Scuola Norm. Sup. Pisa Cl. Sci. (4) \textbf{11}
  (1984), 505--527. Available at
  \href{http://www.numdam.org/item?id=ASNSP_1984_4_11_4_505_0}{\nolinkurl{http://www.numdam.org/item?id=ASNSP_1984_4_11_4_505_0}}

\bibitem{MR193549}
R.~T. Rockafellar, \emph{Characterization of the subdifferentials of convex
  functions}, Pacific J. Math. \textbf{17} (1966), 497--510.

\bibitem{MR3683462}
L.~Rondi, \emph{A {F}riedrichs-{M}az'ya inequality for functions of bounded
  variation}, Math. Nachr. \textbf{290} (2017), 1830--1839.
  \href{http://dx.doi.org/10.1002/mana.201600004}{\nolinkurl{doi:10.1002/mana.201600004}}

\bibitem{MR924157}
W.~Rudin, \emph{Real and complex analysis}, third ed., McGraw-Hill Book Co.,
  New York, 1987.

\bibitem{MR3023384}
M.~Salo and X.~Zhong, \emph{An inverse problem for the {$p$}-{L}aplacian:
  boundary determination}, SIAM J. Math. Anal. \textbf{44} (2012), 2474--2495.
  \href{http://dx.doi.org/10.1137/110838224}{\nolinkurl{doi:10.1137/110838224}}
  Available at
  \href{https://doi.org/10.1137/110838224}{\nolinkurl{https://doi.org/10.1137/110838224}}

\bibitem{MR1422252}
R.~E. Showalter, \emph{Monotone operators in {B}anach space and nonlinear
  partial differential equations}, Mathematical Surveys and Monographs,
  vol.~49, American Mathematical Society, Providence, RI, 1997.
  \href{http://dx.doi.org/10.1090/surv/049}{\nolinkurl{doi:10.1090/surv/049}}

\bibitem{MR3298723}
G.~S. Spradlin and A.~Tamasan, \emph{Not all traces on the circle come from
  functions of least gradient in the disk}, Indiana Univ. Math. J. \textbf{63}
  (2014), 1819--1837.
  \href{http://dx.doi.org/10.1512/iumj.2014.63.5421}{\nolinkurl{doi:10.1512/iumj.2014.63.5421}}

\bibitem{MR1172906}
P.~Sternberg, G.~Williams, and W.~P. Ziemer, \emph{Existence, uniqueness, and
  regularity for functions of least gradient}, J. Reine Angew. Math.
  \textbf{430} (1992), 35--60.

\bibitem{MR1246349}
P.~Sternberg and W.~P. Ziemer, \emph{The {D}irichlet problem for functions of
  least gradient}, Degenerate diffusions ({M}inneapolis, {MN}, 1991), IMA Vol.
  Math. Appl., vol.~47, Springer, New York, 1993, pp.~197--214.
  \href{http://dx.doi.org/10.1007/978-1-4612-0885-3_14}{\nolinkurl{doi:10.1007/978-1-4612-0885-3_14}}
  Available at
  \href{https://doi.org/10.1007/978-1-4612-0885-3_14}{\nolinkurl{https://doi.org/10.1007/978-1-4612-0885-3_14}}

\bibitem{MR1014685}
W.~P. Ziemer, \emph{Weakly differentiable functions}, Graduate Texts in
  Mathematics, vol. 120, Springer-Verlag, New York, 1989, Sobolev spaces and
  functions of bounded variation.
  \href{http://dx.doi.org/10.1007/978-1-4612-1015-3}{\nolinkurl{doi:10.1007/978-1-4612-1015-3}}
  Available at
  \href{https://doi.org/10.1007/978-1-4612-1015-3}{\nolinkurl{https://doi.org/10.1007/978-1-4612-1015-3}}

\end{thebibliography}

\providecommand{\bysame}{\leavevmode\hbox to3em{\hrulefill}\thinspace}

\end{document}